\documentclass[twoside]{amsart}

\usepackage{amsfonts,amsmath,amssymb,amsthm}
\usepackage[bookmarksnumbered,plainpages,hypertex]{hyperref}
\usepackage{srcltx}
\usepackage{graphicx}
\usepackage[all]{xy}

\newtheorem{theorem}{\sc Theorem}[section]
\newtheorem{proposition}[theorem]{\sc Proposition}
\newtheorem{notation}[theorem]{\sc Notation}

\newtheorem{lemma}[theorem]{\sc Lemma}
\newtheorem{corollary}[theorem]{\sc Corollary}
\theoremstyle{definition}
\newtheorem{definition}[theorem]{\sc Definition}

\newtheorem{example}[theorem]{\sc Example}

\theoremstyle{remark}
\newtheorem{remark}[theorem]{\sc Remark}

\newtheorem{claim}[theorem]{}

\allowdisplaybreaks

\newcommand{\diagMonFun}{\xymatrix@R=35pt@C=63pt{(F(U)\otimes'F(V))\otimes'F(W)\ar[d]|{a'_{F(U),F(V),F(W)}}\ar[r]^{\phi_{2}(U,V)\otimes'F(W)} & F(U\otimes
V)\otimes'F(W)\ar[r]^{\phi_{2}(U\otimes V,W)} & F((U\otimes V)\otimes
W)\ar[d]|{F(a_{U,V,W})}\\
F(U)\otimes'(F(V)\otimes'F(W))\ar[r]^{F(U)\otimes'\phi_{2}(V,W)} &
F(U)\otimes'F(V\otimes W)\ar[r]^{\phi_{2}(U,V\otimes W)} &
F(U\otimes(V\otimes W))}}
\newcommand{\balxi}{\xymatrix@C=2cm{\left( F\left( U\right) \otimes H\right) \otimes M \ar@<.5ex>[rr]^{\mu _{F\left(
U\right) }^{r}\otimes M} \ar@<-.5ex>[rr]_{\left( F\left( U\right) \otimes \mu
_{M}^{l}\right) \circ {^Ha^H_{F\left( U\right) ,H,M}}}&&F\left( U\right) \otimes M}}

\begin{document}
\title{Bosonization for dual quasi-bialgebras and preantipode}
\author{Alessandro Ardizzoni}
\address{University of Padova, Department of Pure and Applied Mathematics,
Via Trieste 63, Padova, I-35121, \indent Italy}
\email{ardizzoni@math.unipd.it}
\urladdr{http://www.unife.it/utenti/alessandro.ardizzoni}
\author{Alice Pavarin}
\address{University of Padova, Department of Pure and Applied Mathematics,
Via Trieste 63, Padova, I-35121, \indent Italy}
\email{apavarin@math.unipd.it}
\thanks{This paper was written while the authors were members of GNSAGA}

\begin{abstract}
In this paper, we associate a dual quasi-bialgebra, called bosonization, to every dual quasi-bialgebra $H$ and every bialgebra $R$ in the category of Yetter-Drinfeld modules over $H$. Then, using the fundamental theorem, we characterize  as bosonizations the dual quasi-bialgebras with a projection onto a dual quasi-bialgebra with a
preantipode. As an application we investigate the structure of the graded coalgebra $\mathrm{gr}A$ associated to a dual quasi-bialgebra $A$ with the dual Chevalley property (e.g. $A$ is pointed).
\end{abstract}

\maketitle
\tableofcontents

\pagestyle{headings}

\section{Introduction}

Let $H$ be a bialgebra. Consider the functor $T:=(-)\otimes H:\mathfrak{M}%
\rightarrow \mathfrak{M}_{H}^{H}$ from the category of vector spaces to the
category of right Hopf modules. It is well-known that $T$ determines an
equivalence if and only if $H$ has an antipode i.e. it is a Hopf algebra.
The fact that $T$ is an equivalence is the so-called fundamental (or
structure) theorem for Hopf modules, which is due, in the finite-dimensional
case, to Larson and Sweedler, see \cite[Proposition 1, page 82]{LS}. This
result is crucial in characterizing the structure of bialgebras with a
projection as Radford-Majid bosonizations (see \cite{Radford}). Recall that a
bialgebra $A$ has a projection onto a Hopf algebra $H$ if there exist
bialgebra maps $\sigma :H\rightarrow A$ and $\pi :A\rightarrow H$ such that $%
\pi \circ \sigma =\mathrm{Id}_{H}$. Essentially using the fundamental
theorem, one proves that $A$ is isomorphic, as a vector
space, to the tensor product $R\otimes H$ where $R$ is some bialgebra in the
category ${_{H}^{H}\mathcal{YD}}$ of Yetter-Drinfeld modules over $H$. This
way $R\otimes H$ inherits, from $A$, a bialgebra structure which is called
the \emph{Radford-Majid bosonization of }$\emph{R}$\emph{\ by }$\emph{H}$
and denoted by $R\#H$. It is remarkable that the graded coalgebra $\mathrm{gr%
}A$ associated to a pointed Hopf algebra $A$ (here "pointed" means that all
simple subcoalgebras of $A$ are one-dimensional) always admits a projection
onto its coradical. This is the main ingredient in the so-called lifting
method for the classification of finite dimensional pointed Hopf algebras,
see \cite{AS-Lifting}.

In 1989 Drinfeld introduced the concept of quasi-bialgebra in connection
with the Knizhnik-Zamolodchikov system of partial differential equations.
The axioms defining a quasi-bialgebra are a translation of monoidality of
its representation category with respect to the diagonal tensor product. In 
\cite{Drinfeld}, the antipode for a quasi-bialgebra (whence the concept of
quasi-Hopf algebra) is introduced in order to make the category of its flat
right modules rigid. If we draw our attention to the category of
co-representations of $H$, we get the concepts of dual quasi-bialgebra and
of dual quasi-Hopf algebra. These notions have been introduced in \cite{Majid-Tannaka} in order to prove a Tannaka-Krein type Theorem for quasi-Hopf algebras.

A fundamental theorem for dual quasi-Hopf algebras was proved by Schauenburg
in \cite{Sch-TwoChar} but dual quasi-Hopf algebras do not exhaust the class
of dual quasi-bialgebras satisfying the fundamental theorem. It is
remarkable that the functor $T$ giving the fundamental theorem in the case
of ordinary Hopf algebras must be substituted, in the \textquotedblleft
quasi\textquotedblleft\ case, by the functor $F:=(-)\otimes H$ between the
category $^{H}\mathfrak{M}$ of left $H$-comodules and the category $^{H}%
\mathfrak{M}_{H}^{H}$ of right dual quasi-Hopf $H$-bicomodules (essentially
this is due to the fact that, unlike the classical case, a dual
quasi-bialgebra $H$ is not an algebra in the category of right $H$-comodules
but it is still an algebra in the category of $H$-bicomodules). In \cite[%
Theorem 3.9]{Ardi-Pava}, we showed that, for a dual quasi-bialgebra $H,$ the
functor $F$ is an equivalence if and only if there exists a suitable map $%
S:H\rightarrow H$ that we called a preantipode for $H$. Moreover for any
dual quasi-bialgebra with antipode (i.e. a dual quasi-Hopf algebra) we
constructed a specific preantipode, see \cite[Theorem 3.10]{Ardi-Pava}.\medskip

The main aim of this paper is to introduce and investigate the notion of
bosonization in the setting of dual quasi-bialgebras. Explicitly, we
associate a dual quasi-bialgebra $R\#H$ (that we call bosonization of $R$\
by $H$) to every dual quasi-bialgebra $H$ and bialgebra $R$ in $_{H}^{H}%
\mathcal{YD}$. Then, using the fundamental theorem, we characterize as bosonizations the dual
quasi-bialgebras with a projection onto a dual quasi-bialgebra with a
preantipode. As an application, for any dual quasi-bialgebra $A$ with the dual Chevalley property (i.e. such that the
coradical of $A$ is a dual quasi-subbialgebra of $A$), under the further hypothesis that the coradical $H$ of $A$ has a preantipode, we prove that there
is a bialgebra $R$ in ${_{H}^{H}\mathcal{YD}}$ such that $\mathrm{gr}A$ is
isomorphic to $R\#H$ as a dual quasi-bialgebra. In particular, if $A$ is a
pointed dual quasi-Hopf algebra, then $\mathrm{gr}A$ comes out to be
isomorphic to $R\#\Bbbk \mathbb{G}\left( A\right) $ as dual quasi-bialgebra
where $R$ is the diagram of $A$ and $\mathbb{G}\left( A\right) $ is the set
of grouplike elements in $A$. We point out that the results in this paper are obtained without assuming that the dual quasi-bialgebra considered are finite-dimensional.

\medskip

The paper is organized as follows.

Section \ref{C1} contains preliminary results needed in the next
sections. Moreover in Theorem \ref{teo:cocom}, we investigate cocommutative
dual quasi-bialgebras with a preantipode and in Corollary \ref{coro:cartan},
we provide a Cartier-Gabriel-Kostant type theorem for dual quasi-bialgebras
with a preantipode. In the connected case such a result was achieved in \cite%
[Theorem 4.3]{Huang-QuiverAppr}.

Section \ref{C2}, is devoted to the study of the category ${_{H}^{H}\mathcal{%
YD}}$ of Yetter-Drinfeld modules over a dual quasi-bialgebra $H.$
Explicitly, we consider the pre-braided monoidal category $\left( {_{H}^{H}%
\mathcal{YD}},\otimes ,\Bbbk \right) $ of Yetter-Drinfeld modules over a
dual quasi-bialgebra $H$ and we prove that the functor $F$, as above, induces a functor $F:{_{H}^{H}%
\mathcal{YD}}\rightarrow {_{H}^{H}\mathfrak{M}_{H}^{H}}$ (that is an
equivalence in case $H$ has a preantipode, see Proposition \ref{restr equi}).

In Section \ref{C3}, we prove that the equivalence between the categories ${%
_{H}^{H}\mathfrak{M}_{H}^{H}}$ and ${_{H}^{H}\mathcal{YD}}$ becomes monoidal
if we equip ${_{H}^{H}\mathfrak{M}_{H}^{H}}$ with the tensor product $\otimes _{H}$
(or ${{\square _{H}}}$) and unit $H$ (see Lemma \ref{lem: equiv-m4h-tens}
and Lemma \ref{lem: F monoidal cot}). As a by-product, in Lemma \ref{equi
tensH, cot}, we produce a monoidal equivalence between $({{{_{H}^{H}%
\mathfrak{M}_{H}^{H}}}},\otimes _{H},H)$ and $({_{H}^{H}\mathfrak{M}%
_{H}^{H},\square }_{H},H).$

Section \ref{C4} contains the main results of the paper. In Theorem \ref%
{teo: boso}, to every dual quasi-bialgebra $H$ and bialgebra $R$ in $_{H}^{H}%
\mathcal{YD}$ we associate a dual quasi-bialgebra structure on the tensor
product $R\otimes H$ that we call the \emph{bosonization of }$R$\emph{\ by }$%
H$ and denote by $R\#H.$ Now, let $(A,H,\sigma ,\pi )$ be a dual
quasi-bialgebra with projection and assume that $H$ has a preantipode $S$.
In Lemma \ref{A in M tre H}, we prove that such an $A$ is an object in the
category ${{{_{H}^{H}\mathfrak{M}_{H}^{H}}}}$. Therefore the fundamental
theorem describes $A$ as the tensor product $R\otimes H$ of some vector
space $R$ by $H$. Indeed, in Theorem \ref{teo: projection}, we prove that
the dual quasi-bialgebra structure inherited by $R\otimes H$ through the
claimed isomorphism is exactly the bosonization of $R$ by $H$. The analogous
of this result for quasi-Hopf algebras, anything but trivial, has been
established by Bulacu and Nauwelaerts in \cite{bulacu2}, but their proof can
not be adapted to dual quasi-bialgebras with a preantipode.

In Section \ref{C5} we collect some applications of our results. Let $A$ be
a dual quasi-bialgebra with the dual Chevalley property and coradical $H$. Since $A$ is an
ordinary coalgebra, we can consider the associated graded coalgebra $\mathrm{%
gr}A$. In Proposition \ref{pro: grA}, we prove that $\mathrm{gr}A$ fits into
a dual quasi-bialgebra with projection onto $H$. As a consequence, in
Corollary \ref{coro: grA}, under the further assumption that $H$ has a
preantipode, we show that there is a bialgebra $R$ in ${_{H}^{H}\mathcal{YD}}
$ such that $\mathrm{gr}A$ is isomorphic to $R\#H$ as a dual quasi-bialgebra.
When $A$ is a pointed dual quasi-Hopf algebra it is in particular a dual
quasi-bialgebra with the dual Chevalley property and its coradical has a
preantipode. Using this fact, in Theorem \ref{teo:grpointed} we obtain that $%
\mathrm{gr}A$ is of the form $R\#\Bbbk \mathbb{G}\left( A\right) $ as dual
quasi-bialgebra, where $R$ is the so-called diagram of $A$.

\section{Preliminaries\label{C1}}

In this section we recall the definitions and results that will be needed in
the paper.

\begin{notation}
Throughout this paper $\Bbbk $ will denote a field. All vector spaces will
be defined over $\Bbbk $. The unadorned tensor product $\otimes $ will
denote the tensor product over $\Bbbk $ if not stated otherwise.
\end{notation}

\begin{claim}
\textbf{Monoidal Categories.} Recall that (see \cite[Chap. XI]{Ka}) a \emph{%
monoidal category}\textbf{\ }is a category $\mathcal{M}$ endowed with an
object $\mathbf{1}\in \mathcal{M}$ (called \emph{unit}), a
functor $\otimes :\mathcal{M}\times \mathcal{M}\rightarrow \mathcal{M}$
(called \emph{tensor product}), and functorial isomorphisms $%
a_{X,Y,Z}:(X\otimes Y)\otimes Z\rightarrow $ $X\otimes (Y\otimes Z)$, $l_{X}:%
\mathbf{1}\otimes X\rightarrow X,$ $r_{X}:X\otimes \mathbf{1}\rightarrow X,$
for every $X,Y,Z$ in $\mathcal{M}$. The functorial morphism $a$ is called
the \emph{associativity constraint }and\emph{\ }satisfies the \emph{Pentagon
Axiom, }that is the equality 
\begin{equation*}
(U\otimes a_{V,W,X})\circ a_{U,V\otimes W,X}\circ (a_{U,V,W}\otimes
X)=a_{U,V,W\otimes X}\circ a_{U\otimes V,W,X}
\end{equation*}%
holds true, for every $U,V,W,X$ in $\mathcal{M}.$ The morphisms $l$ and $r$
are called the \emph{unit constraints} and they obey the \emph{Triangle
Axiom, }that is $(V\otimes l_{W})\circ a_{V,\mathbf{1},W}=r_{V}\otimes W$,
for every $V,W$ in $\mathcal{M}$.
\end{claim}

The notions of algebra, module over an algebra, coalgebra and comodule over
a coalgebra can be introduced in the general setting of monoidal categories.
Given an algebra $A$ in $\mathcal{M}$ one can define the categories $_{A}%
\mathcal{M}$, $\mathcal{M}_{A}$ and $_{A}\mathcal{M}_{A}$ of left, right and
two-sided modules over $A$ respectively. Similarly, given a coalgebra $C$ in 
$\mathcal{M}$, one can define the categories of $C$-comodules $^{C}\mathcal{M%
},\mathcal{M}^{C},{^{C}\mathcal{M}^{C}}$. For more details, the reader is
refereed to \cite{AMS-Hoch}.

Let $\mathcal{M}$ be a monoidal category. Assume that $\mathcal{M}$ is
abelian and both the functors $X\otimes (-):\mathcal{M}\rightarrow \mathcal{M%
}$ and $(-)\otimes X:\mathcal{M}\rightarrow \mathcal{M}$ are additive and
right exact, for any $X\in \mathcal{M}.$ Given an algebra $A$ in $\mathcal{M}
$, there exist a suitable functor $\otimes _{A}:{_{A}\mathcal{M}_{A}}\times {%
_{A}\mathcal{M}_{A}}\rightarrow {_{A}\mathcal{M}}_{A}$ and constraints that
make the category $({_{A}\mathcal{M}}_{A},\otimes _{A},A)$ monoidal, see 
\cite[1.11]{AMS-Hoch}. The tensor product over $A$ in $\mathcal{M}$ of a
right $A$-module $(V,\mu^r _{V})$ and a left $A$-module $(W,\mu^l _{W})$ is
defined to be the coequalizer: 
\begin{equation*}
\xymatrix{ (V\otimes A)\otimes W \ar[rr]^{\mu^r_V\otimes W}
\ar[rd]_{a_{V,A,W}}&& V\otimes W \ar[rr]^{_{A}\chi _{V,W}} && V\otimes _{A}W
\ar[r] & 0\\ &V\otimes(A\otimes W)\ar[ru]_{V\otimes \mu^l_W} }
\end{equation*}%
Note that, since $\otimes $ preserves coequalizers, then $V\otimes _{A}W$ is
also an $A$-bimodule, whenever $V$ and $W$ are $A$-bimodules.\medskip 
\newline
Dually, given a coalgebra $(C,\Delta ,\varepsilon )$ in a monoidal category $%
\mathcal{M}$, abelian and with additive and left exact tensor functors,
there exist a suitable functor $\square _{C}:{^{C}\mathcal{M}^{C}}\times {%
^{C}\mathcal{M}^{C}}\rightarrow {^{C}\mathcal{M}^{C}}$ and constraints that
make the category $({^{C}\mathcal{M}^{C}},\square _{C},C)$ monoidal. The cotensor product over $C$ in $\mathcal{M}$ of a right $C$%
-comodule $(V,\rho^r _{V})$ and a left $C$-comodule $(W,\rho^l _{W})$ is
defined to be the equalizer: 
\begin{equation*}
\xymatrix{ 0 \ar[r] & V\square_{C}W \ar[rr]^{_C\varsigma_{V,W}} && V\otimes
W \ar[rr]^{V\otimes \rho^l_W} \ar[rd]_{\rho^r_V\otimes W}
&&V\otimes(C\otimes W)\\ &&&&(V\otimes C)\otimes W\ar[ru]_{a_{V,C,W}} }
\end{equation*}%
Note that, since $\otimes $ preserves equalizers, then $V\square _{C}W$ is
also a $C$-bicomodule, whenever $V$ and $W$ are $C$-bicomodules.

\begin{definition}
A dual quasi-bialgebra is a datum $(H,m,u,\Delta ,\varepsilon ,\omega )$
where

\begin{itemize}
\item $(H,\Delta ,\varepsilon )$ is a coassociative coalgebra;

\item $m:H\otimes H\rightarrow H$ and $u:\Bbbk \rightarrow H$ are coalgebra
maps called multiplication and unit respectively; we set $1_{H}:=u(1_{\Bbbk
})$;

\item $\omega :H\otimes H\otimes H\rightarrow \Bbbk $ is a unital $3$%
-cocycle i.e. it is convolution invertible and satisfies%
\begin{eqnarray}
\omega \left( H\otimes H\otimes m\right) \ast \omega \left( m\otimes
H\otimes H\right) &=&m_{\Bbbk }\left( \varepsilon \otimes \omega \right)
\ast \omega \left( H\otimes m\otimes H\right) \ast m_{\Bbbk }\left( \omega
\otimes \varepsilon \right)  \label{eq:3-cocycle} \\
\text{and}\quad \omega \left( h\otimes k\otimes l\right) &=&\varepsilon
\left( h\right) \varepsilon \left( k\right) \varepsilon \left( l\right)
\qquad \text{whenever}\qquad 1_{H}\in \{h,k,l\};
\label{eq:qusi-unitairity cocycle}
\end{eqnarray}

\item $m$ is quasi-associative and unitary i.e. it satisfies%
\begin{gather}
m\left( H\otimes m\right) \ast \omega =\omega \ast m\left( m\otimes H\right)
,  \label{eq:quasi-associativity} \\
m\left( 1_{H}\otimes h\right) =h,\text{ for all }h\in H,
\label{eq:quasi-leftunitarirty} \\
m\left( h\otimes 1_{H}\right) =h,\text{ for all }h\in H.
\label{eq:quasi-rightunitarity}
\end{gather}%
$\omega $ is called \textit{the reassociator} of the dual quasi-bialgebra.
\end{itemize}

A morphism of dual quasi-bialgebras (see e.g. \cite[Section 2]{Schauenburg}) 
\begin{equation*}
\alpha :\left( H,m,u,\Delta ,\varepsilon ,\omega \right) \rightarrow \left(
H^{\prime },m^{\prime },u^{\prime },\Delta ^{\prime },\varepsilon ^{\prime
},\omega ^{\prime }\right)
\end{equation*}
is a coalgebra homomorphism $\alpha :\left( H,\Delta ,\varepsilon \right)
\rightarrow \left( H^{\prime },\Delta ^{\prime },\varepsilon ^{\prime
}\right) $ such that 
\begin{equation*}
m^{\prime }(\alpha \otimes \alpha )=\alpha m,\qquad \alpha u=u^{\prime
},\qquad \omega ^{\prime }\left( \alpha \otimes \alpha \otimes \alpha
\right) =\omega .
\end{equation*}%
It is an isomorphism of quasi-bialgebras if, in addition, it is invertible.

A dual quasi-subbialgebra of a dual quasi-bialgebra $\left( H^{\prime
},m^{\prime },u^{\prime },\Delta ^{\prime },\varepsilon ^{\prime },\omega
^{\prime }\right) $ is a quasi-bialgebra $\left( H,m,u,\Delta ,\varepsilon
,\omega \right) $ such that $H$ is a vector subspace of $H^{\prime }$ and
the canonical inclusion $\alpha :H\rightarrow H^{\prime }$ yields a morphism
of dual quasi-bialgebras.
\end{definition}

\subsection{The category of bicomodules for a dual quasi-bialgebras}

Let $(H,m,u,\Delta ,\varepsilon ,\omega )$ be a dual quasi-bialgebra. It is
well-known that the category $\mathfrak{M}^{H}$ of right $H$-comodules
becomes a monoidal category as follows. Given a right $H$-comodule $V$, we
denote by $\rho =\rho _{V}^{r}:V\rightarrow V\otimes H,\rho (v)=v_{0}\otimes
v_{1}$, its right $H$-coaction. The tensor product of two right $H$%
-comodules $V$ and $W$ is a comodule via diagonal coaction i.e. $\rho \left(
v\otimes w\right) =v_{0}\otimes w_{0}\otimes v_{1}w_{1}.$ The unit is $\Bbbk
,$ which is regarded as a right $H$-comodule via the trivial coaction i.e. $%
\rho \left( k\right) =k\otimes 1_{H}$. The associativity and unit
constraints are defined, for all $U,V,W\in \mathfrak{M}^{H}$ and $u\in
U,v\in V,w\in W,k\in \Bbbk ,$ by%
\begin{gather*}
a_{U,V,W}^{H}((u\otimes v)\otimes w):=u_{0}\otimes (v_{0}\otimes w_{0})\omega
(u_{1}\otimes v_{1}\otimes w_{1}), \\
l_{U}(k\otimes u):=ku\qquad \text{and}\qquad r_{U}(u\otimes k):=uk.
\end{gather*}%
The monoidal category we have just described will be denoted by $(\mathfrak{M%
}^{H},\otimes ,\Bbbk ,a^{H},l,r).$

Similarly, the monoidal categories $({^{H}\mathfrak{M}},\otimes ,\Bbbk ,{^{H}%
}a,l,r)$ and $({^{H}\mathfrak{M}^{H}},\otimes ,\Bbbk ,{^{H}}a{^{H}},l,r)$
are introduced. We just point out that 
\begin{gather*}
{^{H}}a_{U,V,W}((u\otimes v)\otimes w):=\omega ^{-1}(u_{-1}\otimes
v_{-1}\otimes w_{-1})u_{0}\otimes (v_{0}\otimes w_{0}), \\
{^{H}}a{^{H}}_{U,V,W}((u\otimes v)\otimes w):=\omega ^{-1}(u_{-1}\otimes
v_{-1}\otimes w_{-1})u_{0}\otimes (v_{0}\otimes w_{0})\omega (u_{1}\otimes
v_{1}\otimes w_{1}).
\end{gather*}

\begin{remark}
\label{rem: H alg}We know that, if $(H,m,u,\Delta ,\varepsilon ,\omega )$ is
a dual quasi-bialgebra, we cannot construct the category $\mathfrak{M}_{H}$,
because $H$ is not an algebra. Moreover $H$ is not an algebra in $\mathfrak{M%
}^{H}$ or in $^{H}\mathfrak{M}.$ On the other hand $((H,\rho _{H}^{l},\rho
_{H}^{r}),m,u)$ is an algebra in the monoidal category $({^{H}\mathfrak{M}%
^{H}},\otimes ,\Bbbk ,{^{H}}a{^{H}},l,r)$ with $\rho _{H}^{l}=\rho
_{H}^{r}=\Delta $. Thus, the only way to construct the category ${^{H}%
\mathfrak{M}_{H}^{H}}$ is to consider the right $H$-modules in $^{H}%
\mathfrak{M}^{H}$. Hence, we can set 
\begin{equation*}
{^{H}\mathfrak{M}_{H}^{H}}:=(^{H}\mathfrak{M}^{H})_{H}.
\end{equation*}%
The category ${^{H}\mathfrak{M}_{H}^{H}}$ is the so-called category of \emph{%
right dual quasi-Hopf $H$-bicomodules} \cite[Remark 2.3]{bulacu}.
\end{remark}

\begin{remark}
\label{FuntT}\cite[Example 1.5(a)]{AMS-Hoch} Let $(A,m,u)$ be an algebra in
a given monoidal category $(\mathcal{M},\otimes ,1,a,l,r)$. Then the
assignments $M\longmapsto (M\otimes A,(M\otimes m)\circ a_{A,A,A})$ and $%
f\longmapsto f\otimes A$ define a functor $T:\mathcal{M}\rightarrow \mathcal{%
M}_{A}.$ Moreover the forgetful functor $U:\mathcal{M}_{A}\rightarrow 
\mathcal{M}$ is a right adjoint of $T$.
\end{remark}

\subsection{An adjunction between ${^{H}\mathfrak{M}_{H}^{H}}$ and ${^{H}%
\mathfrak{M}}$}

We are going to construct an adjunction between ${^{H}\mathfrak{M}_{H}^{H}}$
and ${{^{H}\mathfrak{M}}}$ that will be crucial afterwards.

\begin{claim}
\label{claim: adjunctios}Consider the functor $L:{^{H}\mathfrak{M}}%
\rightarrow {^{H}\mathfrak{M}^{H}}$ defined on objects by $L({^{\bullet }}%
V):={^{\bullet }}V^{\circ }$ where the upper empty dot denotes the trivial
right coaction while the upper full dot denotes the given left $H$-coaction
of $V.$ The functor $L$ has a right adjoint $R:{^{H}\mathfrak{M}%
^{H}\rightarrow {^{H}\mathfrak{M}}}$ defined on objects by $R({^{\bullet }}%
M^{\bullet }):={^{\bullet }}M^{coH},$ where $M^{coH}:=\{m\in M\mid
m_{0}\otimes m_{1}=m\otimes 1_{H}\}$ is the space of right $H$-coinvariant
elements in $M$.

By Remark \ref{FuntT}, the forgetful functor $U:{^{H}\mathfrak{M}_{H}^{H}}%
\rightarrow {^{H}\mathfrak{M}^{H},U}\left( {^{\bullet }}M_{\bullet
}^{\bullet }\right) :={^{\bullet }}M^{\bullet }$ has a right adjoint, namely
the functor $T:{^{H}\mathfrak{M}^{H}}\rightarrow {^{H}\mathfrak{M}_{H}^{H}}%
,T\left( {^{\bullet }}M^{\bullet }\right) :={^{\bullet }}M^{\bullet }\otimes 
{^{\bullet }}H_{\bullet }^{\bullet }$. Here the upper dots indicate on which
tensor factors we have a codiagonal coaction and the lower dot indicates
where the action takes place. Explicitly, the structure of $T\left( {%
^{\bullet }}M^{\bullet }\right) $ is given as follows: 
\begin{eqnarray*}
\rho _{M\otimes H}^{l}(m\otimes h) &:&=m_{-1}h_{1}\otimes (m_{0}\otimes
h_{2}), \\
\rho _{M\otimes H}^{r}(m\otimes h) &:&=(m_{0}\otimes h_{1})\otimes
m_{1}h_{2}, \\
\mu _{M\otimes H}^{r}\left[ (m\otimes h)\otimes l\right] &=&(m\otimes
h)l:=\omega ^{-1}(m_{-1}\otimes h_{1}\otimes l_{1})m_{0}\otimes
h_{2}l_{2}\omega (m_{1}\otimes h_{3}\otimes l_{3}).
\end{eqnarray*}%
Define the functors $F:=TL:{^{H}\mathfrak{M}}\rightarrow {^{H}\mathfrak{M}%
_{H}^{H}}$ and $G:=RU:{^{H}\mathfrak{M}_{H}^{H}}\rightarrow {{^{H}\mathfrak{%
M.}}}$ Explicitly $G\left( {^{\bullet }}M_{\bullet }^{\bullet }\right) ={%
^{\bullet }}M^{coH}$ and $F({^{\bullet }}V):={^{\bullet }}V^{\circ }\otimes {%
^{\bullet }}H_{\bullet }^{\bullet }$ so that, for every $v\in V,h,l\in H,$ 
\begin{align*}
\rho _{V\otimes H}^{l}(v\otimes h)& =v_{-1}h_{1}\otimes (v_{0}\otimes h_{2}),
\\
\rho _{V\otimes H}^{r}(v\otimes h)& =(v\otimes h_{1})\otimes h_{2}, \\
\mu _{V\otimes H}^{r}\left[ (v\otimes h)\otimes l\right] & =(v\otimes
h)l=\omega ^{-1}(v_{-1}\otimes h_{1}\otimes l_{1})v_{0}\otimes h_{2}l_{2}.
\end{align*}
\end{claim}

\begin{remark}
\label{adjoint} By the right-hand version of \cite[Lemma 2.1]{Sch-TwoChar},
the functor $F:{^{H}\mathfrak{M}}\rightarrow {^{H}\mathfrak{M}_{H}^{H}}$ is
a left adjoint of the functor $G$, where the counit and the unit of the
adjunction are given respectively by $\epsilon _{M}:FG(M)\rightarrow
M,\epsilon _{M}(x\otimes h):=xh$ and by $\eta _{N}:N\rightarrow GF(N),\eta
_{N}\left( n\right) :=n\otimes 1_{H},$ for every $M\in $ $^{H}\mathfrak{M}%
_{H}^{H},N\in {^{H}\mathfrak{M}}.$ Moreover $\eta _{N}$ is an isomorphism
for any $N\in {^{H}\mathfrak{M}}.$ In particular the functor $F$ is fully
faithful.
\end{remark}

\subsection{The notion of preantipode}

Next result characterizes when the adjunction $(F,G)$ is an equivalence of
categories in term of the existence of a suitable map $\tau $.

\begin{proposition}
\cite[Proposition 3.3]{Ardi-Pava}\label{pro: tau} Let $(H,m,u,\Delta
,\varepsilon ,\omega )$ be a dual quasi-bialgebra. The following assertions
are equivalent.

\begin{enumerate}
\item[$(i)$] The adjunction $(F,G)$ is an equivalence.

\item[$(ii)$] For each $M\in {^{H}\mathfrak{M}_{H}^{H}},$ there exists a $%
\Bbbk $-linear map $\tau :M\rightarrow M^{coH}$ such that: 
\begin{eqnarray}
\tau (mh) &=&\omega ^{-1}[\tau (m_{0})_{-1}\otimes m_{1}\otimes h]\tau
(m_{0})_{0},\text{ for all }h\in H,m\in M,  \label{Tau mh} \\
m_{-1}\otimes \tau (m_{0}) &=&\tau (m_{0})_{-1}m_{1}\otimes \tau (m_{0})_{0},%
\text{ for all }m\in M,  \label{col sx eps} \\
\tau (m_{0})m_{1} &=&m\text{ }\forall m\in M.  \label{inv eps}
\end{eqnarray}

\item[$(iii)$] For each $M\in {^{H}\mathfrak{M}_{H}^{H}},$ there exists a $%
\Bbbk $-linear map $\tau :M\rightarrow M^{coH}$ such that (\ref{inv eps})
holds and 
\begin{equation}
\tau (mh)=m\varepsilon (h),\text{ for all }h\in H,m\in M^{coH}\text{.}
\label{Tau mh simple}
\end{equation}
\end{enumerate}
\end{proposition}

\begin{remark}
\label{rem: tau}Let $\tau :M\rightarrow M^{coH}$ be a $\Bbbk $-linear map
such that (\ref{inv eps}) holds. By \cite[Remark 3.4]{Ardi-Pava}, the map $%
\tau $ fulfills (\ref{Tau mh simple}) if and only if it fulfills (\ref{Tau mh}%
) and (\ref{col sx eps}).
\end{remark}

\begin{definition}
\label{preantipode}\ Following \cite[Definition 3.6]{Ardi-Pava} we will say
that a preantipode for a dual quasi-bialgebra $(H,m,u,\Delta ,\varepsilon
,\omega )$ is a $\Bbbk $-linear map $S:H\rightarrow H$ such that, for all $%
h\in H$,%
\begin{gather}
S(h_{1})_{1}h_{2}\otimes S(h_{1})_{2}=1_{H}\otimes S(h),  \label{col 1 S} \\
S(h_{2})_{1}\otimes h_{1}S(h_{2})_{2}=S(h)\otimes 1_{H},  \label{col 2 S} \\
\omega (h_{1}\otimes S(h_{2})\otimes h_{3})=\varepsilon (h).  \label{fond S}
\end{gather}
\end{definition}

\begin{remark}
\label{Lempreant}\cite[Remark 3.7]{Ardi-Pava} Let $(H,m,u,\Delta
,\varepsilon ,\omega ,S)$ be a dual quasi-bialgebra with a preantipode. Then
the following equalities hold%
\begin{equation}
h_{1}S(h_{2})=\varepsilon S(h)1_{H}=S(h_{1})h_{2}\text{ for all }h\in H.
\label{3S}
\end{equation}
\end{remark}

\begin{lemma}
\label{lem: tau}\cite[Lemma 3.8]{Ardi-Pava} Let $(H,m,u,\Delta ,\varepsilon
,\omega ,S)$ be a dual quasi-bialgebra with a preantipode. For any $M\in {%
^{H}\mathfrak{M}_{H}^{H}}$ and $m\in M$, set 
\begin{equation}
\tau (m):=\omega \lbrack m_{-1}\otimes S(m_{1})_{1}\otimes
m_{2}]m_{0}S(m_{1})_{2}.  \label{deftau}
\end{equation}%
Then (\ref{deftau}) defines a map $\tau :M\rightarrow M^{coH}$ which
fulfills (\ref{Tau mh}), (\ref{col sx eps}) and (\ref{inv eps}).
\end{lemma}

\begin{theorem}
\cite[Theorem 3.9]{Ardi-Pava}\label{Teoequidual} For a dual quasi-bialgebra $%
(H,m,u,\Delta ,\varepsilon ,\omega )$ the following are equivalent.

\begin{enumerate}
\item[$(i)$] The adjunction $(F,G)$ of Remark \ref{adjoint} is an
equivalence of categories.

\item[$(ii)$] There exists a preantipode.
\end{enumerate}
\end{theorem}

We include here some new results that will be needed later on in the paper.

\begin{lemma}
\label{lem:omeg-1S}Let $(H,m,u,\Delta ,\varepsilon ,\omega ,S)$ be a dual
quasi-bialgebra with a preantipode. Then
\begin{equation}
\omega ^{-1}\left[ S\left( h_{1}\right) \otimes h_{2}\otimes S\left(
h_{3}\right) \right] =\varepsilon S\left( h\right), \text{ for all }h\in H .  \label{form: omegasbis}
\end{equation}
\end{lemma}

\begin{proof}
Set $\alpha :=\omega \left( H\otimes H\otimes m\right) \ast \omega \left(
m\otimes H\otimes H\right) \ast m_{\Bbbk }\left( \omega ^{-1}\otimes
\varepsilon \right) $ and $\beta =m_{\Bbbk }\left( \varepsilon \otimes
\omega \right) \ast \omega \left( H\otimes m\otimes H\right) .$ Fix $h\in H$%
. We have%
\begin{eqnarray*}
&&\alpha \left( S\left( h_{1}\right) \otimes h_{2}\otimes S\left(
h_{3}\right) \otimes h_{4}\right) \\
&=&\omega \left[ S\left( h_{1}\right) _{1}\otimes h_{2}\otimes S\left(
h_{5}\right) _{\left( 1\right) }h_{6}\right] \omega \left[ S\left(
h_{1}\right) _{2}h_{3}\otimes S\left( h_{5}\right) _{\left( 2\right)
}\otimes h_{7}\right] \omega ^{-1}\left[ S\left( h_{1}\right) _{3}\otimes
h_{4}\otimes S\left( h_{5}\right) _{\left( 3\right) }\right] \\
&\overset{(\ref{col 1 S})}{=}&\omega \left[ S\left( h_{1}\right) _{1}\otimes
h_{2}\otimes 1_{H}\right] \omega \left[ S\left( h_{1}\right)
_{2}h_{3}\otimes S\left( h_{5}\right) _{\left( 1\right) }\otimes h_{6}\right]
\omega ^{-1}\left[ S\left( h_{1}\right) _{3}\otimes h_{4}\otimes S\left(
h_{5}\right) _{\left( 2\right) }\right] \\
&=&\omega \left[ S\left( h_{1}\right) _{1}h_{2}\otimes S\left( h_{4}\right)
_{\left( 1\right) }\otimes h_{5}\right] \omega ^{-1}\left[ S\left(
h_{1}\right) _{2}\otimes h_{3}\otimes S\left( h_{4}\right) _{\left( 2\right)
}\right] \\
&\overset{(\ref{col 1 S})}{=}&\omega \left[ 1_{H}\otimes S\left(
h_{4}\right) _{\left( 1\right) }\otimes h_{5}\right] \omega ^{-1}\left[
S\left( h_{1}\right) \otimes h_{3}\otimes S\left( h_{4}\right) _{\left(
2\right) }\right] \\
&=&\omega ^{-1}\left[ S\left( h_{1}\right) \otimes h_{2}\otimes S\left(
h_{3}\right) \right] ,
\end{eqnarray*}%
and%
\begin{eqnarray*}
&&\beta \left( S\left( h_{1}\right) \otimes h_{2}\otimes S\left(
h_{3}\right) \otimes h_{4}\right) \\
&=&\omega \left[ h_{2}\otimes S\left( h_{4}\right) _{\left( 1\right)
}\otimes h_{5}\right] \omega \left[ S\left( h_{1}\right) \otimes
h_{3}S\left( h_{4}\right) _{\left( 2\right) }\otimes h_{6}\right] \\
&\overset{(\ref{col 2 S})}{=}&\omega \left[ h_{2}\otimes S\left(
h_{3}\right) \otimes h_{4}\right] \omega \left[ S\left( h_{1}\right) \otimes
1_{H}\otimes h_{5}\right] \\
&=&\omega \left[ h_{2}\otimes S\left( h_{3}\right) \otimes h_{4}\right]
\varepsilon S\left( h_{1}\right) \\
&\overset{(\ref{fond S})}{=}&\varepsilon S\left( h\right) .
\end{eqnarray*}%
By the cocycle condition we have $\alpha =\beta $.
\end{proof}

\begin{definition}
\label{Bulant}\cite[page 66]{Maj2} A dual quasi-Hopf algebra $(H,m,u,\Delta
,\varepsilon ,\omega ,s,\alpha ,\beta )$ is a dual quasi-bialgebra $%
(H,m,u,\Delta ,\varepsilon ,\omega )$ endowed with a coalgebra
anti-homomorphism 
\begin{equation*}
s:H\rightarrow H
\end{equation*}%
and two maps $\alpha ,\beta $ in $H^{\ast }$, such that, for all $h\in H$:%
\begin{eqnarray}
h_{1}\beta (h_{2})s(h_{3}) &=&\beta (h)1_{H},  \label{ant 1} \\
s(h_{1})\alpha (h_{2})h_{3} &=&\alpha (h)1_{H},  \label{ant 2} \\
\omega (h_{1}\otimes \beta (h_{2})s(h_{3})\alpha (h_{4})\otimes h_{5})
&=&\varepsilon (h)=\omega ^{-1}(s(h_{1})\otimes \alpha (h_{2})h_{3}\beta
(h_{4})\otimes s(h_{5})).  \label{ant 3}
\end{eqnarray}
\end{definition}

In \cite[Theorem 3.10]{Ardi-Pava}, we proved that any dual quasi-Hopf
algebra has a preantipode. The following result proves that the converse
holds true whenever $H$ is also cocommutative.

\begin{theorem}
\label{teo:cocom}Let $(H,m,u,\Delta ,\varepsilon ,\omega ,S)$ be a dual
quasi-bialgebra with a preantipode. If $H$ is cocommutative, then $%
(H,m,u,\Delta ,\varepsilon ,s)$ is an ordinary Hopf algebra, where, for all $%
h\in H$, 
\begin{equation*}
s\left( h\right) :=S\left( h_{3}\right) _{1}\omega \left[ h_{1}\otimes
S\left( h_{3}\right) _{2}\otimes h_{2}\right] .
\end{equation*}%
Furthermore $(H,m,u,\Delta ,\varepsilon,\omega,\alpha ,\beta ,s)$ is a dual
quasi-Hopf algebra, where $\alpha :=\varepsilon $ and $\beta :=\varepsilon S$%
. Moreover one has $S=\beta \ast s.$
\end{theorem}

\begin{proof}
By (\ref{eq:quasi-associativity}), cocommutativity and convolution
invertibility of $\omega $, we get that $(hk)l=h(kl)$ for all $h,k,l\in H$. %
Therefore $m$ is associative and hence $(H,m,u,\Delta ,\varepsilon )$ is an
ordinary bialgebra. Let us check that $s$ is an antipode for $H.$ Using
cocommutativity, (\ref{col 1 S}) and (\ref{fond S}) one proves that $s\left(
h_{1}\right) h_{2}=1_{H}\varepsilon \left( h\right) $ for all $h\in H$.
Similarly one gets $h_{1}s(h_{2})=1_{H}\varepsilon \left( h\right) \ $for
all $h\in H$. 
Hence $(H,m,u,\Delta ,\varepsilon ,s)$ is an ordinary Hopf algebra. Note
that, for all $h\in H$, 
\begin{equation}
S(h)=S(h_{1})\left[ h_{2}s(h_{3})\right] =\left[ S(h_{1})h_{2}\right]
s(h_{3})\overset{(\ref{3S})}{=}\varepsilon S(h_{1})s(h_{2})=\beta \left(
h_{1}\right) s(h_{2}).  \label{form:Sfroms}
\end{equation}%
Let us check that $(H,m,u,\Delta ,\varepsilon ,\omega ,\alpha ,\beta ,s)$ is
a dual quasi-Hopf algebra. For all $h\in H$,%
\begin{gather*}
h_{1}\beta \left( h_{2}\right) s\left( h_{3}\right) \overset{(\ref%
{form:Sfroms})}{=}h_{1}S(h_{2})\overset{(\ref{3S})}{=}1_{H}\varepsilon S(h),
\\
s(h_{1})\alpha (h_{2})h_{3}=s(h_{1})h_{2}=1_{H}\varepsilon \left( h\right)
=1_{H}\alpha \left( h\right) , \\
\omega \left[ h_{1}\otimes \beta \left( h_{2}\right) s\left( h_{3}\right)
\alpha \left( h_{4}\right) \otimes h_{5}\right] \overset{(\ref{form:Sfroms})}%
{=}\omega \left[ h_{1}\otimes S\left( h_{2}\right) \otimes h_{3}\right] 
\overset{(\ref{fond S})}{=}1_{H}\varepsilon \left( h\right) .
\end{gather*}%
Now, since $(H,m,u,\Delta ,\varepsilon ,s)$ is an ordinary Hopf algebra, we
have that $s$ is an anti-coalgebra map. Thus%
\begin{gather*}
S\left( h\right) _{1}\otimes S\left( h\right) _{2}\overset{(\ref{form:Sfroms}%
)}{=}\beta \left( h_{1}\right) s\left( h_{2}\right) _{1}\otimes s\left(
h_{2}\right) _{2}=\beta \left( h_{1}\right) s\left( h_{3}\right) \otimes
s\left( h_{2}\right) \\
\overset{\text{cocom.}}{=}\beta \left( h_{1}\right) s\left( h_{2}\right)
\otimes s\left( h_{3}\right) \overset{(\ref{form:Sfroms})}{=}S(h_{1})\otimes
s\left( h_{2}\right)
\end{gather*}%
so that%
\begin{eqnarray*}
&&\omega ^{-1}\left[ s\left( h_{1}\right) \otimes \alpha \left( h_{2}\right)
h_{3}\beta \left( h_{4}\right) \otimes s\left( h_{5}\right) \right] \\
&\overset{(\ref{form:Sfroms})}{=}&\omega ^{-1}\left[ s\left( h_{1}\right)
\otimes h_{2}\otimes S\left( h_{3}\right) \right] \\
&=&\omega ^{-1}\left[ S\left( h_{3}\right) _{1}\otimes h_{4}\otimes S\left(
h_{5}\right) \right] \omega \left[ h_{1}\otimes S\left( h_{3}\right)
_{2}\otimes h_{2}\right] \\
&=&\omega ^{-1}\left[ S\left( h_{3}\right) \otimes h_{5}\otimes S\left(
h_{6}\right) \right] \omega \left[ h_{1}\otimes s\left( h_{4}\right) \otimes
h_{2}\right] \\
&\overset{\text{cocom.}}{=}&\omega ^{-1}\left[ S\left( h_{2}\right) \otimes
h_{3}\otimes S\left( h_{4}\right) \right] \omega \left[ h_{1}\otimes s\left(
h_{5}\right) \otimes h_{6}\right] \\
&\overset{\eqref{form: omegasbis}}{=}&\varepsilon S\left( h_{2}\right)
\omega \left[ h_{1}\otimes s\left( h_{3}\right) \otimes h_{4}\right] \\
&\overset{(\ref{form:Sfroms})}{=}&\omega \left[ h_{1}\otimes S\left(
h_{2}\right) \otimes h_{3}\right] \overset{(\ref{fond S})}{=}%
1_{H}\varepsilon \left( h\right) .
\end{eqnarray*}
\end{proof}

\begin{definition}
A dual quasi-bialgebra $(A,m,u,\Delta ,\varepsilon ,\omega )$ is called
pointed if the underlying coalgebra is pointed, i.e. all its simple
subcoalgebras are one dimensional.
\end{definition}

\begin{definition}
Let $(A,m,u,\Delta ,\varepsilon ,\omega )$ be a dual quasi-bialgebra. The set%
\begin{equation*}
\mathbb{G}\left( A\right) =\left\{ a\in A\mid \Delta (a)=a\otimes a\text{ and }%
\varepsilon (a)=1\right\}
\end{equation*}%
is called the set of the grouplike elements of $A.$
\end{definition}

\begin{remark}
\label{rem: pointed}Let $A$ be a pointed dual quasi-bialgebra. We know that
the $1$-dimensional subcoalgebras of $A$ are exactly those of the form $%
\Bbbk g$ for $g\in G$ (\cite[page 57]{Sweedler}). Thus the coradical of $A$
is $A_{0}=\sum_{g\in G}\Bbbk g=\Bbbk \mathbb{G}\left( A\right) .$
\end{remark}

The following results extends the so-called Cartier-Gabriel-Kostant to dual
quasi-bialgebras with a preantipode. In the connected case such a result was
achieved in \cite[Theorem 4.3]{Huang-QuiverAppr}.

\begin{corollary}
\label{coro:cartan}Let $H$ be a dual quasi-bialgebra with a preantipode over
a field $\Bbbk $ of characteristic zero. If $H$ is cocommutative and
pointed, then $H$ is an ordinary Hopf algebra isomorphic to the biproduct $%
U\left( P\left( H\right) \right) \#\Bbbk \mathbb{G}\left( H\right) ,$ where $%
P\left( H\right) $ denotes the Lie algebra of primitive elements in $H$.
\end{corollary}

\begin{proof}
By Theorem \ref{teo:cocom}, $H$ is an ordinary Hopf algebra. By \cite[%
Section 13.1, page 279]{Sweedler}, we conclude (see also \cite[page 79]%
{Montgomery}).
\end{proof}

\section{Yetter-Drinfeld modules over a dual quasi-bialgebra \label{C2}}

The main aim of this section is to restrict the equivalence between ${^{H}%
\mathfrak{M}_{H}^{H}}$ and ${^{H}\mathfrak{M}}$ of Theorem \ref{Teoequidual}%
, to an equivalence between ${_{H}^{H}\mathfrak{M}_{H}^{H}}$ and ${_{H}^{H}%
\mathcal{YD}}$ (the category of Yetter-Drinfeld modules over $H$) for any
dual quasi-bialgebra $H$ with a preantipode.

\begin{definition}
\label{def: YD}Let $(H,m,u,\Delta ,\varepsilon ,\omega )$ be a dual
quasi-bialgebra. The category ${_{H}^{H}\mathcal{YD}}$ of Yetter-Drinfeld
modules over $H$, is defined as follows. An object in ${_{H}^{H}\mathcal{YD}}
$ is a tern $\left( V,\rho _{V},\vartriangleright \right) ,$ where

\begin{itemize}
\item $(V,\rho )$ is an object in ${^{H}\mathfrak{M}}$

\item $\vartriangleright :H\otimes V\rightarrow V$ is a $\Bbbk $-linear map
such that, for all $h,l\in L$ and $v\in V$ 
\begin{equation}
\left( hl\right) \vartriangleright v=\left[ 
\begin{array}{c}
\omega ^{-1}\left( h_{1}\otimes l_{1}\otimes v_{-1}\right) \omega \left(
h_{2}\otimes \left( l_{2}\vartriangleright v_{0}\right) _{-1}\otimes
l_{3}\right) \\ 
\omega ^{-1}\left( (h_{3}\vartriangleright (l_{2}\vartriangleright
v_{0})_{0}\right) _{-1}\otimes h_{4}\otimes l_{4})\left(
h_{3}\vartriangleright \left( l_{2}\vartriangleright v_{0}\right)
_{0}\right) _{0}%
\end{array}%
\right] ,  \label{ass YD}
\end{equation}%
\begin{equation}
1_{H}\vartriangleright v=v \qquad \text{and}  \label{unitYD}
\end{equation}
\begin{equation}
\left( h_{1}\vartriangleright v\right) _{-1}h_{2}\otimes \left(
h_{1}\vartriangleright v\right) _{0}=h_{1}v_{-1}\otimes \left(
h_{2}\vartriangleright v_{0}\right)  \label{Comp YD}
\end{equation}
\end{itemize}

A morphism $f:(V,\rho ,\vartriangleright )\rightarrow (V^{\prime },\rho
^{\prime },\vartriangleright ^{\prime })$ in $_{H}^{H}\mathcal{YD}$ is a
morphism $f:(V,\rho )\rightarrow (V^{\prime },\rho ^{\prime })$ in $^{H}%
\mathfrak{M}$ such that $f(h\vartriangleright v)=h\vartriangleright ^{\prime
}f(v)$.
\end{definition}

\begin{claim}
The category ${_{H}^{H}\mathcal{YD}}$ is isomorphic to the weak right center
of ${^{H}\mathfrak{M}}$ (regarded as a monoidal category as in Section \ref%
{C1}), see Theorem \ref{teo:WeakCenter}. As a consequence ${_{H}^{H}\mathcal{%
YD}}$ has a pre-braided monoidal structure given as follows. The unit is $%
\Bbbk $ regarded as an object in ${_{H}^{H}\mathcal{YD}}$ via trivial
structures i.e. $\rho _{\Bbbk }\left( k\right) =1_{H}\otimes k$ and $%
h\vartriangleright k=\varepsilon \left( h\right) k.$ The tensor product is
defined by 
\begin{equation*}
\left( V,\rho _{V},\vartriangleright \right) \otimes \left( W,\rho
_{W},\vartriangleright \right) =\left( V\otimes W,\rho _{V\otimes
W},\vartriangleright \right)
\end{equation*}%
where $\rho _{V\otimes W}\left( v\otimes w\right) =v_{-1}w_{-1}\otimes
v_{0}\otimes w_{0}$ and%
\begin{equation}
h\vartriangleright \left( v\otimes w\right) =\left[ 
\begin{array}{c}
\omega \left( h_{1}\otimes v_{-1}\otimes w_{-2}\right) \omega ^{-1}\left(
\left( h_{2}\vartriangleright v_{0}\right) _{-2}\otimes h_{3}\otimes
w_{-1}\right) \\ 
\omega \left( \left( h_{2}\vartriangleright v_{0}\right) _{-1}\otimes \left(
h_{4}\vartriangleright w_{0}\right) _{-1}\otimes h_{5}\right) \left(
h_{2}\vartriangleright v_{0}\right) _{0}\otimes \left(
h_{4}\vartriangleright w_{0}\right) _{0}%
\end{array}%
\right] .  \label{form:YDtens}
\end{equation}%
The constraints are the same of ${^{H}\mathfrak{M}}$ viewed as morphisms in $%
{_{H}^{H}\mathcal{YD}}$. The braiding $c_{V,W}:V\otimes W\rightarrow
W\otimes V$ is given by 
\begin{equation}
c_{V,W}\left( v\otimes w\right) =\left( v_{-1}\vartriangleright w\right)
\otimes v_{0}.  \label{braiding YD}
\end{equation}
\end{claim}

\begin{remark}
It is easily checked that condition (\ref{ass YD}) holds for all $h,l\in L$
and $v\in V$ if and only if 
\begin{equation*}
c_{H\otimes H,V}={^{H}}a_{V,H,H}\circ \left( c_{H,V}\otimes H\right) \circ {%
^{H}}a_{H,V,H}^{-1}\circ \left( H\otimes c_{H,V}\right) \circ {^{H}}%
a_{H,H,V},
\end{equation*}%
where ${^{H}}a$ is the associativity constraint in ${^{H}\mathfrak{M}}$. 
Now, the displayed equality above, can be written as%
\begin{equation*}
{^{H}}a_{V,H,H}^{-1}\circ c_{H\otimes H,V}\circ {^{H}}a_{H,H,V}^{-1}=\left(
c_{H,V}\otimes H\right) \circ {^{H}}a_{H,V,H}^{-1}\circ \left( H\otimes
c_{H,V}\right) .
\end{equation*}%
One easily checks that this is equivalent to ask that 
\begin{eqnarray*}
&&\omega \left( h_{1}\otimes l_{1}\otimes v_{-1}\right) \omega \left( \left(
\left( h_{2}l_{2}\right) \vartriangleright v_{0}\right) _{-1}\otimes
h_{3}\otimes l_{3}\right) \left( \left( h_{2}l_{2}\right) \vartriangleright
v_{0}\right) _{0} \\
&=&\omega (h_{1}\otimes (l_{1}\vartriangleright v)_{-1}\otimes
l_{2})h_{3}\vartriangleright (l_{1}\vartriangleright v)_{0}
\end{eqnarray*}%
holds for all $h,l\in L$ and $v\in V.$ This equation is the left-handed
version of \cite[(3.3)]{Balan}. In conclusion, the axioms defining the
category ${_{H}^{H}\mathcal{YD}}$ are the left-handed version of the ones
appearing in \cite[Definition 3.1]{Balan}.
\end{remark}

\subsection{The restriction of the equivalence (F,G)}

Let $H$ be a dual quasi-bialgebra. From Theorem \ref{Teoequidual}, we know
that the adjunction $(F,G)$ of Remark \ref{adjoint} is an equivalence of
categories when $H$ has a preantipode. Next aim is to prove that 
$(F,G)$ restricts to an equivalence between the categories $_{H}^{H}\mathcal{%
YD}$ and $_{H}^{H}\mathfrak{M}_{H}^{H}$.

Inspired by \cite[page 541]{Sch-TwoChar} we get the following result.

\begin{lemma}
\label{lem: xiScha} Let $(H,m,u,\Delta ,\varepsilon ,\omega )$ be a dual
quasi-bialgebra. For all $U\in {^{H}{\mathfrak{M}}}$ and $M\in {{_{H}^{H}%
\mathfrak{M}_{H}^{H}}}${,} we have a map 
\begin{eqnarray*}
\xi _{U,M} &:&F\left( U\right) \otimes _{H}M\rightarrow U\otimes M, \\
\xi _{U,M}\left( \left( u\otimes h\right) \otimes _{H}m\right) &=&\omega
^{-1}\left( u_{-1}\otimes h_{1}\otimes m_{-1}\right) u_{0}\otimes h_{2}m_{0}
\end{eqnarray*}%
which is a $\Bbbk $-linear natural isomorphism with inverse given by $\xi
_{U,M}^{-1}\left( u\otimes m\right) =\left( u\otimes 1_{H}\right) \otimes
_{H}m.$ Moreover:

\begin{itemize}
\item[1)] the map $\xi _{U,M}$ is a natural isomorphism in ${{^{H}\mathfrak{M%
}_{H}^{H}}}$ where $U\otimes M$ has the following structures:%
\begin{eqnarray*}
{\rho _{U\otimes M}^{l}}\left( u\otimes m\right) &=&u_{-1}m_{-1}\otimes
(u_{0}\otimes m_{0}), \\
\rho _{U\otimes M}^{r}\left( u\otimes m\right) &=&(u\otimes m_{0})\otimes
m_{1}, \\
\mu _{U\otimes M}^{r}(\left( u\otimes m\right) \otimes h) &=&\omega
^{-1}(u_{-1}\otimes m_{-1}\otimes h_{1})u_{0}\otimes m_{0}h_{2};
\end{eqnarray*}

\item[2)] if $U\in {_{H}^{H}\mathcal{YD}}$, the map $\xi _{U,M}$ is a
natural isomorphism in ${{_{H}^{H}\mathfrak{M}_{H}^{H}}}$ where $U\otimes M$
has the structures above along with the following left module structure: 
\begin{equation*}
{\mu ^{l}}_{U\otimes M}(h\otimes \left( u\otimes m\right) )=\omega \left(
h_{1}\otimes u_{-1}\otimes m_{-2}\right) \omega ^{-1}\left( \left(
h_{2}\vartriangleright u_{0}\right) _{-1}\otimes h_{3}\otimes m_{-1}\right)
\left( h_{2}\vartriangleright u_{0}\right) _{0}\otimes h_{4}m_{0}.
\end{equation*}
\end{itemize}
\end{lemma}

\begin{proof}
Clearly $U\otimes M\in {{^{H}\mathfrak{M}^{H}}}$ via ${\rho _{U\otimes M}^{l}%
}$ and $\rho _{U\otimes M}^{r}$. Let $\xi _{U,M}^{\prime }:F\left( U\right)
\otimes M\rightarrow U\otimes M$ be defined by $\xi _{U,M}^{\prime }\left(
\left( u\otimes h\right) \otimes m\right) =\omega ^{-1}\left( u_{-1}\otimes
h_{1}\otimes m_{-1}\right) u_{0}\otimes h_{2}m_{0}.$

Using the quasi-associativity condition \eqref{eq:quasi-associativity}, one
easily checks that $\xi _{U,M}^{\prime }$ is in ${{^{H}\mathfrak{M}^{H}}}$.

Let us check that $\xi _{U,M}^{\prime }$ is balanced in ${{^{H}\mathfrak{M}%
^{H}}}$ i.e. that it equalizes the maps%
\begin{equation*}
\balxi%
\end{equation*}%
%
%
%
%
%
%
%
%
%
%
%
%
%
%
%
%
%
%
%
%
%
%
%
%
%
%
%
%
We have 
\begin{eqnarray*}
&&\xi _{U,M}^{\prime }\left( \mu _{F\left( U\right) }^{r}\otimes M\right)
\left( (\left( u\otimes h\right) \otimes l)\otimes m\right) \\
&=&\omega ^{-1}\left( u_{-1}\otimes h_{1}\otimes l_{1}\right) \xi
_{U,M}^{\prime }\left( \left( u_{0}\otimes h_{2}l_{2}\right) \otimes m\right)
\\
&=&\omega ^{-1}\left( u_{-2}\otimes h_{1}\otimes l_{1}\right) \omega
^{-1}\left( u_{-1}\otimes h_{2}l_{2}\otimes m_{-1}\right) u_{0}\otimes
\left( h_{3}l_{3}\right) m_{0} \\
&=&\left[ 
\begin{array}{c}
\omega ^{-1}\left( u_{-2}\otimes h_{1}\otimes l_{1}\right) \omega
^{-1}\left( u_{-1}\otimes h_{2}l_{2}\otimes m_{-2}\right) \omega ^{-1}\left(
h_{3}\otimes l_{3}\otimes m_{-1}\right) \\ 
u_{0}\otimes h_{4}\left( l_{4}m_{0}\right) \omega \left( h_{5}\otimes
l_{5}\otimes m_{1}\right)%
\end{array}%
\right] \\
&\overset{(\ref{eq:3-cocycle})}{=}&\omega ^{-1}\left( u_{-2}h_{1}\otimes
l_{1}\otimes m_{-2}\right) \omega ^{-1}\left( u_{-1}\otimes h_{2}\otimes
l_{2}m_{-1}\right) u_{0}\otimes h_{3}\left( l_{3}m_{0}\right) \omega \left(
h_{4}\otimes l_{4}\otimes m_{1}\right) \\
&=&\omega ^{-1}\left( u_{-1}h_{1}\otimes l_{1}\otimes m_{-2}\right) \xi
_{U,M}^{\prime }\left( F\left( U\right) \otimes \mu _{M}^{l}\right) \left(
(u_{0}\otimes h_{2})\otimes (l_{2}\otimes m_{0})\right) \omega \left(
h_{3}\otimes l_{3}\otimes m_{1}\right) \\
&=&\xi _{U,M}^{\prime }\left( F\left( U\right) \otimes \mu _{M}^{l}\right) {{%
^{H}}}a_{F\left( U\right) ,H,M}^{H}\left( (\left( u\otimes h\right) \otimes
l)\otimes m\right) .
\end{eqnarray*}%
Hence there exists a unique morphism $\xi _{U,M}:F\left( U\right) \otimes
_{H}M\rightarrow U\otimes M$ in ${{^{H}\mathfrak{M}^{H}}}$ such that $\xi
_{U,M}\left( \left( u\otimes h\right) \otimes _{H}m\right) =\xi
_{U,M}^{\prime }\left( \left( u\otimes h\right) \otimes m\right) .$ This
proves that $\xi _{U,M}$ is well-defined.

We now check that $\xi _{U,M}$ is invertible. Define%
\begin{equation*}
\overline{\xi }_{U,M}:U\otimes M\rightarrow F\left( U\right) \otimes
_{H}M,\qquad \overline{\xi }_{U,M}\left( u\otimes m\right) =\left( u\otimes
1_{H}\right) \otimes _{H}m.
\end{equation*}%
We have $\xi _{U,M}\circ \overline{\xi }_{U,M}=\mathrm{Id}_{U\otimes M}$ 
and%
\begin{eqnarray*}
&&\overline{\xi }_{U,M}\xi _{U,M}\left( \left( u\otimes h\right) \otimes
_{H}m\right) \\
&=&\omega ^{-1}\left( u_{-1}\otimes h_{1}\otimes m_{-1}\right) \left(
u_{0}\otimes 1_{H}\right) \otimes _{H}h_{2}m_{0} \\
&\overset{\text{def }\otimes _{H}}{=}&\left[ 
\begin{array}{c}
\omega ^{-1}\left( u_{-1}\otimes h_{1}\otimes m_{-2}\right) \omega \left(
\left( u_{0}\otimes 1_{H}\right) _{-1}\otimes h_{2}\otimes m_{-1}\right) \\ 
\left( u_{0}\otimes 1_{H}\right) _{0}h_{3}\otimes _{H}m_{0}\omega
^{-1}\left( \left( u_{0}\otimes 1_{H}\right) _{1}\otimes h_{2}\otimes
m_{1}\right)%
\end{array}%
\right] \\
&=&\left[ 
\begin{array}{c}
\omega ^{-1}\left( u_{-2}\otimes h_{1}\otimes m_{-2}\right) \omega \left(
u_{-1}\otimes h_{2}\otimes m_{-1}\right) \\ 
\left( u_{0}\otimes 1_{H}\right) h_{3}\otimes _{H}m_{0}\omega ^{-1}\left(
1_{H}\otimes h_{4}\otimes m_{1}\right)%
\end{array}%
\right] \\
&=&\left( u\otimes 1_{H}\right) h\otimes _{H}m=\left( u\otimes h\right)
\otimes _{H}m.
\end{eqnarray*}%
The proof that $\xi _{U,M}^{-1}:=\overline{\xi }_{U,M}$ is natural in $U$
and $M$ is straightforward.


1) In order to have that $\xi _{U,M}$ is in ${^{H}\mathfrak{M}_{H}^{H}}$, it
suffices to prove that $\xi _{U,M}^{\prime }$ is in ${^{H}\mathfrak{M}%
_{H}^{H}}$. Clearly, $\xi _{U,M}^{\prime }$ is in $^H\mathfrak{M}^H$ being an inverse of $\xi _{U,M}$.

The map $\xi _{U,M}^{\prime }$ is right $H$-linear in ${{^{H}\mathfrak{M}^{H}%
}}$:%
\begin{eqnarray*}
&&\xi _{U,M}^{\prime }\left[ \left( \left( u\otimes h\right) \otimes
m\right) l\right] \\
&=&\omega ^{-1}\left( \left( u\otimes h\right) _{-1}\otimes m_{-1}\otimes
l_{1}\right) \xi _{U,M}^{\prime }\left[ \left( u\otimes h\right) _{0}\otimes
m_{0}l_{2}\right] \omega \left( \left( u\otimes h\right) _{1}\otimes
m_{1}\otimes l_{3}\right) \\
&=&\omega ^{-1}\left( u_{-1}h_{1}\otimes m_{-1}\otimes l_{1}\right) \xi
_{U,M}^{\prime }\left[ \left( u_{0}\otimes h_{2}\right) \otimes m_{0}l_{2}%
\right] \omega \left( h_{3}\otimes m_{1}\otimes l_{3}\right) \\
&=&\left[ 
\begin{array}{c}
\omega ^{-1}\left( u_{-2}h_{1}\otimes m_{-2}\otimes l_{1}\right) \omega
^{-1}\left( u_{-1}\otimes h_{2}\otimes m_{-1}l_{2}\right) \\ 
u_{0}\otimes h_{3}\left( m_{0}l_{3}\right) \omega \left( h_{4}\otimes
m_{1}\otimes l_{4}\right)%
\end{array}%
\right] \\
&\overset{\eqref{eq:quasi-associativity}}{=}&\left[ 
\begin{array}{c}
\omega ^{-1}\left( u_{-2}h_{1}\otimes m_{-3}\otimes l_{1}\right) \omega
^{-1}\left( u_{-1}\otimes h_{2}\otimes m_{-2}l_{2}\right) \omega \left(
h_{3}\otimes m_{-1}\otimes l_{3}\right) \\ 
u_{0}\otimes (h_{4}m_{0})l_{4}%
\end{array}%
\right] \\
&\overset{(\ref{eq:3-cocycle})}{=}&\omega ^{-1}\left( u_{-2}\otimes
h_{1}\otimes m_{-2}\right) \omega ^{-1}\left( u_{-1}\otimes
h_{2}m_{-1}\otimes l_{1}\right) u_{0}\otimes (h_{3}m_{0})l_{2} \\
&=&\omega ^{-1}\left( u_{-1}\otimes h_{1}\otimes m_{-1}\right) \left(
u_{0}\otimes h_{2}m_{0}\right) l \\
&=&\xi _{U,M}^{\prime }\left( \left( u\otimes h\right) \otimes m\right) l
\end{eqnarray*}%
2) $\xi _{U,M}^{\prime }$ is left $H$-linear in ${{^{H}\mathfrak{M}^{H}}}$:%
\begin{eqnarray*}
&&\xi _{U,M}^{\prime }\left[ l\left( \left( u\otimes h\right) \otimes
m\right) \right] \\
&=&\omega \left( l_{1}\otimes \left( u\otimes h\right) _{-1}\otimes
m_{-1}\right) \xi _{U,M}^{\prime }\left[ l_{2}\left( u\otimes h\right)
_{0}\otimes m_{0}\right] \omega ^{-1}\left( l_{3}\otimes \left( u\otimes
h\right) _{1}\otimes m_{1}\right) \\
&=&\omega \left( l_{1}\otimes u_{-1}h_{1}\otimes m_{-1}\right) \xi
_{U,M}^{\prime }\left[ l_{2}\left( u_{0}\otimes h_{2}\right) \otimes m_{0}%
\right] \omega ^{-1}\left( l_{3}\otimes h_{3}\otimes m_{1}\right) \\
&=&\left[ 
\begin{array}{c}
\omega \left( l_{1}\otimes u_{-2}h_{1}\otimes m_{-1}\right) \omega
(l_{2}\otimes u_{-1}\otimes h_{2})\omega ^{-1}((l_{3}\vartriangleright
u_{0})_{-1}\otimes l_{4}\otimes h_{3}) \\ 
\xi _{U,M}^{\prime }\left[ \left\{ (l_{3}\vartriangleright u_{0})_{0}\otimes
l_{5}h_{4}\right\} \otimes m_{0}\right] \omega ^{-1}\left( l_{6}\otimes
h_{5}\otimes m_{1}\right)%
\end{array}%
\right] \\
&=&\left[ 
\begin{array}{c}
\omega \left( l_{1}\otimes u_{-2}h_{1}\otimes m_{-2}\right) \omega
(l_{2}\otimes u_{-1}\otimes h_{2})\omega ^{-1}((l_{3}\vartriangleright
u_{0})_{-2}\otimes l_{4}\otimes h_{3}) \\ 
\omega ^{-1}\left( (l_{3}\vartriangleright u_{0})_{-1}\otimes
l_{5}h_{4}\otimes m_{-1}\right) (l_{3}\vartriangleright u_{0})_{0}\otimes
\left( l_{6}h_{5}\right) m_{0}\omega ^{-1}\left( l_{7}\otimes h_{6}\otimes
m_{1}\right)%
\end{array}%
\right] \\
&\overset{\eqref{eq:quasi-associativity}}{=}&\left[ 
\begin{array}{c}
\omega \left( l_{1}\otimes u_{-2}h_{1}\otimes m_{-3}\right) \omega
(l_{2}\otimes u_{-1}\otimes h_{2})\omega ^{-1}((l_{3}\vartriangleright
u_{0})_{-2}\otimes l_{4}\otimes h_{3}) \\ 
\omega ^{-1}\left( (l_{3}\vartriangleright u_{0})_{-1}\otimes
l_{5}h_{4}\otimes m_{-2}\right) \omega ^{-1}\left( l_{6}\otimes h_{5}\otimes
m_{-1}\right) (l_{3}\vartriangleright u_{0})_{0}\otimes l_{7}(h_{6}m_{0})%
\end{array}%
\right] \\
&\overset{(\ref{eq:3-cocycle})}{=}&\left[ 
\begin{array}{c}
\omega \left( l_{1}\otimes u_{-2}h_{1}\otimes m_{-3}\right) \omega
(l_{2}\otimes u_{-1}\otimes h_{2})\omega ^{-1}\left( (l_{3}\vartriangleright
u_{0})_{-2}l_{4}\otimes h_{3}\otimes m_{-2}\right) \\ 
\omega ^{-1}\left( (l_{3}\vartriangleright u_{0})_{-1}\otimes l_{5}\otimes
h_{4}m_{-1}\right) (l_{3}\vartriangleright u_{0})_{0}\otimes
l_{6}(h_{5}m_{0})%
\end{array}%
\right] \\
&\overset{(\ref{Comp YD})}{=}&\left[ 
\begin{array}{c}
\omega \left( l_{1}\otimes u_{-3}h_{1}\otimes m_{-3}\right) \omega
(l_{2}\otimes u_{-2}\otimes h_{2})\omega ^{-1}\left( l_{3}u_{-1}\otimes
h_{3}\otimes m_{-2}\right) \\ 
\omega ^{-1}\left( \left( l_{4}\vartriangleright u_{0}\right) _{-1}\otimes
l_{5}\otimes h_{4}m_{-1}\right) \left( l_{4}\vartriangleright u_{0}\right)
_{0}\otimes l_{6}(h_{5}m_{0})%
\end{array}%
\right] \\
&\overset{(\ref{eq:3-cocycle})}{=}&\left[ 
\begin{array}{c}
\omega ^{-1}\left( u_{-2}\otimes h_{1}\otimes m_{-3}\right) \omega \left(
l_{1}\otimes u_{-1}\otimes h_{2}m_{-2}\right) \\ 
\omega ^{-1}\left( \left( l_{2}\vartriangleright u_{0}\right) _{-1}\otimes
l_{3}\otimes h_{3}m_{-1}\right) \left( l_{2}\vartriangleright u_{0}\right)
_{0}\otimes l_{4}(h_{4}m_{0})%
\end{array}%
\right] \\
&=&\omega ^{-1}\left( u_{-1}\otimes h_{1}\otimes m_{-1}\right) l\left[
u_{0}\otimes h_{2}m_{0}\right] =l\xi _{U,M}^{\prime }\left( \left( u\otimes
h\right) \otimes m\right) .
\end{eqnarray*}
\end{proof}

\begin{lemma}
\label{lem:a4H} Let $(H,m,u,\Delta ,\varepsilon ,\omega )$ be a dual
quasi-bialgebra. For all $U,V\in {^{H}}\mathfrak{M}${,} consider the map%
\begin{eqnarray*}
\alpha _{U,V} &:&U\otimes \left( V\otimes H\right) \rightarrow \left(
U\otimes V\right) \otimes H \\
\alpha _{U,V}\left( u\otimes \left( v\otimes k\right) \right) &=&\omega
\left( u_{-1}\otimes v_{-1}\otimes k_{1}\right) \left( u_{0}\otimes
v_{0}\right) \otimes k_{2}.
\end{eqnarray*}

\begin{itemize}
\item[1)] The map $\alpha _{U,V}:U\otimes F\left( V\right) \rightarrow
F\left( U\otimes V\right) $ is a natural isomorphism in ${^{H}\mathfrak{M}%
_{H}^{H}}$, where $U\otimes F\left( V\right) $ has the structure described
in Lemma \ref{lem: xiScha} for $M=F\left( V\right) .$

\item[2)] If $U,V\in {_{H}^{H}\mathcal{YD}}$, then $\alpha _{U,V}:U\otimes
F\left( V\right) \rightarrow F\left( U\otimes V\right) $ is a natural
isomorphism in ${_{H}^{H}\mathfrak{M}_{H}^{H}}$, where $U\otimes F\left(
V\right) $ has the structure described in Lemma \ref{lem: xiScha} for $%
M=F\left( V\right) .$
\end{itemize}
\end{lemma}

\begin{proof}
Note that $\alpha _{U,V}=\left( {^{H}}a_{U,V,H}\right) ^{-1}$ so that $%
\alpha _{U,V}\in {^{H}}\mathfrak{M}$ and it is invertible.

1) Let us check that $\alpha _{U,V}:U\otimes F\left( V\right) \rightarrow
F\left( U\otimes V\right) $ is a morphism in ${^{H}\mathfrak{M}_{H}^{H}}$,
where $U\otimes F\left( V\right) $ has the structure described in Lemma \ref%
{lem: xiScha} for $M=F\left( V\right) .$

It is easy to check that $\alpha _{U,V}$ is right $H$-colinear. 
Moreover the $3$-cocycle condition \eqref{eq:3-cocycle} yields that $\alpha
_{U,V}$ is right $H$-linear in ${^{H}\mathfrak{M}^{H}}$, 
i.e. that $\alpha _{U,V}$ is a morphism in ${^{H}\mathfrak{M}_{H}^{H}}$.

2) Let us check that $\alpha _{U,V}$ is left $H$-linear in ${^{H}\mathfrak{M}%
^{H}}$. On the one hand we have 
\begin{eqnarray*}
&&\alpha _{U,V}\left[ h\left( u\otimes \left( v\otimes k\right) \right) %
\right] \\
&=&\omega \left( h_{1}\otimes u_{-1}\otimes \left( v\otimes k\right)
_{-2}\right) \omega ^{-1}\left( \left( h_{2}\vartriangleright u_{0}\right)
_{-1}\otimes h_{3}\otimes \left( v\otimes k\right) _{-1}\right) \alpha
_{U,V} \left[ \left( h_{2}\vartriangleright u_{0}\right) _{0}\otimes
h_{4}\left( v\otimes k\right) _{0}\right] \\
&=&\omega \left( h_{1}\otimes u_{-1}\otimes v_{-2}k_{1}\right) \omega
^{-1}\left( \left( h_{2}\vartriangleright u_{0}\right) _{-1}\otimes
h_{3}\otimes v_{-1}k_{2}\right) \alpha _{U,V}\left[ \left(
h_{2}\vartriangleright u_{0}\right) _{0}\otimes h_{4}\left( v_{0}\otimes
k_{3}\right) \right] \\
&=&\left[ 
\begin{array}{c}
\omega \left( h_{1}\otimes u_{-1}\otimes v_{-3}k_{1}\right) \omega
^{-1}\left( \left( h_{2}\vartriangleright u_{0}\right) _{-1}\otimes
h_{3}\otimes v_{-2}k_{2}\right) \omega (h_{4}\otimes v_{-1}\otimes k_{3}) \\ 
\omega ^{-1}((h_{5}\vartriangleright v_{0})_{-1}\otimes h_{6}\otimes
k_{4})\alpha _{U,V}\left[ \left( h_{2}\vartriangleright u_{0}\right)
_{0}\otimes \left[ (h_{5}\vartriangleright v_{0})_{0}\otimes h_{7}k_{5}%
\right] \right]%
\end{array}%
\right] \\
&=&\left[ 
\begin{array}{c}
\omega \left( h_{1}\otimes u_{-1}\otimes v_{-3}k_{1}\right) \omega
^{-1}\left( \left( h_{2}\vartriangleright u_{0}\right) _{-2}\otimes
h_{3}\otimes v_{-2}k_{2}\right) \omega (h_{4}\otimes v_{-1}\otimes k_{3}) \\ 
\omega ^{-1}((h_{5}\vartriangleright v_{0})_{-2}\otimes h_{6}\otimes
k_{4})\omega \left( \left( h_{2}\vartriangleright u_{0}\right) _{-1}\otimes
(h_{5}\vartriangleright v_{0})_{-1}\otimes h_{7}k_{5}\right) \\ 
\left[ \left( h_{2}\vartriangleright u_{0}\right) _{0}\otimes
(h_{5}\vartriangleright v_{0})_{0}\right] \otimes h_{8}k_{6}%
\end{array}%
\right]
\end{eqnarray*}%
On the other hand%
\begin{eqnarray*}
&&h\alpha _{U,V}\left( u\otimes \left( v\otimes k\right) \right) =\omega
\left( u_{-1}\otimes v_{-1}\otimes k_{1}\right) h\left[ \left( u_{0}\otimes
v_{0}\right) \otimes k_{2}\right] \\
&=&\left[ 
\begin{array}{c}
\omega \left( u_{-1}\otimes v_{-1}\otimes k_{1}\right) \omega (h_{1}\otimes
\left( u_{0}\otimes v_{0}\right) _{-1}\otimes k_{2}) \\ 
\omega ^{-1}((h_{2}\vartriangleright \left( u_{0}\otimes v_{0}\right)
_{0})_{-1}\otimes h_{3}\otimes k_{3})(h_{2}\vartriangleright \left(
u_{0}\otimes v_{0}\right) _{0})_{0}\otimes h_{4}k_{4}%
\end{array}%
\right] \\
&=&\left[ 
\begin{array}{c}
\omega \left( u_{-2}\otimes v_{-2}\otimes k_{1}\right) \omega (h_{1}\otimes
u_{-1}v_{-1}\otimes k_{2}) \\ 
\omega ^{-1}((h_{2}\vartriangleright \left( u_{0}\otimes v_{0}\right)
)_{-1}\otimes h_{3}\otimes k_{3})(h_{2}\vartriangleright \left( u_{0}\otimes
v_{0}\right) )_{0}\otimes h_{4}k_{4}%
\end{array}%
\right] \\
&\overset{(\ref{form:YDtens})}{=}&\left[ 
\begin{array}{c}
\omega \left( u_{-2}\otimes v_{-2}\otimes k_{1}\right) \omega (h_{1}\otimes
u_{-1}v_{-1}\otimes k_{2}) \\ 
\omega \left( \left( h_{2}\right) _{1}\otimes \left( u_{0}\right)
_{-1}\otimes \left( v_{0}\right) _{-2}\right) \omega ^{-1}\left( \left(
\left( h_{2}\right) _{2}\vartriangleright \left( u_{0}\right) _{0}\right)
_{-2}\otimes \left( h_{2}\right) _{3}\otimes \left( v_{0}\right) _{-1}\right)
\\ 
\omega \left( \left( \left( h_{2}\right) _{2}\vartriangleright \left(
u_{0}\right) _{0}\right) _{-1}\otimes \left( \left( h_{2}\right)
_{4}\vartriangleright \left( v_{0}\right) _{0}\right) _{-1}\otimes \left(
h_{2}\right) _{5}\right) \\ 
\omega ^{-1}([(\left( h_{2}\right) _{2}\vartriangleright \left( u_{0}\right)
_{0})_{0}\otimes \left( \left( h_{2}\right) _{4}\vartriangleright \left(
v_{0}\right) _{0}\right) _{0}]_{-1}\otimes h_{3}\otimes k_{3}) \\ 
\left[ \left( \left( h_{2}\right) _{2}\vartriangleright \left( u_{0}\right)
_{0}\right) _{0}\otimes \left( \left( h_{2}\right) _{4}\vartriangleright
\left( v_{0}\right) _{0}\right) _{0}\right] _{0}\otimes h_{4}k_{4}%
\end{array}%
\right] \\
&=&\left[ 
\begin{array}{c}
\omega \left( u_{-3}\otimes v_{-4}\otimes k_{1}\right) \omega (h_{1}\otimes
u_{-2}v_{-3}\otimes k_{2})\omega \left( h_{2}\otimes u_{-1}\otimes
v_{-2}\right) \\ 
\omega ^{-1}\left( \left( h_{3}\vartriangleright u_{0}\right) _{-2}\otimes
h_{4}\otimes v_{-1}\right) \omega \left( \left( h_{3}\vartriangleright
u_{0}\right) _{-1}\otimes \left( h_{5}\vartriangleright v_{0}\right)
_{-1}\otimes h_{6}\right) \\ 
\omega ^{-1}((\left( h_{3}\vartriangleright u_{0}\right) _{0}\otimes \left(
h_{5}\vartriangleright v_{0}\right) _{0})_{-1}\otimes h_{7}\otimes k_{3}) \\ 
(\left( h_{3}\vartriangleright u_{0}\right) _{0}\otimes \left(
h_{5}\vartriangleright v_{0}\right) _{0})_{0}\otimes h_{8}k_{4}%
\end{array}%
\right] \\
&\overset{(\ref{eq:3-cocycle})}{=}&\left[ 
\begin{array}{c}
\omega (h_{1}\otimes u_{-2}\otimes v_{-3}k_{1})\omega (h_{2}u_{-1}\otimes
v_{-2}\otimes k_{2})\omega ^{-1}\left( \left( h_{3}\vartriangleright
u_{0}\right) _{-2}\otimes h_{4}\otimes v_{-1}\right) \\ 
\omega \left( \left( h_{3}\vartriangleright u_{0}\right) _{-1}\otimes \left(
h_{5}\vartriangleright v_{0}\right) _{-1}\otimes h_{6}\right) \omega
^{-1}((\left( h_{3}\vartriangleright u_{0}\right) _{0}\otimes \left(
h_{5}\vartriangleright v_{0}\right) _{0})_{-1}\otimes h_{7}\otimes k_{3}) \\ 
(\left( h_{3}\vartriangleright u_{0}\right) _{0}\otimes \left(
h_{5}\vartriangleright v_{0}\right) _{0})_{0}\otimes h_{8}k_{4}%
\end{array}%
\right] \\
&=&\left[ 
\begin{array}{c}
\omega (h_{1}\otimes u_{-2}\otimes v_{-3}k_{1})\omega (h_{2}u_{-1}\otimes
v_{-2}\otimes k_{2})\omega ^{-1}\left( \left( h_{3}\vartriangleright
u_{0}\right) _{-3}\otimes h_{4}\otimes v_{-1}\right) \\ 
\omega \left( \left( h_{3}\vartriangleright u_{0}\right) _{-2}\otimes \left(
h_{5}\vartriangleright v_{0}\right) _{-2}\otimes h_{6}\right) \omega
^{-1}(\left( h_{3}\vartriangleright u_{0}\right) _{-1}\left(
h_{5}\vartriangleright v_{0}\right) _{-1}\otimes h_{7}\otimes k_{3}) \\ 
(\left( h_{3}\vartriangleright u_{0}\right) _{0}\otimes \left(
h_{5}\vartriangleright v_{0}\right) _{0})\otimes h_{8}k_{4}%
\end{array}%
\right] \\
&\overset{(\ref{eq:3-cocycle})}{=}&\left[ 
\begin{array}{c}
\omega (h_{1}\otimes u_{-2}\otimes v_{-3}k_{1})\omega (h_{2}u_{-1}\otimes
v_{-2}\otimes k_{2})\omega ^{-1}\left( \left( h_{3}\vartriangleright
u_{0}\right) _{-3}\otimes h_{4}\otimes v_{-1}\right) \\ 
\omega ^{-1}(\left( h_{3}\vartriangleright u_{0}\right) _{-2}\otimes \left(
h_{5}\vartriangleright v_{0}\right) _{-3}h_{6}\otimes k_{3})\omega
^{-1}(\left( h_{5}\vartriangleright v_{0}\right) _{-2}\otimes h_{7}\otimes
k_{4}) \\ 
\omega (\left( h_{3}\vartriangleright u_{0}\right) _{-1}\otimes \left(
h_{5}\vartriangleright v_{0}\right) _{-1}\otimes h_{8}k_{5})(\left(
h_{3}\vartriangleright u_{0}\right) _{0}\otimes \left(
h_{5}\vartriangleright v_{0}\right) _{0})\otimes h_{9}k_{6}%
\end{array}%
\right] \\
&\overset{(\ref{Comp YD})}{=}&\left[ 
\begin{array}{c}
\omega (h_{1}\otimes u_{-2}\otimes v_{-4}k_{1})\omega (h_{2}u_{-1}\otimes
v_{-3}\otimes k_{2})\omega ^{-1}\left( \left( h_{3}\vartriangleright
u_{0}\right) _{-3}\otimes h_{4}\otimes v_{-2}\right) \\ 
\omega ^{-1}(\left( h_{3}\vartriangleright u_{0}\right) _{-2}\otimes
h_{5}v_{-1}\otimes k_{3})\omega ^{-1}(\left( h_{6}\vartriangleright
v_{0}\right) _{-2}\otimes h_{7}\otimes k_{4}) \\ 
\omega (\left( h_{3}\vartriangleright u_{0}\right) _{-1}\otimes \left(
h_{6}\vartriangleright v_{0}\right) _{-1}\otimes h_{8}k_{5})(\left(
h_{3}\vartriangleright u_{0}\right) _{0}\otimes \left(
h_{6}\vartriangleright v_{0}\right) _{0})\otimes h_{9}k_{6}%
\end{array}%
\right] \\
&\overset{(\ref{eq:3-cocycle})}{=}&\left[ 
\begin{array}{c}
\omega (h_{1}\otimes u_{-2}\otimes v_{-5}k_{1})\omega (h_{2}u_{-1}\otimes
v_{-4}\otimes k_{2})\omega ^{-1}(\left( h_{3}\vartriangleright u_{0}\right)
_{-3}h_{4}\otimes v_{-3}\otimes k_{3}) \\ 
\omega ^{-1}(\left( h_{3}\vartriangleright u_{0}\right) _{-2}\otimes
h_{5}\otimes v_{-2}k_{4})\omega (h_{6}\otimes v_{-1}\otimes k_{5})\omega
^{-1}(\left( h_{7}\vartriangleright v_{0}\right) _{-2}\otimes h_{8}\otimes
k_{6}) \\ 
\omega (\left( h_{3}\vartriangleright u_{0}\right) _{-1}\otimes \left(
h_{7}\vartriangleright v_{0}\right) _{-1}\otimes h_{9}k_{7})(\left(
h_{3}\vartriangleright u_{0}\right) _{0}\otimes \left(
h_{7}\vartriangleright v_{0}\right) _{0})\otimes h_{10}k_{8}%
\end{array}%
\right] \\
&\overset{(\ref{Comp YD})}{=}&\left[ 
\begin{array}{c}
\omega (h_{1}\otimes u_{-3}\otimes v_{-5}k_{1})\omega (h_{2}u_{-2}\otimes
v_{-4}\otimes k_{2})\omega ^{-1}(h_{3}u_{-1}\otimes v_{-3}\otimes k_{3}) \\ 
\omega ^{-1}(\left( h_{4}\vartriangleright u_{0}\right) _{-2}\otimes
h_{5}\otimes v_{-2}k_{4})\omega (h_{6}\otimes v_{-1}\otimes k_{5})\omega
^{-1}(\left( h_{7}\vartriangleright v_{0}\right) _{-2}\otimes h_{8}\otimes
k_{6}) \\ 
\omega (\left( h_{4}\vartriangleright u_{0}\right) _{-1}\otimes \left(
h_{7}\vartriangleright v_{0}\right) _{-1}\otimes h_{9}k_{7})(\left(
h_{4}\vartriangleright u_{0}\right) _{0}\otimes \left(
h_{7}\vartriangleright v_{0}\right) _{0})\otimes h_{10}k_{8}%
\end{array}%
\right] \\
&=&\left[ 
\begin{array}{c}
\omega (h_{1}\otimes u_{-1}\otimes v_{-3}k_{1})\omega ^{-1}(\left(
h_{2}\vartriangleright u_{0}\right) _{-2}\otimes h_{3}\otimes
v_{-2}k_{2})\omega (h_{4}\otimes v_{-1}\otimes k_{3}) \\ 
\omega ^{-1}(\left( h_{5}\vartriangleright v_{0}\right) _{-2}\otimes
h_{6}\otimes k_{4})\omega (\left( h_{2}\vartriangleright u_{0}\right)
_{-1}\otimes \left( h_{5}\vartriangleright v_{0}\right) _{-1}\otimes
h_{7}k_{5}) \\ 
(\left( h_{2}\vartriangleright u_{0}\right) _{0}\otimes \left(
h_{5}\vartriangleright v_{0}\right) _{0})\otimes h_{8}k_{6}%
\end{array}%
\right] .
\end{eqnarray*}

Summing up, we have proved that $\alpha _{U,V}:U\otimes F\left( V\right)
\rightarrow F\left( U\otimes V\right) $ is an isomorphism in ${_{H}^{H}%
\mathfrak{M}_{H}^{H}.}$ Now, since $\alpha _{U,V}=\left( {^{H}}%
a_{U,V,H}\right) ^{-1},$ we have that $\alpha _{U,V}$ is natural in $U,V$
for all morphisms in ${^{H}}\mathfrak{M}$ (in particular in ${_{H}^{H}%
\mathcal{YD}}$).
\end{proof}

\begin{lemma}
\label{Restr F} Let $(H,m,u,\Delta ,\varepsilon ,\omega )$ be a dual
quasi-bialgebra. The functor $F:\left( -\right) \otimes H:{^{H}\mathfrak{M}}%
\rightarrow {^{H}\mathfrak{M}_{H}^{H}}$ of \ref{claim: adjunctios}. induces a functor $F:{_{H}^{H}%
\mathcal{YD}}\rightarrow {_{H}^{H}\mathfrak{M}_{H}^{H}}$. Explicitly $%
F\left( M\right) \in {_{H}^{H}\mathfrak{M}_{H}^{H}}$ with the following
structures, for all $m\in M,h,l\in H$,%
\begin{eqnarray}
\mu _{M\otimes H}^{l}\left[ l\otimes (m\otimes h)\right] &:&=l\cdot
(m\otimes h):=\omega (l_{1}\otimes m_{-1}\otimes
h_{1})(l_{2}\vartriangleright m_{0}\otimes l_{3})\cdot h_{2}
\label{strutt mod left} \\
&=&\omega (l_{1}\otimes m_{-1}\otimes h_{1})\omega
^{-1}((l_{2}\vartriangleright m_{0})_{-1}\otimes l_{3}\otimes
h_{2})(l_{2}\vartriangleright m_{0})_{0}\otimes l_{4}h_{3}  \notag \\
\mu _{M\otimes H}^{r}\left[ (m\otimes h)\otimes l\right] &:&=(m\otimes
h)\cdot l:=\omega ^{-1}(m_{-1}\otimes h_{1}\otimes l_{1})m_{0}\otimes
h_{2}l_{2},  \label{strutt mod right} \\
\rho _{M\otimes H}^{l}\left( m\otimes h\right) &:&=m_{-1}h_{1}\otimes
(m_{0}\otimes h_{2}),  \notag \\
\rho _{M\otimes H}^{r}\left( m\otimes h\right) &:&=(m\otimes h_{1})\otimes
h_{2},  \notag
\end{eqnarray}
\end{lemma}

\begin{proof}
Let $M\in $ ${_{H}^{H}\mathcal{YD}}$. Consider $H\otimes M$ as an object in $%
{^{H}\mathfrak{M}^{H}}$ via%
\begin{eqnarray*}
{\rho }_{H\otimes M}^{r}\left( h\otimes m\right) &:&=\left( h_{1}\otimes
m\right) \otimes h_{2}, \\
{\rho }_{H\otimes M}^{l}\left( h\otimes m\right) &:&=h_{1}m_{-1}\otimes
\left( h_{2}\otimes m_{0}\right) .
\end{eqnarray*}%
Since $(H\otimes M,\rho^l_{H\otimes M})\in{^H\mathfrak{M}}$, by Lemma \ref%
{lem:a4H}, the map $\alpha _{H,M}:H\otimes F\left( M\right) \rightarrow
F\left( H\otimes M\right) $ is a natural isomorphism in ${^{H}\mathfrak{M}%
_{H}^{H}}$, where $H\otimes F\left( M\right) $ has the structure described
in Lemma \ref{lem: xiScha} for $"M"=F\left( M\right) $, i.e. for all $h\in
H,x\in M\otimes H$%
\begin{eqnarray*}
{\rho }_{H\otimes F\left( M\right) }^{l}\left( h\otimes x\right)
&=&h_{1}x_{-1}\otimes (h_{2}\otimes x_{0}), \\
\rho _{H\otimes F\left( M\right) }^{r}\left( h\otimes x\right) &=&(h\otimes
x_{0})\otimes x_{1}, \\
\mu _{H\otimes F\left( M\right) }^{r}(\left( h\otimes x\right) \otimes k)
&=&\omega ^{-1}(h_{1}\otimes x_{-1}\otimes k_{1})h_{2}\otimes x_{0}k_{2}.
\end{eqnarray*}%
In particular, we have%
\begin{equation*}
\rho _{T\left( H\otimes M\right) }^{l}\alpha _{H,M}=\rho _{F\left( H\otimes
M\right) }^{l}\alpha _{H,M}=\left( H\otimes \alpha _{H,M}\right) \rho
_{H\otimes F\left( M\right) }^{l}
\end{equation*}%
where $T:{^{H}\mathfrak{M}^{H}}\rightarrow {^{H}\mathfrak{M}_{H}^{H}}%
,T\left( {^{\bullet }}M^{\bullet }\right) :={^{\bullet }}M^{\bullet }\otimes 
{^{\bullet }}H_{\bullet }^{\bullet }$ is the functor of \ref{claim:
adjunctios}. Now, consider on $H\otimes F\left( M\right) $ the following new
structures 
\begin{eqnarray*}
\widetilde{{\rho }}_{H\otimes F\left( M\right) }^{l}\left( h\otimes x\right)
&=&h_{1}x_{-1}\otimes (h_{2}\otimes x_{0}), \\
\widetilde{{\rho }}_{H\otimes F\left( M\right) }^{r}\left( h\otimes x\right)
&=&(h_{1}\otimes x_{0})\otimes h_{2}x_{1}, \\
\widetilde{\mu }_{H\otimes F\left( M\right) }^{r}(\left( h\otimes x\right)
\otimes k) &=&\omega ^{-1}(h_{1}\otimes x_{-1}\otimes k_{1})h_{2}\otimes
x_{0}k_{2}\omega (h_{3}\otimes x_{1}\otimes k_{3}),
\end{eqnarray*}%
note that $\widetilde{\mu }_{H\otimes F\left( M\right) }^{r}=\left( H\otimes
\mu _{F\left( M\right) }^{r}\right) \circ {^{H}}a{_{H,F\left( M\right)
,H}^{H}.}$ Moreover one gets 
\begin{equation*}
{\rho }_{T\left( H\otimes M\right) }^{r}\alpha _{H,M}=\left( \alpha
_{H,M}\otimes H\right) \widetilde{{\rho }}_{H\otimes F\left( M\right) }^{r}
\end{equation*}
and 
\begin{eqnarray*}
&&\mu _{T\left( H\otimes M\right) }^{r}\left( \alpha _{H,M}\otimes H\right)
\left( \left[ h\otimes \left( m\otimes k\right) \right] \otimes l\right) \\
&=& \omega \left( h_{1}\otimes m_{-2}\otimes k_{1}\right) \omega ^{-1}\left(
h_{2}m_{-1}\otimes k_{2}\otimes l_{1}\right) \left( h_{3}\otimes
m_{0}\right) \otimes k_{3}l_{2}\omega \left( h_{4}\otimes k_{4}\otimes
l_{3}\right) \\
&=&\mu _{F\left( H\otimes M\right) }^{r}\left( \alpha _{H,M}\otimes H\right) 
\left[ \left( h_{1}\otimes (m\otimes k_{1})_{0}\right) \otimes l_{1}\right]
\omega (h_{2}\otimes (m\otimes k_{1})_{1}\otimes l_{2}) \\
&=&\alpha _{H,M}\mu _{H\otimes F\left( M\right) }^{r}\left[ \left(
h_{1}\otimes (m\otimes k_{1})_{0}\right) \otimes l_{1}\right] \omega
(h_{2}\otimes (m\otimes k_{1})_{1}\otimes l_{2}) \\
&=&\omega ^{-1}(h_{1}\otimes \left( m\otimes k\right) _{-1}\otimes
l_{1})\alpha _{H,M}\left[ h_{2}\otimes (m\otimes k_{1})_{0}\cdot l_{2}\right]
\omega (h_{3}\otimes (m\otimes k_{1})_{1}\otimes l_{3}) \\
&=&\alpha _{H,M}\widetilde{\mu }_{H\otimes F\left( M\right) }^{r}\left( %
\left[ h\otimes \left( m\otimes k\right) \right] \otimes l\right)
\end{eqnarray*}%
We have so proved that $\alpha _{H,M}$ can be regarded as a morphism in ${%
^{H}\mathfrak{M}_{H}^{H}}$ from $H\otimes F\left( M\right) $ to $T\left(
H\otimes M\right) $, where $H\otimes F\left( M\right) $ has structures $%
\widetilde{{\rho }}_{H\otimes F\left( M\right) }^{l},\widetilde{{\rho }}%
_{H\otimes F\left( M\right) }^{r}$ and $\widetilde{\mu }_{H\otimes F\left(
M\right) }^{r}$.

Consider the map $c_{H,M}:H\otimes M\rightarrow M\otimes H,$ as in (\ref%
{braiding YD}) i.e. $c_{H,M}\left( h\otimes m\right) =\left(
h_{1}\vartriangleright m\right) \otimes h_{2}.$ 

Using \eqref{Comp YD} one can prove that $c_{H,M}:H\otimes M\rightarrow
F\left( M\right) $ is a morphism in ${{^{H}\mathfrak{M}^{H}}}$ (where $%
H\otimes M$ is regarded as an object in $^H\mathfrak{M}^H$ as at the
beginning of this proof) whence $T\left( c_{H,M}\right) $ is in ${{^{H}%
\mathfrak{M}_{H}^{H}}}$ (note that we do not know that $H$ is in ${_{H}^{H}%
\mathcal{YD}}$ so that we cannot say that $c_{H,M}$ is in ${_{H}^{H}%
\mathcal{YD}}$ directly).

Now, consider the morphism $\mu _{F\left( M\right) }^{r}:F\left( M\right)
\otimes H\rightarrow F\left( M\right) $. Clearly $\mu _{F\left( M\right)
}^{r}$ can be regarded as a morphism in ${{^{H}\mathfrak{M}_{H}^{H}}}$ from $%
TF\left( M\right) $ to $F\left( M\right) $. Summing up we can consider in ${{%
^{H}\mathfrak{M}_{H}^{H}}}$ the composition%
\begin{equation*}
\mu _{M\otimes H}^{l}:=\left( H\otimes F\left( M\right) \overset{\alpha
_{H,M}}{\longrightarrow }T\left( H\otimes M\right) \overset{T\left(
c_{H,M}\right) }{\longrightarrow }TF\left( M\right) \overset{\mu _{F\left(
M\right) }^{r}}{\longrightarrow }F\left( M\right) \right)
\end{equation*}%
where $H\otimes F\left( M\right) $ has structures $\widetilde{{\rho }}%
_{H\otimes F\left( M\right) }^{l},\widetilde{{\rho }}_{H\otimes F\left(
M\right) }^{r}$ and $\widetilde{\mu }_{H\otimes F\left( M\right) }^{r}$.
Thus $\mu _{M\otimes H}^{l}$ is a morphism in ${{^{H}\mathfrak{M}^{H}}}$
such that%
\begin{equation}
\mu _{M\otimes H}^{r}\circ \left( \mu _{M\otimes H}^{l}\otimes H\right) =\mu
_{M\otimes H}^{l}\circ \left( H\otimes \mu _{M\otimes H}^{r}\right) \circ {{%
^{H}}}a{{_{H,M\otimes H,H}^{H}.}}  \label{eq bimodule}
\end{equation}%
It remains to prove that $\left( M\otimes H,\mu _{M\otimes H}^{l}\right) $
is a left $H$-module in ${{^{H}\mathfrak{M}^{H}}}$. Let us prove that 
\begin{equation}
\mu _{M\otimes H}^{l}\circ \left( H\otimes \mu _{M\otimes H}^{l}\right)
\circ {{^{H}}}a{{_{H,H,M\otimes H}^{H}}}=\mu _{M\otimes H}^{l}\circ \left[
m\otimes \left( M\otimes H\right) \right] .  \label{star}
\end{equation}

First note that, using \eqref{eq bimodule} and \eqref{eq:3-cocycle} one
checks that 
\begin{eqnarray*}
&&\mu _{M\otimes H}^{l}\left( H\otimes \mu _{M\otimes H}^{l}\right) {{^{H}}}a%
{{_{H,H,M\otimes H}^{H}}}\left[ \left( h\otimes k\right) \otimes (m\otimes l)%
\right] \\
&=&\omega (h_{1}k_{1}\otimes m_{-1}\otimes l_{1})\left[ \mu _{M\otimes
H}^{l}\left( H\otimes \mu _{M\otimes H}^{l}\right) {{^{H}}}a{{_{H,H,M\otimes
H}^{H}}}\left[ \left( h_{2}\otimes k_{2}\right) \otimes (m_{0}\otimes 1_{H})%
\right] \right] l_{2}
\end{eqnarray*}%
and 
\begin{eqnarray*}
&&\mu _{M\otimes H}^{l}\left[ m\otimes \left( M\otimes H\right) \right] %
\left[ \left( h\otimes k\right) \otimes (m\otimes l)\right] \\
&\overset{(\ref{eq bimodule})}{=}&\omega (h_{1}k_{1}\otimes m_{-1}\otimes
l_{1})\left[ \left( h_{2}k_{2}\right) (m_{0}\otimes 1_{H})\right] l_{2} \\
&=&\omega (h_{1}k_{1}\otimes m_{-1}\otimes l_{1})\mu _{M\otimes H}^{l}\left[
m\otimes \left( M\otimes H\right) \right] \left[ \left( h_{2}\otimes
k_{2}\right) \otimes (m_{0}\otimes 1_{H})\right] l_{2}
\end{eqnarray*}%
Thus we have to prove that 
\eqref{star} holds on elements of the form $\left( h\otimes k\right) \otimes
(m\otimes 1_{H}).$

We have%
\begin{eqnarray*}
&&\mu _{M\otimes H}^{l}\left( H\otimes \mu _{M\otimes H}^{l}\right) {{^{H}}}a%
{{_{H,H,M\otimes H}^{H}}}\left[ \left( h\otimes k\right) \otimes (m\otimes
1_{H})\right] \\
&=&\left[ 
\begin{array}{c}
\omega ^{-1}(h_{1}\otimes k_{1}\otimes m_{-1})\omega (h_{2}\otimes
(k_{2}\vartriangleright m_{0})_{-1}\otimes k_{3}) \\ 
\lbrack h_{3}\vartriangleright (k_{2}\vartriangleright m_{0})_{0}\otimes
h_{4}]k_{4}%
\end{array}%
\right] \\
&=&\left[ 
\begin{array}{c}
\omega ^{-1}(h_{1}\otimes k_{1}\otimes m_{-1})\omega (h_{2}\otimes
(k_{2}\vartriangleright m_{0})_{-1}\otimes k_{3}) \\ 
\omega ^{-1}((h_{3}\vartriangleright (k_{2}\vartriangleright
m_{0})_{0})_{-1}\otimes h_{4}\otimes k_{4})(h_{3}\vartriangleright
(k_{2}\vartriangleright m_{0})_{0})_{0}\otimes h_{5}k_{5}%
\end{array}%
\right] \\
&\overset{\eqref{ass YD}}{=}&h_{1}k_{1}\vartriangleright m\otimes
h_{2}k_{2}=\left( hk\right) (m\otimes 1_{H})=\mu _{M\otimes H}^{l}\left[
m\otimes \left( M\otimes H\right) \right] \left[ \left( h\otimes k\right)
\otimes (m\otimes 1_{H})\right] .
\end{eqnarray*}

Finally one checks that, for each morphism $f:M\rightarrow N$ in $_{H}^{H}%
\mathcal{YD}$, we have $F(f):=f\otimes H\in {_{H}^{H}\mathfrak{M}_{H}^{H}}.$

\end{proof}

\begin{lemma}
Let $(H,m,u,\Delta ,\varepsilon ,\omega ,S)$ be a dual quasi-bialgebra with
a preantipode. The functor $G:\left( -\right) ^{coH}:{^{H}\mathfrak{M}%
_{H}^{H}}\rightarrow {^{H}\mathfrak{M}}$ of \ref{claim: adjunctios} induces a functor $G:{_{H}^{H}%
\mathfrak{M}_{H}^{H}}\rightarrow {_{H}^{H}\mathcal{YD}}$. Explicitly $%
G\left( M\right) \in {_{H}^{H}\mathcal{YD}}$ with the following structures,
for all $m\in M^{coH},h\in H$,%
\begin{eqnarray*}
\rho _{M^{coH}}^{l}\left( m\right) &:&=\rho _{M}^{l}(m), \\
\mu _{M^{coH}}^{l}\left( h\otimes m\right) &:&=h\vartriangleright m:=\tau
(hm)=\omega \lbrack h_{1}m_{-1}\otimes S(h_{3})_{1}\otimes h_{4}]\left(
h_{2}m_{0}\right) S(h_{3})_{2}.
\end{eqnarray*}
\end{lemma}

\begin{proof}
Let $M\in {_{H}^{H}\mathfrak{M}_{H}^{H}}$. We already know that $G(M)\in {%
^{H}\mathfrak{M.}}$ In order to prove that $G(M)$ is in ${_{H}^{H}\mathcal{YD%
}}$, we consider the canonical isomorphism $\epsilon _{M}:FG\left( M\right)
\rightarrow M$ of Remark \ref{adjoint}. A priori, this is a morphism in ${%
^{H}\mathfrak{M}_{H}^{H}.}$ Since $M$ is in ${_{H}^{H}\mathfrak{M}_{H}^{H}}$%
, we can endow $FG\left( M\right) $ with a left $H$-module structure as
follows%
\begin{eqnarray*}
l\cdot (m\otimes h) &:=&\epsilon _{M}^{-1}\left( l\epsilon _{M}(m\otimes
h)\right) =\epsilon _{M}^{-1}\left( l(mh)\right) =\tau \left[
l_{1}(m_{0}h_{1})\right] \otimes l_{2}(m_{1}h_{2}) \\
&=&\tau \left[ l_{1}(mh_{1})\right] \otimes
l_{2}h_{2}=l_{1}\vartriangleright (mh_{1})\otimes l_{2}h_{2}
\end{eqnarray*}%
so that 
\begin{equation}
l\cdot (m\otimes h)=l_{1}\vartriangleright (mh_{1})\otimes l_{2}h_{2},\text{
for all }m\in M^{\mathrm{co}H},h\in H.  \label{form: RYD}
\end{equation}%
By associativity we have%
\begin{equation*}
\left( lk\right) \cdot (m\otimes h)=\omega ^{-1}\left( l_{1}\otimes
k_{1}\otimes m_{-1}h_{1}\right) l_{2}\left( k_{2}\left( m_{0}\otimes
h_{2}\right) \right) \omega \left( l_{3}\otimes k_{3}\otimes h_{3}\right)
\end{equation*}%
i.e., for $h=1_{H},$%
\begin{equation*}
\left( lk\right) \cdot (m\otimes 1_{H})=\omega ^{-1}\left( l_{1}\otimes
k_{1}\otimes m_{-1}\right) l_{2}\left( k_{2}\left( m_{0}\otimes 1_{H}\right)
\right) .
\end{equation*}%
The first term is%
\begin{equation*}
\left( lk\right) \cdot (m\otimes 1_{H})\overset{(\ref{form: RYD})}{=}\left(
l_{1}k_{1}\right) \vartriangleright m\otimes l_{2}k_{2}.
\end{equation*}%
The second term is%
\begin{eqnarray*}
&&\omega ^{-1}\left( l_{1}\otimes k_{1}\otimes m_{-1}\right) l_{2}\left(
k_{2}\left( m_{0}\otimes 1_{H}\right) \right) \overset{(\ref{form: RYD})}{=}%
\omega ^{-1}\left( l_{1}\otimes k_{1}\otimes m_{-1}\right) l_{2}\left(
k_{2}\vartriangleright m_{0}\otimes k_{3}\right) \\
&\overset{(\ref{form: RYD})}{=}&\omega ^{-1}\left( l_{1}\otimes k_{1}\otimes
m_{-1}\right) l_{2}\vartriangleright (\left( k_{2}\vartriangleright
m_{0}\right) k_{3})\otimes l_{3}k_{4} \\
&=&\omega ^{-1}\left( l_{1}\otimes k_{1}\otimes m_{-1}\right) \tau \left[
l_{2}(\left( k_{2}\vartriangleright m_{0}\right) k_{3})\right] \otimes
l_{3}k_{4} \\
&=&\omega ^{-1}\left( l_{1}\otimes k_{1}\otimes m_{-1}\right) \omega \left(
l_{2}\otimes \left( k_{2}\vartriangleright m_{0}\right) _{-1}\otimes
k_{3}\right) \tau \left[ \left( l_{3}\left( k_{2}\vartriangleright
m_{0}\right) _{0}\right) k_{4}\right] \otimes l_{4}k_{5} \\
&=&\left[ 
\begin{array}{c}
\omega ^{-1}\left( l_{1}\otimes k_{1}\otimes m_{-1}\right) \omega \left(
l_{2}\otimes \left( k_{2}\vartriangleright m_{0}\right) _{-1}\otimes
k_{3}\right) \\ 
\omega ^{-1}\left( (l_{3}\vartriangleright (k_{2}\vartriangleright
m_{0})_{0}\right) _{-1}\otimes l_{4}\otimes k_{4})\left(
l_{3}\vartriangleright \left( k_{2}\vartriangleright m_{0}\right)
_{0}\right) _{0}\otimes l_{5}k_{5}%
\end{array}%
\right]
\end{eqnarray*}%
Hence, we obtain%
\begin{equation*}
\left( l_{1}k_{1}\right) \vartriangleright m\otimes l_{2}k_{2}=\left[ 
\begin{array}{c}
\omega ^{-1}\left( l_{1}\otimes k_{1}\otimes m_{-1}\right) \omega \left(
l_{2}\otimes \left( k_{2}\vartriangleright m_{0}\right) _{-1}\otimes
k_{3}\right) \\ 
\omega ^{-1}\left( (l_{3}\vartriangleright (k_{2}\vartriangleright
m_{0})_{0}\right) _{-1}\otimes l_{4}\otimes k_{4})\left(
l_{3}\vartriangleright \left( k_{2}\vartriangleright m_{0}\right)
_{0}\right) _{0}\otimes l_{5}k_{5}%
\end{array}%
\right] .
\end{equation*}%
By applying $M\otimes \varepsilon _{H}$ on both sides, we arrive at %
\eqref{ass YD}. 
Moreover, by \eqref{Tau mh simple}, we have $1_{H}\vartriangleright m=\tau
\left( m\right) =m$ and 
\begin{eqnarray*}
&&\left( h_{1}\vartriangleright m\right) _{-1}h_{2}\otimes \left(
h_{1}\vartriangleright m\right) _{0}=\tau \left( h_{1}m\right)
_{-1}h_{2}\otimes \tau \left( h_{1}m\right) _{0}=\tau \left( \left(
hm\right) _{0}\right) _{-1}\left( hm\right) _{1}\otimes \tau \left( \left(
hm\right) _{0}\right) _{0} \\
\overset{(\ref{col sx eps})}{=} &&\left( hm\right) _{-1}\otimes \tau \left(
\left( hm\right) _{0}\right) =h_{1}m_{-1}\otimes \tau \left(
h_{2}m_{0}\right) =h_{1}m_{-1}\otimes \left( h_{2}\vartriangleright
m_{0}\right) .
\end{eqnarray*}%
We have so proved that $G(M)\in {{_{H}^{H}\mathcal{YD}}}$. Now it is easy to
verify that for every $g:M\rightarrow N\in {_{H}^{H}\mathfrak{M}_{H}^{H}},$
we have that $G(g):M^{coH}\rightarrow N^{coH}\in {_{H}^{H}\mathcal{YD}}$.%
\end{proof}

\begin{proposition}
\label{restr equi} Let $(H,m,u,\Delta ,\varepsilon ,\omega ,S)$ be a dual
quasi-bialgebra with a preantipode. $(F,G)$ is an equivalence between $%
_{H}^{H}\mathfrak{M}_{H}^{H}$ and $_{H}^{H}\mathcal{YD}$, i.e. the morphisms 
$\varepsilon _{M}$ and $\eta _{N}$ of Remark \ref{adjoint} are in $_{H}^{H}%
\mathfrak{M}_{H}^{H}$ and in ${{_{H}^{H}\mathfrak{{\mathcal{YD}}}}}$
respectively, for each $M\in {{_{H}^{H}\mathfrak{M}_{H}^{H}}},N\in {{_{H}^{H}%
\mathfrak{{\mathcal{YD}}}}}$.
\end{proposition}

\begin{proof}
We already know that $\varepsilon _{M}\in $ $^{H}\mathfrak{M}_{H}^{H}.$ Let
us check that $\varepsilon _{M}$ is left $H$-linear.

\begin{gather*}
\varepsilon _{M}\mu _{M^{coH}\otimes H}(h\otimes m\otimes k)=\varepsilon
_{M}(h\cdot (m\otimes k))\overset{(\ref{strutt mod left})}{=}\varepsilon
_{M}[\omega (h_{1}\otimes m_{-1}\otimes k_{1})(h_{2}\vartriangleright
m_{0}\otimes h_{3})k_{2}] \\
\overset{\varepsilon _{M}\text{ right lin}}{=}\omega (h_{1}\otimes
m_{-1}\otimes k_{1})\varepsilon _{M}[(h_{2}\vartriangleright m_{0}\otimes
h_{3})]k_{2}=\omega (h_{1}\otimes m_{-1}\otimes
k_{1})[(h_{2}\vartriangleright m_{0})h_{3}]k_{2} \\
=\omega (h_{1}\otimes m_{-1}\otimes k_{1})[\tau
(h_{2}m_{0})h_{3}]k_{2}=\omega (h_{1}\otimes m_{-1}\otimes k_{1})[\tau
(h_{2}m_{0})(h_{3}m_{1})]k_{2} \\
\overset{(\ref{inv eps})}{=}\omega (h_{1}\otimes m_{-1}\otimes
k_{1})(h_{2}m_{0})k_{2}\overset{\eqref{eq:quasi-associativity}}{=}%
h_{1}(m_{0}k_{1})\omega (h_{2}\otimes m_{1}\otimes k_{2})=h(mk)=\mu
_{M}(H\otimes \varepsilon _{M})(h\otimes m\otimes k).
\end{gather*}%
Now let us check the compatibility of $\eta $ with $\vartriangleright .$ For 
$N\in {_{H}^{H}\mathcal{YD}}$ and $n\in N,$ 
\begin{gather*}
\left[ \mu _{\left( N\otimes H\right) ^{coH}}^{l}\circ {^{H}}a_{H,N,H}\circ
(H\otimes \eta _{N})\right] (h\otimes n)=\left[ \mu _{\left( N\otimes
H\right) ^{coH}}^{l}\circ {^{H}}a_{H,N,H}\right] (h\otimes \left( n\otimes
1_{H}\right) ) \\
=\omega ^{-1}(h_{1}\otimes n_{-1}\otimes 1_{H})\mu _{\left( N\otimes
H\right) ^{coH}}^{l}(h_{2}\otimes \left( n_{0}\otimes 1_{H}\right) )=\mu
_{\left( N\otimes H\right) ^{coH}}^{l}(h\otimes \left( n\otimes 1_{H}\right)
)=\tau (h\left( n\otimes 1_{H}\right) ) \\
\overset{(\ref{strutt mod left})}{=}\tau (h_{1}\vartriangleright n\otimes
h_{2})\overset{(\ref{strutt mod right})}{=}\tau (\left(
h_{1}\vartriangleright n\otimes 1_{H}\right) h_{2})=\tau (\eta
_{N}(h_{1}\vartriangleright n)h_{2})\overset{(\ref{Tau mh simple})}{=}\eta
_{N}(h_{1}\vartriangleright n)\varepsilon _{H}\left( h_{2}\right) =\eta
_{N}(h\vartriangleright n).
\end{gather*}%
So $\eta _{N}\in $ $_{H}^{H}\mathcal{YD}$, for each $N\in $ $_{H}^{H}%
\mathcal{YD}$.
\end{proof}

\section{Monoidal equivalences\label{C3}}

In this section we prove that the equivalence between the categories ${%
_{H}^{H}\mathfrak{M}_{H}^{H}}$ and ${_{H}^{H}\mathcal{YD}}$ becomes monoidal
if we equip ${_{H}^{H}\mathfrak{M}_{H}^{H}}$ with the tensor product $\otimes _{H}$
(or ${{\square _{H}}}$) and unit $H.$ As a by-product we produce a monoidal
equivalence between $({{{_{H}^{H}\mathfrak{M}_{H}^{H}}}},\otimes _{H},H)$
and $({_{H}^{H}\mathfrak{M}_{H}^{H},\square }_{H},H).$

\begin{lemma}
\label{lem:m4hmonoidal} Let $(H,m,u,\Delta ,\varepsilon ,\omega )$ be a dual
quasi-bialgebra. The category $(_{H}^{H}\mathfrak{M}_{H}^{H},\otimes _{H},H)$
is monoidal with respect to the following constraints:%
\begin{eqnarray*}
a_{U,V,W}((u\otimes _{H}v)\otimes _{H}w) &=&\omega ^{-1}(u_{-1}\otimes
v_{-1}\otimes w_{-1})u_{0}\otimes _{H}(v_{0}\otimes _{H}w_{0})\omega
(u_{1}\otimes v_{1}\otimes w_{1}) \\
l_{U}(h\otimes _{H}u) &=&hu \\
r_{U}(u\otimes _{H}h) &=&uh
\end{eqnarray*}
\end{lemma}

\begin{proof}
See e.g. \cite[Theorem 1.12]{AMS-Hoch}.
\end{proof}

\begin{lemma}
\label{strutt mod lat} Let $(H,m,u,\Delta ,\varepsilon ,\omega )$ be a dual
quasi-bialgebra. Let $U\in {^{H}\mathfrak{M}^{H}},V\in {^{H}\mathfrak{M}%
_{H}^{H}}.$ Then $(U\otimes V,\rho ^{l},\rho ^{r},\mu )\in {^{H}\mathfrak{M}%
_{H}^{H}}$ with the following structures:%
\begin{eqnarray*}
\rho _{U\otimes V}^{l}(u\otimes v) &=&u_{-1}v_{-1}\otimes (u_{0}\otimes v_{0})
\\
\rho _{U\otimes V}^{r}(u\otimes v) &=& (u_{0}\otimes v_{0})\otimes u_{1}v_{1} \\
\mu _{U\otimes V}^{r}((u\otimes v)\otimes h) &=&\omega ^{-1}(u_{-1}\otimes
v_{-1}\otimes h_{1})u_{0}\otimes v_{0}h_{2}\omega (u_{1}\otimes v_{1}\otimes
h_{3}).
\end{eqnarray*}
\end{lemma}

\begin{proof}
It is left to the reader. 
\end{proof}

\begin{definition}
We recall that a \emph{lax monoidal functor} 
\begin{equation*}
(F,\phi _{0},\phi _{2}):(\mathcal{M},\otimes ,\mathbf{1},a,l,r\mathbf{%
)\rightarrow (}\mathcal{M}^{\prime }\mathfrak{,}\otimes ^{\prime },\mathbf{%
1^{\prime }},a^{\prime },l^{\prime },r^{\prime }\mathbf{)}
\end{equation*}%
between two monoidal categories consists of

\begin{itemize}
\item a functor $F:\mathcal{M}\rightarrow \mathcal{M}^{\prime },$

\item a natural transformation $\phi _{2}(U,V):F(U)\otimes ^{\prime
}F(V)\rightarrow F(U\otimes V),$ with $U,V\in \mathcal{M}$, and

\item a natural transformation $\phi _{0}:\mathbf{1^{\prime }}\rightarrow F(%
\mathbf{1})$ such that the diagram 
\begin{equation}  \label{mon funct1}
\diagMonFun%
\end{equation}%
commutes and the following conditions are satisfied: 
\begin{eqnarray}
F(l_{U})\circ \phi _{2}(\mathbf{1},U)\circ (\phi _{0}\otimes F(U))
&=&l_{F(U)}^{\prime },  \label{mon funct2} \\
F(r_{U})\circ \phi _{2}(U,\mathbf{1})\circ \left( F(U)\otimes \phi
_{0}\right) &=&r_{F(U)}^{\prime }.  \label{mon funct3}
\end{eqnarray}
\end{itemize}

The morphisms $\phi _{2}(U,V)$ and $\phi _{0}$ are called \emph{structure
morphisms}.

\emph{Colax monoidal functors} are defined similarly but with the directions
of the structure morphisms reversed. A strong monoidal functor or simply a 
\emph{monoidal functor} is a lax monoidal functor with invertible structure
morphisms.
\end{definition}

\begin{lemma}
\label{lem: equiv-m4h-tens}Let $(H,m,u,\Delta ,\varepsilon ,\omega )$ be a
dual quasi-bialgebra. The functor $F:{_{H}^{H}\mathcal{YD}}\rightarrow {%
_{H}^{H}\mathfrak{M}_{H}^{H}}$ defines a monoidal functor $F:({_{H}^{H}%
\mathcal{YD}},\otimes ,\Bbbk )\rightarrow ({_{H}^{H}\mathfrak{M}%
_{H}^{H},\otimes _{H},H).}$ For $U,V\in {_{H}^{H}\mathcal{YD}}$, the
structure morphisms are 
\begin{equation*}
\varphi _{2}(U,V):F(U)\otimes _{H}F(V)\rightarrow F(U\otimes V)\qquad \text{%
and} \qquad \varphi _{0}:H\rightarrow F(\Bbbk )
\end{equation*}%
which are defined, for every $u\in U,v\in V,h,k\in H$, by 
\begin{equation*}
\varphi _{2}(U,V)[(u\otimes h)\otimes _{H}(v\otimes k)]:=\left[ 
\begin{array}{c}
\omega ^{-1}\left( u_{-2}\otimes h_{1}\otimes v_{-2}k_{1}\right) \omega
(h_{2}\otimes v_{-1}\otimes k_{2}) \\ 
\omega ^{-1}((h_{3}\vartriangleright v_{0})_{-2}\otimes h_{4}\otimes
k_{3})\omega \left( u_{-1}\otimes (h_{3}\vartriangleright v_{0})_{-1}\otimes
h_{5}k_{4})\right) \\ 
\left( u_{0}\otimes (h_{3}\vartriangleright v_{0})_{0}\right) \otimes
h_{6}k_{5}%
\end{array}%
\right]
\end{equation*}%
and 
\begin{equation*}
\varphi _{0}(h):=1_{\Bbbk }\otimes h.
\end{equation*}%
Moreover 
\begin{equation*}
\varphi _{2}(U,V)^{-1}\left( \left( u\otimes v\right) \otimes k\right)
=\omega ^{-1}\left( u_{-1}\otimes v_{-1}\otimes k_{1}\right) \left(
u_{0}\otimes 1_{H}\right) \otimes _{H}\left( v_{0}\otimes k_{2}\right) .
\label{form:fi2meno1}\end{equation*}
\end{lemma}

\begin{proof}
Let us check that $\varphi _{0}$ is a morphism in ${{_{H}^{H}\mathfrak{M}%
_{H}^{H}}}$. Since $\varphi _{0}=l_{H}^{-1}:H\rightarrow \mathbb{\Bbbk
\otimes }H$, i.e. the inverse of the left unit constraint in ${{^{H}%
\mathfrak{M}^{H}}}$, then $\varphi _{0}$ is in ${{^{H}\mathfrak{M}^{H}}}$
and it is invertible. It is easy to check it is $H$-bilinear in ${{^{H}%
\mathfrak{M}^{H}}}$.

Let us consider now $\varphi _{2}(U,V).$

By Lemma \ref{lem: xiScha}, for all $U,V\in {_{H}^{H}\mathcal{YD}}${,} the
map $\xi _{U,F\left( V\right) }:F\left( U\right) \otimes _{H}F\left(
V\right) \rightarrow U\otimes F\left( V\right) ,$ 
is a natural isomorphism in ${{_{H}^{H}\mathfrak{M}_{H}^{H}}}$. 
By Lemma \ref{lem:a4H}, $\alpha _{U,V}:U\otimes F\left( V\right) \rightarrow
F\left( U\otimes V\right) $ is a natural isomorphism in ${_{H}^{H}\mathfrak{M%
}_{H}^{H}}$, where $U\otimes F\left( V\right) $ has the structure described
in Lemma \ref{lem: xiScha} for $M=F\left( V\right) .$

Thus $\alpha _{U,V}\xi _{U,F\left( V\right) }:F\left(
U\right) \otimes _{H}F\left( V\right) \rightarrow F\left( U\otimes V\right) $
is a natural isomorphism in ${_{H}^{H}\mathfrak{M}_{H}^{H}.}$ A direct
computation shows that $\varphi _{2}(U,V)=\alpha _{U,V}\xi _{U,F\left(
V\right) }$ and hence $\varphi _{2}(U,V)$ is a well-defined isomorphism in ${%
_{H}^{H}\mathfrak{M}_{H}^{H}}$. Moreover $\varphi _{2}(U,V)^{-1}=\xi
_{U,F\left( V\right) }^{-1}\alpha _{U,V}^{-1}$ fulfills \eqref{form:fi2meno1}.

In order to check the commutativity of the diagram (\ref{mon funct1}) it
suffices to prove the following equality:%
\begin{equation*}
\lbrack \varphi _{2}^{-1}(U,V)\otimes _{H}F(W)]\varphi _{2}^{-1}(U\otimes
V,W)F(a_{U,V,W}^{-1})=a_{F(U),F(V),F(W)}^{-1}[F(U)\otimes _{H}\varphi
_{2}^{-1}(V,W)]\varphi _{2}^{-1}(U,V\otimes W)
\end{equation*}

Since these maps are right $H$-linear, it suffices to check this equality on
elements of the form $(u\otimes (v\otimes w))\otimes 1_{H},$ where $u\in
U,v\in V,w\in W$. This computation and the ones of \eqref{mon funct2} and %
\eqref{mon funct3} are straightforward.
\end{proof}

We now compute explicitly the braiding induced on ${_{H}^{H}\mathfrak{M}%
_{H}^{H}}$ through the functor $F$ in Lemma \ref{lem:
equiv-m4h-tens} in case $F$ comes out to be an equivalence i.e. when $H$ has a preantipode.

\begin{lemma}
Let $(H,m,u,\Delta ,\varepsilon ,\omega ,S)$ be a dual quasi-bialgebra with
a preantipode. Through the monoidal equivalence $(F,G)$ we have that $({{%
_{H}^{H}\mathfrak{M}_{H}^{H},\otimes _{H},H)}}$ becomes a pre-braided
monoidal category, with braiding defined as follows:%
\begin{equation*}
c_{M,N}(m\otimes _{H}n)=\omega (m_{-2}\otimes \tau (n_{0})_{-1}\otimes
n_{1})(m_{-1}\vartriangleright \tau (n_{0})_{0}\otimes _{H}m_{0})\cdot n_{2},
\end{equation*}%
where $M,N\in {{_{H}^{H}\mathfrak{M}_{H}^{H}}}$ and $m\in M,n\in N.$
\end{lemma}

\begin{proof}
First of all, for any $U,V\in {_{H}^{H}\mathcal{YD}}$, let us consider the
following composition:%
\begin{equation*}
\lambda _{U,V}:=\left( F\left( U\right) \otimes _{H}F\left( V\right) \overset%
{\varphi _{2}(U,V)}{\longrightarrow }F(U\otimes V)\overset{F(c_{U,V})}{%
\longrightarrow }F(V\otimes U)\overset{\varphi _{2}^{-1}(V,U)}{%
\longrightarrow }F\left( V\right) \otimes _{H}F\left( U\right) \right) .
\end{equation*}%
This map is right $H$-linear, so, if we compute%
\begin{eqnarray*}
&&\lambda _{U,V}[(u\otimes h)\otimes _{H}(v\otimes 1_{H})] \\
&=&\left[ 
\begin{array}{c}
\omega ^{-1}\left( u_{-4}\otimes h_{1}\otimes v_{-1}\right) \omega \left(
u_{-3}\otimes (h_{2}\vartriangleright v_{0})_{-1}\otimes h_{3})\right) \\ 
\omega ^{-1}\left( \left( u_{-2}\vartriangleright (h_{2}\vartriangleright
v_{0})_{0}\right) _{-1}\otimes u_{-1}\otimes h_{4}\right) \left( \left(
u_{-2}\vartriangleright (h_{2}\vartriangleright v_{0})_{0}\right)
_{0}\otimes 1_{H}\right) \otimes _{H}u_{0}\otimes h_{5}%
\end{array}%
\right] \\
&\overset{(\ref{ass YD})}{=}&((u_{-1}h_{1})\vartriangleright v\otimes
1_{H})\otimes _{H}(u_{0}\otimes h_{2}),
\end{eqnarray*}
we obtain%
\begin{eqnarray*}
&&\lambda _{U,V}[(u\otimes h)\otimes _{H}(v\otimes k)] \\
&=&\lambda _{U,V}[(u\otimes h)\otimes _{H}(v\otimes 1_{H})\cdot k] \\
&=&\omega (u_{-1}h_{1}\otimes v_{-1}\otimes k_{1})\lambda
_{U,V}[[(u_{0}\otimes h_{2})\otimes _{H}(v_{0}\otimes 1_{H})]\cdot
k_{2}]\omega ^{-1}(h_{3}\otimes 1_{H}\otimes k_{3}) \\
&=&\omega (u_{-1}h_{1}\otimes v_{-1}\otimes k_{1})\lambda
_{U,V}[[(u_{0}\otimes h_{2})\otimes _{H}(v_{0}\otimes 1_{H})]\cdot k_{2}] \\
&=&\omega (u_{-1}h_{1}\otimes v_{-1}\otimes k_{1})\lambda
_{U,V}[(u_{0}\otimes h_{2})\otimes _{H}(v_{0}\otimes 1_{H})]\cdot k_{2} \\
&=&\omega (u_{-2}h_{1}\otimes v_{-1}\otimes
k_{1})[(u_{-1}h_{2})\vartriangleright v_{0}\otimes 1_{H})\otimes
_{H}(u_{0}\otimes h_{3})]\cdot k_{2}.
\end{eqnarray*}

Now, using the map $\lambda _{U,V}$, we construct the braiding of ${{_{H}^{H}%
\mathfrak{M}_{H}^{H}}}$ in this way: 
\begin{equation*}
M\otimes _{H}N\overset{\epsilon _{M}^{-1}\otimes _{H}\epsilon _{N}^{-1}}{%
\longrightarrow }FG(M)\otimes _{H}FG(N)\overset{\lambda _{G(M),G(N)}}{%
\longrightarrow }FG(N)\otimes _{H}FG(M)\overset{\epsilon _{N}\otimes
_{H}\epsilon _{M}}{\longrightarrow }N\otimes _{H}M.
\end{equation*}

Therefore
\begin{eqnarray*}
&&(\epsilon _{N}\otimes _{H}\epsilon _{M})\lambda _{G(M),G(N)}(\epsilon
_{M}^{-1}\otimes _{H}\epsilon _{N}^{-1})(m\otimes _{H}n) \\
&=&(\epsilon _{N}\otimes _{H}\epsilon _{M})\lambda _{G(M),G(N)}\left\{ [\tau
(m_{0})\otimes m_{1}]\otimes _{H}[\tau (n_{0})\otimes n_{1}]\right\} \\
&=&\left[ 
\begin{array}{c}
\omega (\tau (m_{0})_{-2}m_{1}\otimes \tau (n_{0})_{-1}\otimes n_{1}) \\ 
(\epsilon _{N}\otimes _{H}\epsilon _{M})\left\{ [(\tau
(m_{0})_{-1}m_{2})\vartriangleright \tau (n_{0})_{0}\otimes 1_{H})\otimes
_{H}(\tau (m_{0})_{0}\otimes m_{3})]\cdot n_{2}\right\}%
\end{array}%
\right] \\
&\overset{(\ref{col sx eps})}{=}&\left[ 
\begin{array}{c}
\omega (m_{-2}\otimes \tau (n_{0})_{-1}\otimes n_{1}) \\ 
(\epsilon _{N}\otimes _{H}\epsilon _{M})[(m_{-1}\vartriangleright \tau
(n_{0})_{0}\otimes 1_{H})\otimes _{H}(\tau (m_{0})\otimes m_{1})]\cdot n_{2}%
\end{array}%
\right] \\
&=&\omega (m_{-2}\otimes \tau (n_{0})_{-1}\otimes
n_{1})[(m_{-1}\vartriangleright \tau (n_{0})_{0}\otimes _{H}\tau
(m_{0})m_{1}]\cdot n_{2} \\
&\overset{(\ref{inv eps})}{=}&\omega (m_{-2}\otimes \tau (n_{0})_{-1}\otimes
n_{1})[(m_{-1}\vartriangleright \tau (n_{0})_{0}\otimes _{H}m_{0}]\cdot
n_{2}.
\end{eqnarray*}
\end{proof}

Next aim is to prove that the equivalence between the categories ${_{H}^{H}%
\mathfrak{M}_{H}^{H}}$ and ${_{H}^{H}\mathcal{YD}}$ becomes monoidal if we
equip ${_{H}^{H}\mathfrak{M}_{H}^{H}}$ with the tensor product $\square _{H}$
and unit $H.$

\begin{claim}
\label{constr cot}Let $(H,m,u,\Delta ,\varepsilon ,\omega )$ be a dual
quasi-bialgebra. Note that, since $H$ is an ordinary coalgebra, we have that 
$\left( {{^{H}\mathfrak{M}^{H},\square _{H},H,b,r,l}}\right) $ is a monoidal
category with constraints defined, for all $L,M,N\in {{{^{H}\mathfrak{M}^{H}}%
}}${,} by%
\begin{equation*}
b_{L,M,N}:(L\square _{H}M)\square _{H}N\rightarrow L\square _{H}(M\square
_{H}N):(l\square _{H}m)\square _{H}n\mapsto l\square _{H}(m\square _{H}n),
\end{equation*}%
\begin{eqnarray*}
r_{M} &:&M\square _{H}H\longrightarrow M:m\square _{H}h\mapsto m\varepsilon
_{H}(h), \\
l_{M} &:&H\square _{H}M\longrightarrow M:h\square _{H}m\mapsto \varepsilon
_{H}(h)m.
\end{eqnarray*}%
where, for sake of brevity we just wrote $m\square _{H}n$ in place of the
more precise $\sum_{i}m^{i}\square _{H}n^{i}$.

We want to endow ${{_{H}^{H}\mathfrak{M}_{H}^{H}}}$ with a monoidal
structure, following the dual version of \cite{hausser and nill} (see also 
\cite[Definition 3.2]{Schauenburg-HopfMod}). The definition of the claimed
structure is given in such a way that the forgetful functor ${{_{H}^{H}%
\mathfrak{M}_{H}^{H}\rightarrow {^{H}\mathfrak{M}^{H}}}}$ is a strict
monoidal functor. Hence the constraints are induced by the ones of ${{^{H}%
\mathfrak{M}^{H}}}$ (i.e. $b_{L,M,N}$,$l_{M}$ and $r_{M}$), and the tensor
product is given by $M\square _{H}N$ with structures%
\begin{eqnarray*}
\rho _{M\square _{H}N}^{l}(m\square _{H}n) &=&m_{-1}\otimes (m_{0}\square
_{H}n), \\
\rho _{M\square _{H}N}^{r}(m\square _{H}n) &=&(m\square _{H}n_{0})\otimes
n_{1}, \\
\mu _{M\square _{H}N}^{l}\left[ h\otimes (m\square _{H}n)\right] &=&h\cdot
(m\square _{H}n)=h_{1}m\square _{H}h_{2}n, \\
\mu _{M\square _{H}N}^{r}\left[ (m\square _{H}n)\otimes h\right]
&=&(m\square _{H}n)\cdot h=mh_{1}\square _{H}nh_{2}.
\end{eqnarray*}

The unit of the category is $H$ endowed with the following structures: 
\begin{eqnarray*}
\rho _{H}^{l}(h) &=&h_{1}\otimes h_{2},\qquad \rho _{H}^{r}(h)=h_{1}\otimes
h_{2}, \\
h\cdot l &=&hl,\qquad l\cdot h=lh.
\end{eqnarray*}
\end{claim}

The following result is similar to 2) in Lemma \ref{lem: xiScha}.

\begin{lemma}
\label{lem: xiSchaCOT}Let $(H,m,u,\Delta ,\varepsilon ,\omega )$ be a dual
quasi-bialgebra. For all $V\in {_{H}^{H}\mathcal{YD}}$ and $M\in {{_{H}^{H}%
\mathfrak{M}_{H}^{H}}}${,} the map 
\begin{equation*}
\beta _{V,M}:F(V)\square _{H}M\longrightarrow V\otimes M:(v\otimes h)\square
_{H}m\mapsto v\varepsilon (h)\otimes m
\end{equation*}%
is a natural isomorphism in ${{_{H}^{H}\mathfrak{M}_{H}^{H}}}$ where $%
V\otimes M$ has the structures as in Lemma \ref{lem: xiScha}. The inverse of 
$\beta _{V,M}$ is given by 
\begin{equation*}
\beta _{V,M}^{-1}:V\otimes M\longrightarrow (V\otimes H)\square
_{H}M:v\otimes m\mapsto \left( v\otimes m_{-1}\right) \square _{H}m_{0}.
\end{equation*}
\end{lemma}

\begin{proof}
The proof is straightforward and is based on the fact that $\left( v\otimes
h\right) \square _{H}m\in (V\otimes H)\square _{H}M$ implies 
\begin{equation}
\left( v\otimes h\right) \otimes m=\left( v\varepsilon \left( h\right)
\otimes m_{-1}\right) \otimes m_{0}.  \label{form: cot}
\end{equation}

\end{proof}

\begin{lemma}
\label{lem: F monoidal cot}(cf. \cite[Proposition 3.6]{Schauenburg-HopfMod})
Let $(H,m,u,\Delta ,\varepsilon ,\omega )$ be a dual quasi-bialgebra. The
functor $F:{_{H}^{H}\mathcal{YD}}\rightarrow {_{H}^{H}\mathfrak{M}_{H}^{H}}$
defines a monoidal functor $F:({_{H}^{H}\mathcal{YD}},\otimes ,\Bbbk
)\rightarrow ({_{H}^{H}\mathfrak{M}_{H}^{H},\square }_{H},H).$ For $U,V\in {%
_{H}^{H}\mathcal{YD}}$, the structure morphisms are 
\begin{equation*}
\psi _{2}(U,V):F(U)\square _{H}F(V)\rightarrow F(U\otimes V)\qquad \text{and}%
\qquad \psi _{0}:H\rightarrow F(\Bbbk )
\end{equation*}%
which are defined, for every $u\in U,v\in V,h,k\in H$, by 
\begin{equation}
\psi _{2}(U,V)[(u\otimes h)\otimes (v\otimes k)]:=\omega (u_{-1}\otimes
v_{-1}\otimes k_{1})u_{0}\varepsilon (h)\otimes v_{0}\otimes k_{2}
\label{form: psi2}
\end{equation}%
and 
\begin{equation*}
\psi _{0}(h):=1_{\Bbbk }\otimes h.
\end{equation*}

Moreover 
\begin{equation}
\psi _{2}(U,V)^{-1}\left( \left( u\otimes v\right) \otimes h\right) =\omega
^{-1}(u_{-1}\otimes v_{-2}\otimes h_{1})(u_{0}\otimes v_{-1}h_{2})\otimes
(v_{0}\otimes h_{3}).  \label{form: psi2-1}
\end{equation}
\end{lemma}

\begin{proof}
Since $\psi _{0}=\varphi _{0}$ as in Lemma \ref{lem: equiv-m4h-tens}, we
already know that $\psi _{0}$ is an isomorphism in ${_{H}^{H}\mathfrak{M}%
_{H}^{H}}$. Let us deal with $\psi _{2}(U,V)$. By Lemma \ref{lem:a4H}, the map 
$\alpha _{U,V}:U\otimes F\left( V\right) \rightarrow F\left( U\otimes
V\right) $ is a natural isomorphism in ${_{H}^{H}\mathfrak{M}_{H}^{H}}$,
where $U\otimes F\left( V\right) $ has the structure described in Lemma \ref%
{lem: xiScha} for $M=F\left( V\right) .$ 
By Lemma \ref{lem: xiSchaCOT}, $\beta _{U,F(V)}=\beta :F(U)\square
_{H}F(V)\longrightarrow U\otimes F(V)$ is a natural isomorphism in ${_{H}^{H}%
\mathfrak{M}_{H}^{H}}$, where $U\otimes F\left( V\right) $ has the structure
described in Lemma \ref{lem: xiScha} for $M=F\left( V\right) .$ Hence it
makes sense to consider the composition $\psi _{2}(U,V):=\alpha _{U,V}\circ
\beta _{U,V\otimes H}$. Then $\psi _{2}(U,V)$ fulfills (\ref{form: psi2}). It
is clear that $\psi _{2}(U,V):F(U)\square _{H}F(V)\rightarrow F\left(
U\otimes V\right) $ is a natural isomorphism in ${_{H}^{H}\mathfrak{M}%
_{H}^{H}}$ with inverse given by $\psi _{2}(U,V)^{-1}:=\beta _{U,V\otimes
H}^{-1}\circ \alpha _{U,V}^{-1}.$ Moreover $\psi _{2}(U,V)^{-1}$ satisfies (\ref%
{form: psi2-1}).

In order to check the commutativity of the diagram (\ref{mon funct1}) it
suffices to prove the following equality:%
\begin{eqnarray*}
&&(\psi _{2}(U,V)^{-1}\otimes F(W))\psi _{2}(U\otimes
V,W)^{-1}F(a_{U,V,W}^{-1})[(u\otimes (v\otimes w))\otimes h] \\
&=&b_{F(U),F(V).F(W)}^{-1}[F(U)\otimes \psi _{2}(V,W)^{-1}]\psi
_{2}(U,V\otimes W)^{-1}[(u\otimes (v\otimes w))\otimes h].
\end{eqnarray*}%
By right $H$-linearity, it suffices to check the displayed equality for $%
h=1_{H}$. The proof of this fact and of (\ref{mon funct2}) and (\ref{mon
funct3}) is straightforward.
\end{proof}

If $H$ has a preantipode, the functor $F$ of Lemma \ref{lem: F monoidal cot}
is an equivalence. As a consequence, its adjoint $G$ is monoidal too. For
future reference we include here its explicit monoidal structure.

\begin{lemma}
\label{lem: G monoidal} Let $(H,m,u,\Delta ,\varepsilon ,\omega ,S)$ be a
dual quasi-bialgebra with a preantipode. The right adjoint $G:{_{H}^{H}%
\mathfrak{M}_{H}^{H}\rightarrow {_{H}^{H}\mathcal{YD}}}$ of the functor $F$,
defines a monoidal functor $G:({_{H}^{H}\mathfrak{M}_{H}^{H},\square }%
_{H},H)\rightarrow ({_{H}^{H}\mathcal{YD}},\otimes ,\Bbbk ).$ For $M,N\in {%
_{H}^{H}\mathfrak{M}_{H}^{H}\,,}$ the structure morphisms are%
\begin{equation*}
\psi _{2}^{G}(M,N):G(M)\otimes G(N)\rightarrow G(M{\square }_{H}N)\qquad 
\text{and}\qquad \psi _{0}^{G}:\mathbb{\Bbbk }\rightarrow G(H)
\end{equation*}%
which are defined, for every $m\in M,n\in N,k\in H,$ by 
\begin{equation*}
\psi _{2}^{G}(M,N)\left( m\otimes n\right) =mn_{-1}{\square }_{H}n_{0}\qquad 
\text{and}\qquad \psi _{0}^{G}(k):=k1_{H}.
\end{equation*}%
Moreover, for all $m\in M,n\in N,$ 
\begin{equation*}
\psi _{2}^{G}(M,N)^{-1}\left( m\square _{H}n\right) =\tau \left( m\right)
\otimes \tau \left( n\right) .
\end{equation*}
\end{lemma}

\begin{proof}
Apply \cite[Proposition 1.4]{AMS-Amonoidal} to the functor $F$. Then $G$ is
monoidal with structure morphisms 
\begin{eqnarray*}
\psi _{2}^{G}(M,N) &:&=G\left( \epsilon _{M}{\square }_{H}\epsilon
_{M}\right) \circ G\left( \psi _{2}(GM,GN)^{-1}\right) \circ \eta
_{GM\otimes GN}, \\
\psi _{0}^{G} &:&=G\left( \psi _{0}^{-1}\right) \circ \eta _{\Bbbk }
\end{eqnarray*}%
A direct computation shows that they are the desired maps.


The inverse of $\psi _{2}^{G}(M,N)$ can be computed by%
\begin{equation*}
\psi _{2}^{G}(M,N)^{-1}:=\eta _{GM\otimes GN}^{-1}\circ G\left( \psi
_{2}(GM,GN)\right) \circ G\left( \epsilon _{M}^{-1}{\square }_{H}\epsilon
_{M}^{-1}\right)
\end{equation*}%
%
%
%
%
%
%
%
%
%
%
%
%
%
%
%
%
%
%
%
%
%
%
%
%
%
%
%
%
\end{proof}

\begin{remark}
\label{k} Consider the composition%
\begin{equation*}
\kappa =\kappa (U,V):=\psi _{2}(U,V)^{-1}\circ \varphi _{2}(U,V):(U\otimes
H)\otimes _{H}(V\otimes H)\longrightarrow (U\otimes H)\square _{H}(V\otimes
H).
\end{equation*}%
We have 
\begin{eqnarray*}
&&\kappa (U,V)\left[ (u\otimes h)\otimes _{H}(v\otimes k)\right] \\
&=&\psi _{2}(U,V)^{-1}\varphi _{2}(U,V)\left[ (u\otimes h)\otimes
_{H}(v\otimes k)\right] \\
&=&\left[ 
\begin{array}{c}
\omega ^{-1}\left( u_{-2}\otimes h_{1}\otimes v_{-2}k_{1}\right) \omega
(h_{2}\otimes v_{-1}\otimes k_{2}) \\ 
\omega ^{-1}((h_{3}\vartriangleright v_{0})_{-2}\otimes h_{4}\otimes
k_{3})\omega \left( u_{-1}\otimes (h_{3}\vartriangleright v_{0})_{-1}\otimes
h_{5}k_{4})\right) \\ 
\psi _{2}(U,V)^{-1}[\left( u_{0}\otimes (h_{3}\vartriangleright
v_{0})_{0}\right) \otimes (h_{6}k_{5})]%
\end{array}%
\right] \\
&=&\left[ 
\begin{array}{c}
\omega ^{-1}\left( u_{-2}\otimes h_{1}\otimes v_{-2}k_{1}\right) \omega
(h_{2}\otimes v_{-1}\otimes k_{2}) \\ 
\omega ^{-1}((h_{3}\vartriangleright v_{0})_{-2}\otimes h_{4}\otimes
k_{3})\omega \left( u_{-1}\otimes (h_{3}\vartriangleright v_{0})_{-1}\otimes
h_{5}k_{4})\right) \\ 
\omega ^{-1}(u_{0-1}\otimes (h_{3}\vartriangleright v_{0})_{0-2}\otimes
(h_{6}k_{5})_{1})\\(u_{00}\otimes (h_{3}\vartriangleright
v_{0})_{0-1}(h_{6}k_{5})_{2})\square _{H}((h_{3}\vartriangleright
v_{0})_{00}\otimes (h_{6}k_{5})_{3})%
\end{array}%
\right] \\
&=&\left[ 
\begin{array}{c}
\omega ^{-1}\left( u_{-3}\otimes h_{1}\otimes v_{-2}k_{1}\right) \omega
(h_{2}\otimes v_{-1}\otimes k_{2}) \\ 
\omega ^{-1}((h_{3}\vartriangleright v_{0})_{-4}\otimes h_{4}\otimes
k_{3})\omega (u_{-2}\otimes (h_{3}\vartriangleright v_{0})_{-3}\otimes
h_{5}k_{4}) \\ 
\omega ^{-1}(u_{-1}\otimes (h_{3}\vartriangleright v_{0})_{-2}\otimes
h_{6}k_{5})(u_{0}\otimes (h_{3}\vartriangleright
v_{0})_{-1}(h_{7}k_{6}))\square _{H}((h_{3}\vartriangleright
v_{0})_{0}\otimes h_{8}k_{7})%
\end{array}%
\right] \\
&=&\left[ 
\begin{array}{c}
\omega ^{-1}\left( u_{-1}\otimes h_{1}\otimes v_{-2}k_{1}\right) \omega
(h_{2}\otimes v_{-1}\otimes k_{2}) \\ 
\omega ^{-1}((h_{3}\vartriangleright v_{0})_{-2}\otimes h_{4}\otimes k_{3})
\\ 
(u_{0}\otimes (h_{3}\vartriangleright v_{0})_{-1}(h_{5}k_{4}))\square
_{H}((h_{3}\vartriangleright v_{0})_{0}\otimes h_{6}k_{5})%
\end{array}%
\right] \\
&\overset{(\ref{eq:quasi-associativity})}{=}&\left[ 
\begin{array}{c}
\omega ^{-1}\left( u_{-1}\otimes h_{1}\otimes v_{-2}k_{1}\right) \omega
(h_{2}\otimes v_{-1}\otimes k_{2}) \\ 
(u_{0}\otimes ((h_{3}\vartriangleright v_{0})_{-2}h_{4})k_{3})) \\ 
\omega ^{-1}((h_{3}\vartriangleright v_{0})_{-1}\otimes h_{5}\otimes
k_{4})\square _{H}((h_{3}\vartriangleright v_{0})_{0}\otimes h_{6}k_{5})%
\end{array}%
\right] \\
&=&\left[ 
\begin{array}{c}
\omega ^{-1}\left( u_{-1}\otimes h_{1}\otimes v_{-2}k_{1}\right) \omega
(h_{2}\otimes v_{-1}\otimes k_{2}) \\ 
(u_{0}\otimes ((h_{3}\vartriangleright v_{0})_{-1}h_{4})k_{3}))\square
_{H}((h_{3}\vartriangleright v_{0})_{0}\otimes h_{5})\cdot k_{4})%
\end{array}%
\right] \\
&\overset{(\ref{Comp YD})}{=}&\left[ 
\begin{array}{c}
\omega ^{-1}\left( u_{-1}\otimes h_{1}\otimes v_{-3}k_{1}\right) \omega
(h_{2}\otimes v_{-2}\otimes k_{2}) \\ 
(u_{0}\otimes ((h_{3}v_{-1})k_{3}))\square _{H}((h_{4}\vartriangleright
v_{0})\otimes h_{5})\cdot k_{4})%
\end{array}%
\right] \\
&\overset{(\ref{eq:quasi-associativity})}{=}&\left[ 
\begin{array}{c}
\omega ^{-1}\left( u_{-1}\otimes h_{1}\otimes v_{-3}k_{1}\right) \\ 
(u_{0}\otimes (h_{2}(v_{-2}k_{2})) \\ 
\omega (h_{3}\otimes v_{-1}\otimes k_{3})\square
_{H}((h_{4}\vartriangleright v_{0})\otimes h_{5})\cdot k_{4})%
\end{array}%
\right] \\
&\overset{(\ref{strutt mod left})}{=}&\omega ^{-1}\left( u_{-1}\otimes
h_{1}\otimes v_{-2}k_{1}\right) (u_{0}\otimes (h_{2}(v_{-1}k_{2}))\square
_{H}(h_{3}\cdot (v_{0}\otimes k_{3})) \\
&=&(u_{0}\otimes h_{1})\cdot (v_{-1}k_{1})\square _{H}h_{3}\cdot
(v_{0}\otimes k_{3}) \\
&=&(u\otimes h)_{0}\cdot (v\otimes k)_{-1}\square _{H}(u\otimes h)_{1}\cdot
(v\otimes k)_{0}.
\end{eqnarray*}%
so that%
\begin{equation*}
\kappa (U,V)\left[ (u\otimes h)\otimes _{H}(v\otimes k)\right] =(u\otimes
h)_{0}\cdot (v\otimes k)_{-1}\square _{H}(u\otimes h)_{1}\cdot (v\otimes
k)_{0}.
\end{equation*}%
Thus, for $M,N\in {{_{H}^{H}\mathfrak{M}_{H}^{H},}}$, using that the
counit $\epsilon $ is in ${{_{H}^{H}\mathfrak{M}_{H}^{H}}}$, one gets%
\begin{equation*}
\left[ \left( \epsilon _{M}\square _{H}\epsilon _{N}\right) \circ \kappa
(M^{coH},N^{coH})\circ \left( \epsilon _{M}^{-1}\otimes _{H}\epsilon
_{N}^{-1}\right) \right] \left( m\otimes _{H}n\right) =m_{0}n_{-1}\square
_{H}m_{1}n_{0}.
\end{equation*}%
%
%
%

We can also compute $\kappa (U,V)^{-1}:=\varphi _{2}(U,V)^{-1}\circ \psi
_{2}(U,V).$ We have:%
\begin{equation*}
\kappa (U,V)^{-1}((u\otimes h)\square _{H}(v\otimes k))=(u\varepsilon
(h)\otimes 1_{H})\otimes _{H}(v\otimes k).
\end{equation*}%
%
%
%
%
%
%
%
%
%
%
%
%
%
\end{remark}

We are now able to provide a monoidal equivalence between $({{{_{H}^{H}%
\mathfrak{M}_{H}^{H}}}},\otimes _{H},H)$ and $({_{H}^{H}\mathfrak{M}%
_{H}^{H},\square }_{H},H).$ This result is similar to \cite[Corollary 6.1]%
{Schauenburg-YD}.

\begin{lemma}
\label{equi tensH, cot}Let $(H,m,u,\Delta ,\varepsilon ,\omega ,S)$ be a
dual quasi-bialgebra with a preantipode. The identity functor on ${{_{H}^{H}%
\mathfrak{M}_{H}^{H}}}$ defines a monoidal functor $E:({{{_{H}^{H}\mathfrak{M%
}_{H}^{H}}}},\otimes _{H},H)\rightarrow ({_{H}^{H}\mathfrak{M}%
_{H}^{H},\square }_{H},H).$ For $M,N\in {{{_{H}^{H}\mathfrak{M}_{H}^{H}}}}$,
the structure morphisms are%
\begin{equation*}
\vartheta _{2}(M,N):E(M)\square _{H}E(V)\rightarrow E(M\otimes _{H}N)\qquad%
\text{and}\qquad \vartheta _{0}:H\rightarrow E(H)=H
\end{equation*}%
which are defined, for every $m\in M,n\in N,h\in H$, by 
\begin{equation*}
\vartheta _{2}(M,N)(m\square _{H}n):=\tau (m)\otimes _{H}n\qquad \text{and}%
\qquad \vartheta _{0}(h):=h.
\end{equation*}

Moreover 
\begin{eqnarray}
\vartheta _{2}(M,N)^{-1}\left( m\otimes _{H}n\right) &=&m_{0}n_{-1}\square
_{H}m_{1}n_{0}  \label{form:thetamen1},\\
\vartheta _{2}(FU,FV)&=&\varphi _{2}(U,V)^{-1}\circ \psi _{2}(U,V).
\label{form:tetaphipsi}
\end{eqnarray}
\end{lemma}

\begin{proof}
Using the map $\kappa $ of Remark \ref{k}, for each $M$, $N\in {{_{H}^{H}%
\mathfrak{M}_{H}^{H}}}$, we set%
\begin{equation*}
\vartheta _{2}(M,N):=\left( \epsilon _{M}\otimes _{H}\epsilon _{N}\right)
\circ \kappa (M^{coH},N^{coH})^{-1}\circ \left( \epsilon _{M}^{-1}\square
_{H}\epsilon _{N}^{-1}\right) .
\end{equation*}%
Clearly, by Remark \ref{k}, $\vartheta _{2}(M,N)^{-1}$ fulfills (\ref%
{form:thetamen1}). Moreover, using \eqref{inv eps}, one gets 
\begin{eqnarray*}
&&\vartheta _{2}(M,N)(m\square _{H}n) =\tau (m)\otimes _{H}n.
\end{eqnarray*}

It is straightforward to check that $\vartheta _{2}^{-1}$ makes commutative
the diagram (\ref{mon funct1}) and that \eqref{mon funct2} and 
\eqref{mon
funct3} hold.
Let us check that \eqref{form:tetaphipsi} holds:
\begin{eqnarray*}
&&\vartheta _{2}(FU,FV)=\left( \epsilon _{FU}\otimes _{H}\epsilon
_{FV}\right) \circ \kappa (GFU,GFV)^{-1}\circ \left( \epsilon
_{FU}^{-1}\square _{H}\epsilon _{FV}^{-1}\right) \\
&=&\left( \epsilon _{FU}\otimes _{H}\epsilon _{FV}\right) \circ \varphi
_{2}(GFU,GFV)^{-1}\circ \psi _{2}(GFU,GFV)\circ \left( \epsilon
_{FU}^{-1}\square _{H}\epsilon _{FV}^{-1}\right) \\
&=&\left[ 
\begin{array}{c}
\left( \epsilon _{FU}\otimes _{H}\epsilon _{FV}\right) \circ \varphi
_{2}(GFU,GFV)^{-1}\circ F\left( \eta _{U}\otimes \eta _{V}\right) \\ 
F\left( \eta _{U}^{-1}\otimes \eta _{V}^{-1}\right) \circ \psi
_{2}(GFU,GFV)\circ \left( \epsilon _{FU}^{-1}\square _{H}\epsilon
_{FV}^{-1}\right)%
\end{array}%
\right] \\
&=&\left[ 
\begin{array}{c}
\left( \epsilon _{FU}\otimes _{H}\epsilon _{FV}\right) \circ \left( F\eta
_{U}\otimes F\eta _{V}\right) \circ \varphi _{2}(U,V)^{-1} \\ 
\psi _{2}(U,V)\circ \left( F\eta _{U}^{-1}\otimes F\eta _{V}^{-1}\right)
\circ \left( \epsilon _{FU}^{-1}\square _{H}\epsilon _{FV}^{-1}\right)%
\end{array}%
\right] =\varphi _{2}(U,V)^{-1}\circ \psi _{2}(U,V).
\end{eqnarray*}
\end{proof}

The following result is similar to \cite[Proposition 3.11]%
{Schauenburg-HopfMod}.

\begin{corollary}
Let $(H,m,u,\Delta ,\varepsilon ,\omega )$ be a dual quasi-bialgebra. The
identity functor on ${{_{H}^{H}\mathfrak{M}_{H}^{H}}}$ defines a monoidal
functor $\Xi :({_{H}^{H}\mathfrak{M}_{H}^{H},\square }_{H},H)\rightarrow ({{{%
_{H}^{H}\mathfrak{M}_{H}^{H}}}},\otimes _{H},H).$ For $M,N\in {{{_{H}^{H}%
\mathfrak{M}_{H}^{H}}}}$, the structure morphisms are 
\begin{equation*}
\gamma _{2}(M,N):\Xi (M)\otimes _{H}\Xi (V)\rightarrow \Xi (M\square
_{H}N)\qquad\text{and}\qquad \gamma _{0}:H\rightarrow \Xi (H)
\end{equation*}%
which are defined by $\gamma _{2}(M,N):=\vartheta _{2}^{-1}(M,N)$ and $%
\gamma _{0}:=\vartheta _{0}^{-1}$ using Lemma \ref{equi tensH, cot}.
\end{corollary}

\begin{proof}
It follows by \cite[Proposition 1.4]{AMS-Amonoidal}.
\end{proof}

Next, we include a technical result that will be used in section \ref{C4}.

\begin{lemma}
\label{lax mon}Let $(\mathcal{M}$,$\otimes ,\mathbf{1})$ be a monoidal
category which is abelian.

\begin{enumerate}
\item Let $A$ be an algebra in $\mathcal{M}$. Assume that the tensor
functors are additive and right exact (see \cite[Theorem 1.12]{AMS-Hoch}).
Then the forgetful functor 
\begin{equation*}
D:({{_{A}}}\mathcal{M}{{_{A}}}\,,\otimes _{A},A)\longrightarrow (\mathcal{M}%
,\otimes ,\mathbf{1})
\end{equation*}%
is a lax monoidal functor with structure morphisms 
\begin{equation*}
\zeta _{2}(M,N):D(M)\otimes D(N)\rightarrow D(M\otimes _{A}N)\qquad\text{and}%
\qquad \zeta _{0}:\mathbf{1}\rightarrow D(A),
\end{equation*}%
where $\zeta _{2}$ is the canonical epimorphism and $\zeta _{0}$ is the
unity of $A.$

\item Let $C$ be a coalgebra in $\mathcal{M}$. Assume that the tensor
functors are additive and left exact. Then the forgetful functor 
\begin{equation*}
D:(^{C}\mathcal{M}^{C},\square _{C},C)\longrightarrow (\mathcal{M},\otimes ,%
\mathbf{1})
\end{equation*}%
is a colax monoidal functor with structure morphisms 
\begin{equation*}
\zeta _{2}(M,N):D(M\square _{C}N)\rightarrow D(M)\otimes D(N)\qquad\text{and}%
\qquad \zeta _{0}:D(C\mathbb{)}\rightarrow \mathbf{1},
\end{equation*}%
where $\zeta _{2}$ is the canonical monomorphism and $\zeta _{0}$ is the
counit of $C.$
\end{enumerate}
\end{lemma}

\begin{proof}
1) From \cite[1.11]{AMS-Hoch}, for all $M,N,S\in {{_{A}}}\mathcal{M}{{_{A}}}%
, $ we deduce 
\begin{eqnarray*}
&&D\left( ^{A}a_{M,N,S}^{A}\right) \circ \zeta _{2}(M\otimes _{A}N,S)\circ 
\left[ \zeta _{2}(M,N)\otimes D\left( S\right) \right] \\
&=&\zeta _{2}(M,N\otimes _{A}S)\circ \left[ D\left( M\right) \otimes \zeta
_{2}(N,S)\right] \circ a_{M,N,S}.
\end{eqnarray*}%
Moreover, for all $M\in {{_{A}}}\mathcal{M}{{_{A}}},$ we have%
\begin{eqnarray*}
D\left( ^{A}l_{M}^{A}\right) \circ \zeta _{2}(A,M)\circ \left[ \zeta
_{0}\otimes D\left( M\right) \right] &=&\text{ }^{A}l_{M}^{A}\circ \zeta
_{2}(A,M)\circ \left( \zeta _{0}\otimes M\right) \\
&=&\mu _{M}^{l}\circ \left( u_{A}\otimes M\right) =l_{M}.
\end{eqnarray*}%
Similarly $D\left( ^{A}r_{M}^{A}\right) \circ \zeta _{2}(M,A)\circ \left[
D\left( M\right) \otimes \zeta _{0}\right] =r_{M}.$ We have so proved that $%
D $ is a lax monoidal functor.

2) It follows by dual arguments.
\end{proof}

\section{The main results: bosonization \label{C4}}

\begin{claim}
Let $H$ be a Hopf algebra, let $A$ be a bialgebra and let $\sigma
:H\rightarrow A$ and $\pi :A\rightarrow H$ be morphisms of bialgebras such
that $\pi \sigma =\mathbb{I}_{H}.$ In this case $A$ is called a bialgebra
with projection onto $H$ and $A\in {_{H}^{H}\mathfrak{M}_{H}^{H}}$ through 
\begin{gather*}
\rho ^{r}(a)=a_{1}\otimes \pi (a_{2}),\qquad \rho ^{l}(a)=\pi (a_{1})\otimes
a_{2}, \\
\mu ^{r}(a\otimes h)=a\sigma (h),\qquad \mu ^{l}(h\otimes a)=\sigma (h)a.
\end{gather*}%
Define now the map $\tau :A\rightarrow A:a\longmapsto a_{1}\sigma
S(a_{2}). $ It can be proved that $\mathrm{Im}\tau =A^{coH}=:R$ and, when $%
H$ is the coradical of $A$, that $R$ is connected. Indeed it is well-known
that $R$ becomes a connected bialgebra in the pre-braided monoidal category $%
_{H}^{H}\mathcal{YD}$ of Yetter-Drinfeld modules over $H$ (cf. \cite{Radford}%
).

Now, from the fact that $(F,G)$ is an equivalence we know that $\epsilon
_{A}:R\otimes H\rightarrow A$ is an isomorphism. Conversely, it can be
proved that, given  a Hopf algebra $H$ and a braided bialgebra $R$ in $_{H}^{H}\mathcal{YD}$, we can endow $R\otimes H$ with a bialgebra structure and
define two bialgebras morphisms $\sigma $ and $\pi $ such that $\pi \sigma =%
\mathrm{Id}_{H},$ see (\cite{Radford}). This bialgebra is called \emph{%
Radford-Majid Bosonization }(or Radford biproduct) and permits to classify
different kinds of bialgebras as "compositions" (crossed product) of
different objects in the same category.
\end{claim}

The main aim of this section is to extend the results above to the setting
of dual quasi-bialgebras.

\begin{theorem}
\label{teo: boso}Let $\left( H,m_{H},u_{H},\Delta _{H},\varepsilon
_{H},\omega _{H}\right) $ be a dual quasi-bialgebra.  

Let $(R,\mu
_{R},\rho _{R},\Delta _{R},\varepsilon _{R},m_{R},u_{R})$ be a bialgebra in $%
_{H}^{H}\mathcal{YD}$ and use the following notations%
\begin{eqnarray*}
h\vartriangleright r &:&=\mu _{R}\left( h\otimes r\right) ,\qquad
r_{-1}\otimes r_{0}:=\rho _{R}\left( r\right) , \\
r\cdot _{R}s &:&=m_{R}\left( r\otimes s\right) ,\qquad 1_{R}:=u_{R}\left(
1_{\Bbbk }\right) , \\
r^{1}\otimes r^{2} &:&=\Delta _{R}\left( r\right) .
\end{eqnarray*}

Let us consider on $B:=F\left( R\right) =R\otimes H$ the following
structures:%
\begin{eqnarray*}
m_{B}[(r\otimes h)\otimes (s\otimes k)] &=&\left[ 
\begin{array}{c}
\omega _{H}^{-1}(r_{-2}\otimes h_{1}\otimes s_{-2}k_{1})\omega
_{H}(h_{2}\otimes s_{-1}\otimes k_{2}) \\ 
\omega _{H}^{-1}[(h_{3}\vartriangleright s_{0})_{-2}\otimes h_{4}\otimes
k_{3}]\omega _{H}(r_{-1}\otimes (h_{3}\vartriangleright s_{0})_{-1}\otimes
h_{5}k_{4}) \\ 
r_{0}\cdot _{R}(h_{3}\vartriangleright s_{0})_{0}\otimes h_{6}k_{5}%
\end{array}%
\right] \\
u_{B}(k) &=&k1_{R}\otimes 1_{H} \\
\Delta _{B}(r\otimes h) &=&\omega _{H}^{-1}(r_{-1}^{1}\otimes
r_{-2}^{2}\otimes h_{1})r_{0}^{1}\otimes r_{-1}^{2}h_{2}\otimes
r_{0}^{2}\otimes h_{3} \\
\varepsilon _{B}(r\otimes h) &=&\varepsilon _{R}(r)\varepsilon _{H}(h) \\
\omega _{B}((r\otimes h)\otimes (s\otimes k)\otimes (t\otimes l))
&=&\varepsilon _{R}(r)\varepsilon _{R}(s)\varepsilon _{R}(t)\omega
_{H}(h\otimes k\otimes l).
\end{eqnarray*}%
Then $(B,\Delta _{B},\varepsilon _{B},m_{B},u_{B},\omega _{B})$ is a dual
quasi-bialgebra.
\end{theorem}

\begin{proof}
Recall that, by Lemma \ref{lem: equiv-m4h-tens}, the functor $F:{_{H}^{H}%
\mathcal{YD}}\rightarrow {_{H}^{H}\mathfrak{M}_{H}^{H}}$ \ defines a
monoidal functor $F:({_{H}^{H}\mathcal{YD}},\otimes ,\Bbbk )\rightarrow ({%
_{H}^{H}\mathfrak{M}_{H}^{H},\otimes _{H},H)}$ where, for $U,V\in {_{H}^{H}%
\mathcal{YD}}$, the structure morphisms are given by $\varphi _{2}(U,V),\varphi
_{0}.$ By \cite[Proposition 1.5]{AMS-Amonoidal}, we have that $\left(
B,m_{B}^{\prime },u_{B}^{\prime }\right) $ is an algebra in $({_{H}^{H}%
\mathfrak{M}_{H}^{H},\otimes _{H},H)}$ where%
\begin{equation*}
m_{B}^{\prime }:=F\left( m_{R}\right) \circ \varphi _{2}(R,R),\qquad
u_{B}^{\prime }:=F\left( u_{R}\right) \circ \varphi _{0}.
\end{equation*}%
Explicitly we have%
\begin{eqnarray*}
m_{B}^{\prime }\left( \left( r\otimes h\right) \otimes _{H}\left( s\otimes
k\right) \right) &=&\left[ 
\begin{array}{c}
\omega _{H}^{-1}\left( r_{-2}\otimes h_{1}\otimes s_{-2}k_{1}\right) \omega
_{H}(h_{2}\otimes s_{-1}\otimes k_{2}) \\ 
\omega _{H}^{-1}((h_{3}\vartriangleright s_{0})_{-2}\otimes h_{4}\otimes
k_{3})\omega _{H}\left( r_{-1}\otimes (h_{3}\vartriangleright
s_{0})_{-1}\otimes h_{5}k_{4})\right) \\ 
r_{0}\cdot _{R}(h_{3}\vartriangleright s_{0})_{0}\otimes h_{6}k_{5}%
\end{array}%
\right] \\
&=&m_{B}[(r\otimes h)\otimes (s\otimes k)],
\end{eqnarray*}%
\begin{equation*}
u_{B}^{\prime }\left( h\right) =u_{R}\left( 1_{\Bbbk }\right) \otimes
h=1_{R}\otimes h.
\end{equation*}%
Since $m_{B}^{\prime }$ is associative in $({_{H}^{H}\mathfrak{M}%
_{H}^{H},\otimes _{H},H)}$, we have that 
\begin{equation*}
m_{B}^{\prime }\circ \left( m_{B}^{\prime }\otimes _{H}B\right)
=m_{B}^{\prime }\circ \left( B\otimes _{H}m_{B}^{\prime }\right) \circ
a_{B,B,B}
\end{equation*}%
where $a_{B,B,B}$ is the one defined in Lemma \ref{lem:m4hmonoidal}. Let $%
\pi :B\rightarrow H$ be defined by $\pi \left( r\otimes h\right)
:=\varepsilon _{R}\left( r\right) h.$ Then%
\begin{equation}
\omega _{H}\left( \pi \otimes \pi \otimes \pi \right) =\omega _{B}.
\label{form: Omegapi}
\end{equation}%
One easily gets that%
\begin{equation}
\pi \left( x_{1}\right) \otimes x_{2}\otimes \pi \left( x_{3}\right)
=x_{-1}\otimes x_{0}\otimes x_{1},\text{ for all }x\in B\text{.}
\label{form: assB}
\end{equation}

Let $x,y,z\in B$, then 
\begin{eqnarray*}
m_{B}^{\prime }\left( m_{B}^{\prime }\otimes _{H}B\right) \left( \left(
x\otimes _{H}y\right) \otimes _{H}z\right) &=&m_{B}\left( m_{B}\otimes
B\right) \left( \left( x\otimes y\right) \otimes z\right) 
\end{eqnarray*}%
and 
\begin{eqnarray*}
&&m_{B}^{\prime }\left( B\otimes _{H}m_{B}^{\prime }\right) a_{B,B,B}\left(
\left( x\otimes _{H}y\right) \otimes _{H}z\right) \\
&=&\omega _{H}^{-1}\left( x_{-1}\otimes y_{-1}\otimes z_{-1}\right)
m_{B}^{\prime }\left( B\otimes _{H}m_{B}^{\prime }\right) \left(
x_{0}\otimes _{H}\left( y_{0}\otimes _{H}z_{0}\right) \right) \omega
_{H}\left( x_{1}\otimes y_{1}\otimes z_{1}\right) \\
&=&\omega _{H}^{-1}\left( x_{-1}\otimes y_{-1}\otimes z_{-1}\right)
m_{B}\left( B\otimes m_{B}\right) \left( x_{0}\otimes \left( y_{0}\otimes
z_{0}\right) \right) \omega _{H}\left( x_{1}\otimes y_{1}\otimes z_{1}\right)
\\
&\overset{(\ref{form: assB})}{=}&\omega _{H}^{-1}\left( \pi \left(
x_{1}\right) \otimes \pi \left( y_{1}\right) \otimes \pi \left( z_{1}\right)
\right) m_{B}\left( B\otimes m_{B}\right) \left( x_{2}\otimes \left(
y_{2}\otimes z_{2}\right) \right) \omega _{H}\left( \pi \left( x_{3}\right)
\otimes \pi \left( y_{3}\right) \otimes \pi \left( z_{3}\right) \right) \\
&\overset{(\ref{form: Omegapi})}{=}&\omega _{B}^{-1}\left( x_{1}\otimes
y_{1}\otimes z_{1}\right) m_{B}\left( B\otimes m_{B}\right) \left(
x_{2}\otimes \left( y_{2}\otimes z_{2}\right) \right) \omega _{B}\left(
x_{3}\otimes y_{3}\otimes z_{3}\right) \\
&=&\left[ \omega _{B}^{-1}\ast \left[ m_{B}\left( B\otimes m_{B}\right) %
\right] \ast \omega _{B}\right] \left( \left( x\otimes y\right) \otimes
z\right)
\end{eqnarray*}%
so that $m_{B}\left( m_{B}\otimes B\right) =\omega _{B}^{-1}\ast \left[
m_{B}\left( B\otimes m_{B}\right) \right] \ast \omega _{B}$.

Since $m_{B}^{\prime }$ is unitary in $({_{H}^{H}\mathfrak{M}%
_{H}^{H},\otimes _{H},H)}$, we have that $m_{B}^{\prime }\left(
u_{B}^{\prime }\otimes _{H}B\right) =l_{B}$. From this equality, we get 
$m_{B}\left( u_{B}\otimes B\right) =l_{B}.$ Similarly $m_{B}\left( B\otimes
u_{B}\right) =r_{B}.$ Let us recall that, by Lemma \ref{lem: F monoidal cot}%
, the functor $F:{_{H}^{H}\mathcal{YD}}\rightarrow {_{H}^{H}\mathfrak{M}%
_{H}^{H}}$ \ defines a monoidal functor $F:({_{H}^{H}\mathcal{YD}},\otimes
,\Bbbk )\rightarrow ({_{H}^{H}\mathfrak{M}_{H}^{H},\square }_{H},H)$ , with
structure morphisms $\psi _{2}(U,V),\psi _{0},$ with $U,V\in {_{H}^{H}%
\mathcal{YD}}.$ By \cite[Proposition 1.5]{AMS-Amonoidal}, we have that $%
\left( B,\overline{\Delta }_{B},\overline{\varepsilon }_{B}\right) $ is a
coalgebra in $({_{H}^{H}\mathfrak{M}_{H}^{H},\square _{H},H)}$ where%
\begin{equation*}
\overline{\Delta }_{B}:=\psi _{2}(R,R)^{-1}\circ F\left( \Delta _{R}\right)
,\qquad \overline{\varepsilon }_{B}:=\psi _{0}^{-1}\circ F\left( \varepsilon
_{R}\right) .
\end{equation*}%
Explicitly we have%
\begin{eqnarray*}
\overline{\Delta }_{B}\left( r\otimes h\right) &=&\psi _{2}(R,R)^{-1}\left(
\left( r^{1}\otimes r^{2}\right) \otimes h\right) \\
&=&\omega ^{-1}(r_{-1}^{1}\otimes r_{-2}^{2}\otimes h_{1})(r_{0}^{1}\otimes
r_{-1}^{2}h_{2})\square _{H}(r_{0}^{2}\otimes h_{3}) \\
&=&\Delta _{B}(r\otimes h),
\end{eqnarray*}%
and%
\begin{equation*}
\overline{\varepsilon }_{B}\left( r\otimes h\right) =\psi _{0}^{-1}\left(
\varepsilon _{R}\left( r\right) \otimes h\right) =\varepsilon _{R}\left(
r\right) h.
\end{equation*}%
From the fact that $\left( B,\overline{\Delta }_{B},\overline{\varepsilon }%
_{B}\right) $ is a coalgebra in $({_{H}^{H}\mathfrak{M}_{H}^{H},\square
_{H},H)}$ one easily gets that $\left( B,\Delta _{B},\varepsilon _{B}\right) 
$ is an ordinary coalgebra.

It is straightforward to prove that $\pi $ is multiplicative,
comultiplicative, counitary and unitary i.e. 
\begin{equation}
\pi m_{B}=m_{H}\left( \pi \otimes \pi \right) ,\text{\quad }\left( \pi
\otimes \pi \right) \Delta _{B}=\Delta _{H}\pi ,\text{\quad }\varepsilon
_{B}=\varepsilon _{H}\pi ,\text{\quad }\pi u_{B}=u_{H}.  \label{form: piunit}
\end{equation}%
Using these equalities plus (\ref{form: Omegapi}), one easily gets that the
cocycle and unitary conditions for $\omega _{B}$ follow from the ones of $%
\omega _{H}$.

Now we want to prove that $m_{B}$ is a morphism of coalgebras. It is
counitary as%
\begin{equation*}
\varepsilon _{B}m_{B}\overset{(\ref{form: piunit})}{=}\varepsilon _{H}\pi
m_{B}\overset{(\ref{form: piunit})}{=}\varepsilon _{H}m_{H}\left( \pi
\otimes \pi \right) =m_{\Bbbk }\left( \varepsilon _{H}\otimes \varepsilon
_{H}\right) \left( \pi \otimes \pi \right) \overset{(\ref{form: piunit})}{=}%
m_{\Bbbk }\left( \varepsilon _{B}\otimes \varepsilon _{B}\right) .
\end{equation*}%
Hence we just have to prove that%
\begin{equation*}
\Delta _{B}[(r\otimes h)\cdot _{B}(s\otimes k)]=(r\otimes h)_{1}\cdot
_{B}(s\otimes k)_{1}\otimes (r\otimes h)_{2}\cdot _{B}(s\otimes k)_{2},
\end{equation*}%
where $x\cdot _{B}y:=m_{B}\left( x\otimes y\right) $ and $x_{1}\otimes
x_{2}:=\Delta _{B}\left( x\right) ,$ for all $x,y\in B.$ Equivalently we
will prove that%
\begin{equation*}
\Delta _{B}m_{B}=\left( m_{B}\otimes m_{B}\right) \Delta _{B\otimes B}.
\end{equation*}

Since ${_{H}^{H}\mathcal{YD}}$ is a pre-braided monoidal category and $%
(R,\Delta _{R},\varepsilon _{R})$ is a coalgebra in this category, then we
can define two morphisms $\Delta _{R\otimes R}$ and $\varepsilon _{R\otimes
R}$ in ${_{H}^{H}\mathcal{YD}}$ such that $(R\otimes R,\Delta _{R\otimes
R},\varepsilon _{R\otimes R})$ is a coalgebra in ${_{H}^{H}\mathcal{YD}}$
too. We have:%
\begin{eqnarray*}
\Delta _{R\otimes R} &:&=a_{R,R,R\otimes R}^{-1}\circ (R\otimes
a_{R,R,R})\circ (R\otimes \left( c_{R,R}\otimes R\right) )\circ (R\otimes
a_{R,R,R}^{-1})\circ a_{R,R,R\otimes R}\circ (\Delta _{R}\otimes \Delta
_{R}), \\
\varepsilon _{R\otimes R} &:&=\varepsilon _{R}\otimes \varepsilon _{R}.
\end{eqnarray*}%
Explicitly we obtain%
\begin{eqnarray}
\label{form:DeltaRotR}\Delta _{R\otimes R}(r\otimes s) &=&\left[ 
\begin{array}{c}
\omega ^{-1}(r_{-2}^{1}\otimes r_{-5}^{2}\otimes s_{-2}^{1}s_{-4}^{2})\omega
(r_{-4}^{2}\otimes s_{-1}^{1}\otimes s_{-3}^{2}) \\ 
\omega ^{-1}[(r_{-3}^{2}\vartriangleright s_{0}^{1})_{-2}\otimes
r_{-2}^{2}\otimes s_{-2}^{2}) \\ 
\omega (r_{-1}^{1}\otimes (r_{-3}^{2}\vartriangleright
s_{0}^{1})_{-1}\otimes r_{-1}^{2}s_{-1}^{2}) \\ 
\left[ r_{0}^{1}\otimes (r_{-3}^{2}\vartriangleright s_{0}^{1})_{0}\right]
\otimes \left( r_{0}^{2}\otimes s_{0}^{2}\right)%
\end{array}%
\right] , \\
\notag\varepsilon _{R\otimes R}(r\otimes s) &:&=\varepsilon _{R}\left( r\right)
\varepsilon _{R}\left( s\right) .
\end{eqnarray}%
Consider the canonical maps%
\begin{equation*}
j_{M,N} :M\square _{H}N\rightarrow M\otimes N\qquad\text{and}\qquad \chi
_{M,N} :M\otimes N\rightarrow M\otimes _{H}N,
\end{equation*}%
for all $M,N\in {_{H}^{H}\mathfrak{M}_{H}^{H}}$. Set 
\begin{eqnarray*}
\widehat{\Delta _{R}m_{R}} &:&=j_{F\left( R\right) ,F\left( R\right) }\circ
\psi _{2}\left( R,R\right) ^{-1}\circ F\left( \Delta _{R}m_{R}\right) \circ
\varphi _{2}\left( R,R\right) \circ \chi _{F\left( R\right) ,F\left(
R\right) }, \\
\widehat{\Delta _{R}} &:&=j_{F\left( R\right) ,F\left( R\right) }\circ \psi
_{2}\left( R,R\right) ^{-1}\circ F\left( \Delta _{R}\right) , \\
\widehat{m_{R}} &:&=F\left( m_{R}\right) \circ \varphi _{2}\left( R,R\right)
\circ \chi _{F\left( R\right) ,F\left( R\right) }, \\
\widehat{\left( m_{R}\otimes m_{R}\right) \Delta _{R\otimes R}}
&:&=j_{F\left( R\right) ,F\left( R\right) }\circ \psi _{2}\left( R,R\right)
^{-1}\circ F\left( \left( m_{R}\otimes m_{R}\right) \Delta _{R\otimes
R}\right) \circ \varphi _{2}\left( R,R\right) \circ \chi _{F\left( R\right)
,F\left( R\right) }.
\end{eqnarray*}

We have%
\begin{eqnarray*}
\widehat{\Delta _{R}m_{R}} &=&j_{F\left( R\right) ,F\left( R\right) }\circ
\psi _{2}\left( R,R\right) ^{-1}\circ F\left( \Delta _{R}\right) \circ
F\left( m_{R}\right) \circ \varphi _{2}\left( R,R\right) \circ \chi
_{F\left( R\right) ,F\left( R\right) } \\
&=&\widehat{\Delta _{R}}\circ \widehat{m_{R}}.
\end{eqnarray*}%
Moreover%
\begin{eqnarray*}
\widehat{\Delta _{R}} &=&j_{F\left( R\right) ,F\left( R\right) }\circ 
\overline{\Delta }_{B}=\Delta _{B}, \\
\widehat{m_{R}} &=&m_{B}^{\prime }\circ \chi _{F\left( R\right) ,F\left(
R\right) }=m_{B},
\end{eqnarray*}%
so that, since $\left( m_{R}\otimes m_{R}\right) \Delta _{R\otimes R}=\Delta
_{R}m_{R},$ we obtain 
\begin{equation}
\widehat{\left( m_{R}\otimes m_{R}\right) \Delta _{R\otimes R}}=\widehat{%
\Delta _{R}m_{R}}=\Delta _{B}m_{B}.  \label{prima m coal}
\end{equation}%
It remains to prove that 
\begin{equation}
\widehat{\left( m_{R}\otimes m_{R}\right) \Delta _{R\otimes R}}=\left(
m_{B}\otimes m_{B}\right) \Delta _{B\otimes B}.  \label{form:starDavid}
\end{equation}

First, one checks that $\left( m_{B}\otimes m_{B}\right) \Delta _{B\otimes
B} $ is $H$-balanced. 
Hence there is a unique map $\zeta :B\otimes _{H}B\rightarrow B\otimes B$
such that 
\begin{equation*}
\zeta \circ \chi _{F\left( R\right) ,F\left( R\right) }=\left( m_{B}\otimes
m_{B}\right) \Delta _{B\otimes B}.
\end{equation*}%
Our aim is to prove that \eqref{form:starDavid} holds %
i.e. that%
\begin{equation*}
j_{F\left( R\right) ,F\left( R\right) }\circ \psi _{2}\left( R,R\right)
^{-1}\circ F\left( \left( m_{R}\otimes m_{R}\right) \Delta _{R\otimes
R}\right) \circ \varphi _{2}\left( R,R\right) \circ \chi _{F\left( R\right)
,F\left( R\right) }=\zeta \circ \chi _{F\left( R\right) ,F\left( R\right) }.
\end{equation*}%
Since $\chi _{F\left( R\right) ,F\left( R\right) }$ is an epimorphism, the
latter displayed equality is equivalent to%
\begin{equation}
j_{F\left( R\right) ,F\left( R\right) }\circ \psi _{2}\left( R,R\right)
^{-1}\circ F\left( \left( m_{R}\otimes m_{R}\right) \Delta _{R\otimes
R}\right) =\zeta \circ \varphi _{2}\left( R,R\right) ^{-1}.
\label{form: vale}
\end{equation}%
Now%
\begin{eqnarray*}
\zeta \left( x\otimes _{H}y\right) &=&\zeta \circ \chi _{F\left( R\right)
,F\left( R\right) }\left( x\otimes y\right) =\left( m_{B}\otimes
m_{B}\right) \Delta _{B\otimes B}\left( x\otimes y\right) \\
&=&x_{1}\cdot _{B}y_{1}\otimes x_{2}\cdot _{B}y_{2}.
\end{eqnarray*}%
%
%
%
%
%
%
%
%
%
%
One proves that $\zeta \left( x\otimes _{H}y\right) \in B\square _{H}B.$
Then there is a unique map $\zeta ^{\prime }:B\otimes _{H}B\rightarrow
B\square _{H}B$ such that $j_{F\left( R\right) ,F\left( R\right) }\circ
\zeta ^{\prime }=\zeta $. Hence (\ref{form: vale}) is equivalent to%
\begin{equation*}
j_{F\left( R\right) ,F\left( R\right) }\circ \psi _{2}\left( R,R\right)
^{-1}\circ F\left( \left( m_{R}\otimes m_{R}\right) \Delta _{R\otimes
R}\right) =j_{F\left( R\right) ,F\left( R\right) }\circ \zeta ^{\prime
}\circ \varphi _{2}\left( R,R\right) ^{-1}
\end{equation*}%
i.e. to%
\begin{equation}
F\left( \left( m_{R}\otimes m_{R}\right) \Delta _{R\otimes R}\right) =\psi
_{2}\left( R,R\right) \circ \zeta ^{\prime }\circ \varphi _{2}\left(
R,R\right) ^{-1}.  \label{form: vale2}
\end{equation}%
By construction%
\begin{equation*}
\zeta ^{\prime }\left( x\otimes _{H}y\right) =x_{1}\cdot _{B}y_{1}\square
_{H}x_{2}\cdot _{B}y_{2}.
\end{equation*}%

It is straightforward to prove that $\zeta ^{\prime }$ is right $H$-linear.
Thus it suffices to check that (\ref{form: vale2}) holds on elements of the
form $\left( r\otimes s\right) \otimes 1_{H}.$ Thus, for $r,s\in R,h\in H$%
\begin{eqnarray*}
&&\left[ \psi _{2}\left( R,R\right) \circ \zeta ^{\prime }\circ \varphi
_{2}\left( R,R\right) ^{-1}\right] \left( \left( r\otimes s\right) \otimes
1_{H}\right) \\
&=&\psi _{2}\left( R,R\right) \left[ \left( r\otimes 1_{H}\right) _{1}\cdot
_{B}\left( s\otimes 1_{H}\right) _{1}\square _{H}\left( r\otimes
1_{H}\right) _{2}\cdot _{B}\left( s\otimes 1_{H}\right) _{2}\right] \\
&=&\psi _{2}\left( R,R\right) \left[ \left( r^{1}\otimes r_{-1}^{2}\right)
\cdot _{B}\left( s^{1}\otimes s_{-1}^{2}\right) \otimes \left(
r_{0}^{2}\otimes 1_{H}\right) \cdot _{B}\left( s_{0}^{2}\otimes 1_{H}\right) %
\right] \\
&=&\psi _{2}\left( R,R\right) \left[ \left( r^{1}\otimes r_{-1}^{2}\right)
\cdot _{B}\left( s^{1}\otimes s_{-1}^{2}\right) \otimes \left(
r_{0}^{2}\cdot _{R}s_{0}^{2}\otimes 1_{H}\right) \right] \\
&=&\left[ 
\begin{array}{c}
\omega _{H}^{-1}(\left( r^{1}\right) _{-2}\otimes \left( r_{-1}^{2}\right)
_{1}\otimes \left( s^{1}\right) _{-2}\left( s_{-1}^{2}\right) _{1})\omega
_{H}(\left( r_{-1}^{2}\right) _{2}\otimes \left( s^{1}\right) _{-1}\otimes
\left( s_{-1}^{2}\right) _{2}) \\ 
\omega _{H}^{-1}[(\left( r_{-1}^{2}\right) _{3}\vartriangleright \left(
s^{1}\right) _{0})_{-2}\otimes \left( r_{-1}^{2}\right) _{4}\otimes \left(
s_{-1}^{2}\right) _{3}]\\\omega _{H}(\left( r^{1}\right) _{-1}\otimes (\left(
r_{-1}^{2}\right) _{3}\vartriangleright \left( s^{1}\right)
_{0})_{-1}\otimes \left( r_{-1}^{2}\right) _{5}\left( s_{-1}^{2}\right) _{4})
\\ 
\psi _{2}\left( R,R\right) \left[ \left( r^{1}\right) _{0}\cdot _{R}(\left(
r_{-1}^{2}\right) _{3}\vartriangleright \left( s^{1}\right) _{0})_{0}\otimes
\left( r_{-1}^{2}\right) _{6}\left( s_{-1}^{2}\right) _{5}\otimes \left(
r_{0}^{2}\cdot _{R}s_{0}^{2}\otimes 1_{H}\right) \right]%
\end{array}%
\right] \\
&=&\left[ 
\begin{array}{c}
\omega _{H}^{-1}(\left( r^{1}\right) _{-2}\otimes \left( r_{-1}^{2}\right)
_{1}\otimes \left( s^{1}\right) _{-2}\left( s_{-1}^{2}\right) _{1})\omega
_{H}(\left( r_{-1}^{2}\right) _{2}\otimes \left( s^{1}\right) _{-1}\otimes
\left( s_{-1}^{2}\right) _{2}) \\ 
\omega _{H}^{-1}[(\left( r_{-1}^{2}\right) _{3}\vartriangleright \left(
s^{1}\right) _{0})_{-2}\otimes \left( r_{-1}^{2}\right) _{4}\otimes \left(
s_{-1}^{2}\right) _{3}]\\\omega _{H}(\left( r^{1}\right) _{-1}\otimes (\left(
r_{-1}^{2}\right) _{3}\vartriangleright \left( s^{1}\right)
_{0})_{-1}\otimes \left( r_{-1}^{2}\right) _{5}\left( s_{-1}^{2}\right) _{4})
\\ 
\left[ \left( r^{1}\right) _{0}\cdot _{R}(\left( r_{-1}^{2}\right)
_{3}\vartriangleright \left( s^{1}\right) _{0})_{0}\otimes \left(
r_{0}^{2}\cdot _{R}s_{0}^{2}\otimes 1_{H}\right) \right]%
\end{array}%
\right] \\
&=&\left[ 
\begin{array}{c}
\omega _{H}^{-1}(r_{-2}^{1}\otimes r_{-5}^{2}\otimes
s_{-2}^{1}s_{-4}^{2})\omega _{H}(r_{-4}^{2}\otimes s_{-1}^{1}\otimes
s_{-3}^{2}) \\ 
\omega _{H}^{-1}[(r_{-3}^{2}\vartriangleright s_{0}^{1})_{-2}\otimes
r_{-2}^{2}\otimes s_{-2}^{2}] \\ 
\omega _{H}(r_{-1}^{1}\otimes (r_{-3}^{2}\vartriangleright
s_{0}^{1})_{-1}\otimes r_{-1}^{2}s_{-1}^{2}) \\ 
\left[ r_{0}^{1}\cdot _{R}(r_{-3}^{2}\vartriangleright s_{0}^{1})_{0}\otimes
\left( r_{0}^{2}\cdot _{R}s_{0}^{2}\otimes 1_{H}\right) \right]%
\end{array}%
\right] \\
&=&\left[ \left( m_{R}\otimes m_{R}\right) \Delta _{R\otimes R}\left(
r\otimes s\right) \right] \otimes 1_{H} \\
&=&F\left( \left( m_{R}\otimes m_{R}\right) \Delta _{R\otimes R}\right)
\left( \left( r\otimes s\right) \otimes 1_{H}\right) .
\end{eqnarray*}%
Hence we have proved that (\ref{form: vale2}) holds and hence (\ref%
{form:starDavid}) is fulfilled. Thus, from (\ref{prima m coal}), we can
conclude that $m_{B}$ is a coalgebra morphism. Finally, it is easy to prove
that $u_{B}$ is a coalgebra map. 
\end{proof}

\begin{definition}
With hypotheses and notations as in Theorem \ref{teo: boso}, the bialgebra $%
B $ will be called the \emph{bosonization of }$R$\emph{\ by }$H$ and denoted
by $R\#H$.
\end{definition}

\begin{definition}
Let $(H,m,u,\Delta ,\varepsilon ,\omega )$ and $(A,m_{A},u_{A},\Delta
_{A},\varepsilon _{A},\omega _{A})$ be dual quasi-bialgebras, and suppose
there exist morphisms of dual quasi-bialgebras 
\begin{equation*}
\sigma :H\rightarrow A\qquad \text{and}\qquad \pi :A\rightarrow H
\end{equation*}%
such that $\pi \sigma =\mathrm{Id}_{H}$. Then $(A,H,\sigma ,\pi )$ is called
a \emph{dual quasi-bialgebra with a projection onto }$H.$
\end{definition}

\begin{proposition}
Keep the hypotheses and notations of Theorem \ref{teo: boso}. Then $%
(R\#H,H,\sigma ,\pi )$ is a dual quasi-bialgebra with projection onto $H$
where%
\begin{equation*}
\sigma :H\rightarrow R\#H,\sigma \left( h\right) :=1_{R}\#h,\qquad \pi
:R\#H\rightarrow H,\pi \left( r\#h\right) :=\varepsilon _{R}\left( r\right)
h.
\end{equation*}
\end{proposition}

\begin{proof}
The proof that $\sigma$ is a morphism of dual quasi-bialgebras is
straightforward.

%
The map $\pi $ is a morphism of dual quasi-bialgebras in view of (\ref{form:
Omegapi}) and (\ref{form: piunit}). Finally, we have $\pi \sigma \left(
h\right) =\pi \left( 1_{R}\#h\right) =\varepsilon _{R}\left( 1_{R}\right)
h=h $.
\end{proof}

Next aim is to characterize dual quasi-bialgebras with a projection onto a
dual quasi-bialgebra with a preantipode as bosonizations.

\begin{lemma}
\label{A in M tre H}Let $\left( A,m_{A},u_{A},\Delta _{A},\varepsilon
_{A},\omega _{A}\right) $ and $\left( H,m_{H},u_{H},\Delta _{H},\varepsilon
_{H},\omega _{H}\right) $ be dual quasi-bialgebras such that $(A,H,\sigma
,\pi )$ is a dual quasi-bialgebra with a projection onto $H$. Then $A$ is an
object in ${_{H}^{H}\mathfrak{M}_{H}^{H}}\ $through 
\begin{align*}
\rho _{A}^{r}(a)& =a_{1}\otimes \pi (a_{2}),\qquad \rho _{A}^{l}(a)=\pi
(a_{1})\otimes a_{2}, \\
\mu _{A}^{r}(a\otimes h)& =a\sigma (h),\qquad \mu _{A}^{l}\left( h\otimes
a\right) =\sigma \left( h\right) a.
\end{align*}
\end{lemma}

\begin{proof}
It is straightforward. %
\end{proof}

\begin{theorem}
\label{teo: projection}Let $\left( A,m_{A},u_{A},\Delta _{A},\varepsilon
_{A},\omega _{A}\right) $ and $\left( H,m_{H},u_{H},\Delta _{H},\varepsilon
_{H},\omega _{H}\right) $ be dual quasi-bialgebras such that $(A,H,\sigma
,\pi )$ is a dual quasi-bialgebra with projection onto $H$. Assume that $H$
has a preantipode $S$. For all $a,b\in A,$ we set $a_{1}\otimes
a_{2}:=\Delta _{A}\left( a\right) $ and $ab=m_{A}\left( a\otimes b\right) $.
Then, for all $a\in A$ we have%
\begin{equation*}
\tau (a):=\omega _{A}[a_{1}\otimes \sigma S\pi (a_{3})_{1}\otimes
a_{4}]a_{2}\sigma S\pi (a_{3})_{2}
\end{equation*}%
and $R:=G\left( A\right) $ is a bialgebra $\left( \left( R,\mu _{R},\rho
_{R}\right) ,m_{R},u_{R},\Delta _{R},\varepsilon _{R},\omega _{R}\right) $
in ${_{H}^{H}\mathcal{YD}}$ where, for all $r,s\in R,h\in H,k\in \Bbbk ,\,$%
we have%
\begin{gather*}
h\vartriangleright r:=\mu _{R}\left( h\otimes r\right) :=\tau \left[ \sigma
\left( h\right) r\right] ,\qquad r_{-1}\otimes r_{0}:=\rho _{R}\left(
r\right) :=\pi \left( r_{1}\right) \otimes r_{2}, \\
m_{R}\left( r\otimes s\right) :=rs,\qquad u_{R}\left( k\right) :=k1_{A}, \\
r^{1}\otimes r^{2}:=\Delta _{R}\left( r\right) :=\tau \left( r_{1}\right)
\otimes \tau \left( r_{2}\right) ,\qquad \varepsilon _{R}\left( r\right)
:=\varepsilon _{A}\left( r\right) .
\end{gather*}%
Moreover there is a dual quasi-bialgebra isomorphism $\epsilon
_{A}:R\#H\rightarrow A$ given by 
\begin{equation*}
\epsilon _{A}\left( r\otimes h\right) =r\sigma \left( h\right) ,\qquad
\epsilon _{A}^{-1}\left( a\right) =\tau \left( a_{1}\right) \otimes \pi
\left( a_{2}\right) .
\end{equation*}
\end{theorem}

\begin{proof}
We have%
\begin{equation*}
\rho _{A}^{r}\left( a_{1}\right) \otimes a_{2}=a_{1}\otimes \pi
(a_{2})\otimes a_{3}=a_{1}\otimes \rho _{A}^{l}(a_{2})
\end{equation*}%
so that $\Delta _{A}\left( a\right) \in A{{\square _{H}A}}$ for all $a\in A$%
. Let $\overline{\Delta }_{A}:A\rightarrow A{{\square _{H}A}}$ be the
corestriction of $\Delta _{A}$ to $A{{\square _{H}A}}$. Using that $\omega
_{H}=\omega _{A}\left( \pi \otimes \pi \otimes \pi \right) ,$ we obtain%
\begin{equation*}
m_{A}\circ \left( A\otimes \mu _{A}^{l}\right) \circ {^{H}}%
a_{A,H,A}^{H}=m_{A}\circ \left( \mu _{A}^{r}\otimes A\right) .
\end{equation*}%
Denote by $\chi _{X,Y}:X\otimes Y\rightarrow X\otimes _{H}Y$ the canonical
projection, for all $X,Y$ objects in ${_{H}^{H}\mathfrak{M}_{H}^{H}.}$

Since $\left( A\otimes _{H}A,\chi _{A,A}\right) $ is the coequalizer of $%
\left( \left( A\otimes \mu _{A}^{l}\right) {^{H}}a_{A,H,A}^{H},\left( \mu
_{A}^{r}\otimes A\right) \right) ,$ we get that $m_{A}$ quotient to a map $%
m_{A}^{\prime }:A\otimes _{H}A\rightarrow A$ such that $m_{A}^{\prime }\circ
\chi _{A,A}=m_{A}.$ Consider the canonical map $\vartheta _{2}(M,N):M\square
_{H}N\rightarrow M\otimes _{H}N$ of Lemma \ref{equi tensH, cot} defined by $%
\vartheta _{2}(M,N)(m\square _{H}n):=\tau (m)\otimes _{H}n$ and let $%
\overline{m}_{A}:=m_{A}^{\prime }\circ \vartheta _{2}(A,A)$. Then 
\begin{equation*}
\overline{m}_{A}\left( a\square _{H}b\right) =m_{A}^{\prime }\left( \tau
(a)\otimes _{H}b\right) =\tau (a)b.
\end{equation*}%
Note that, by Lemma \ref{lem: tau}, the map $\tau :A\rightarrow A^{coH}$ is
defined, for all $a\in A$, by 
\begin{eqnarray*}
\tau (a) &=&\omega _{H}[a_{-1}\otimes S(a_{1})_{1}\otimes
a_{2}]a_{0}S(a_{1})_{2} \\
&=&\omega _{H}[\pi \left( a_{1}\right) \otimes S\pi (a_{3})_{1}\otimes \pi
\left( a_{4}\right) ]a_{2}\sigma \left[ S\pi (a_{3})_{2}\right] \\
&=&\omega _{H}[\pi \left( a_{1}\right) \otimes \pi \sigma \left[ S\pi
(a_{3})_{1}\right] \otimes \pi \left( a_{4}\right) ]a_{2}\sigma \left[ S\pi
(a_{3})_{2}\right] \\
&=&\omega _{A}[a_{1}\otimes \sigma \left[ S\pi (a_{3})_{1}\right] \otimes
a_{4}]a_{2}\sigma \left[ S\pi (a_{3})_{2}\right] \\
&=&\omega _{A}[a_{1}\otimes \sigma S\pi (a_{3})_{1}\otimes a_{4}]a_{2}\sigma
S\pi (a_{3})_{2}.
\end{eqnarray*}%
It is straightforward to prove that $\left( A,\overline{\Delta }_{A},%
\overline{\varepsilon }_{A}:=\pi \right) $ is a coalgebra in $({{_{H}^{H}%
\mathfrak{M}_{H}^{H},\square _{H},H).}}$

One checks that $\left( A,m_{A}^{\prime },\sigma \right) $ is an algebra in $%
({{_{H}^{H}\mathfrak{M}_{H}^{H},\otimes }}_{H}{,H).}$

Now, by applying \cite[Proposition 1.5]{AMS-Amonoidal} to the monoidal
functor $E:({{{_{H}^{H}\mathfrak{M}_{H}^{H}}}},\otimes _{H},H)\rightarrow ({%
_{H}^{H}\mathfrak{M}_{H}^{H},\square }_{H},H)$ of Lemma \ref{equi tensH, cot}
we have that $\left( E\left( A\right) ,m_{E\left( A\right) },u_{E\left(
A\right) }\right) $ is an algebra in $({_{H}^{H}\mathfrak{M}_{H}^{H},\square 
}_{H},H)$ where%
\begin{equation*}
m_{E\left( A\right) }=E\left( m_{A}^{\prime }\right) \circ \vartheta
_{2}(A,A)\qquad \text{and}\qquad u_{E\left( A\right) }=E\left( \sigma
\right) \circ \vartheta _{0}.
\end{equation*}%
It is clear that $\left( E\left( A\right) ,m_{E\left( A\right) },u_{E\left(
A\right) }\right) =(A,\overline{m}_{A},\overline{u}_{A}=\sigma ).$ Thus $%
\left( A,\overline{m}_{A},\overline{u}_{A}\right) $ is an algebra in $({{%
_{H}^{H}\mathfrak{M}_{H}^{H},\square _{H}},H).}$ %
%

Now, we apply \cite[Proposition 1.5]{AMS-Amonoidal} to the functor $G:{%
_{H}^{H}\mathfrak{M}_{H}^{H}}\rightarrow {_{H}^{H}\mathcal{YD}}$ of Lemma %
\ref{lem: G monoidal}. Set $R:=G\left( A\right) =A^{coH}$. Then $R$ is both
an algebra and a coalgebra in ${_{H}^{H}\mathcal{YD}}$ through%
\begin{eqnarray*}
m_{R} &:&=G\left( \overline{m}_{A}\right) \circ \psi _{2}^{G}(A,A),\qquad
u_{R}:=G\left( \overline{u}_{A}\right) \circ \psi _{0}^{G}, \\
\Delta _{R} &:&=\psi _{2}^{G}(A,A)^{-1}\circ G\left( \overline{\Delta }%
_{A}\right) ,\qquad \varepsilon _{R}:=\left( \psi _{0}^{G}\right) ^{-1}\circ
G\left( \overline{\varepsilon }_{A}\right) .
\end{eqnarray*}%
Explicitly, for all $r,s\in R,k\in \Bbbk $%
\begin{equation*}
m_{R}\left( r\otimes s\right) =\tau \left( rs_{-1}\right) s_{0}\overset{(\ref%
{Tau mh simple})}{=}r\varepsilon _{H}\left( s_{-1}\right) s_{0}=rs,
\end{equation*}%
\begin{equation*}
u_{R}\left( k\right) =G\left( \overline{u}_{A}\right) \psi _{0}^{G}\left(
k\right) =\overline{u}_{A}\left( k1_{H}\right) =k\sigma \left( 1_{H}\right)
=k1_{A},
\end{equation*}%
\begin{equation*}
\Delta _{R}\left( r\right) =\tau \left( r_{1}\right) \otimes \tau \left(
r_{2}\right) ,
\end{equation*}%
\begin{equation*}
\varepsilon _{R}\left( r\right) =\left( \psi _{0}^{G}\right) ^{-1}G\left( 
\overline{\varepsilon }_{A}\right) \left( r\right) =\left( \psi
_{0}^{G}\right) ^{-1}\pi \left( r\right) =\pi \left( r\right) =\varepsilon
_{A}\left( r_{1}\right) \pi \left( r_{2}\right) =\varepsilon _{A}\left(
r_{0}\right) r_{1}=\varepsilon _{A}\left( r\right) 1_{H}.
\end{equation*}%
We will use the following notations for all $r,s\in R$,%
\begin{equation*}
r\cdot _{R}s:=m_{R}\left( r\otimes s\right) ,\qquad 1_{R}:=u_{R}\left(
1_{\Bbbk }\right) .
\end{equation*}%
Now, by \cite[Corollary 1.7]{AMS-Amonoidal}, we have that $\epsilon
_{A}:FG\left( A\right) \rightarrow A$ is an algebra and a coalgebra
isomorphism in $({_{H}^{H}\mathfrak{M}_{H}^{H},\square }_{H},H)$. Let us
write the algebra and coalgebra structure of $FG\left( A\right) =R\otimes H.$
By construction (\cite[Proposition 1.5]{AMS-Amonoidal}), we have%
\begin{eqnarray*}
\overline{m}_{F\left( R\right) } &:&=F\left( m_{R}\right) \circ \psi
_{2}(R,R):F\left( R\right) \square _{H}F\left( R\right) \rightarrow F\left(
R\right) , \\
\overline{u}_{F\left( R\right) } &:&=F\left( u_{R}\right) \circ \psi
_{0}:H\rightarrow F\left( R\right) , \\
\overline{\Delta }_{F\left( R\right) } &:&=\psi _{2}(R,R)^{-1}\circ F\left(
\Delta _{R}\right) :F\left( R\right) \rightarrow F\left( R\right) \square
_{H}F\left( R\right) , \\
\overline{\varepsilon }_{F\left( R\right) } &:&=\psi _{0}^{-1}\circ F\left(
\varepsilon _{R}\right) :F\left( R\right) \rightarrow H.
\end{eqnarray*}%
Explicitly we have 
\begin{equation*}
\overline{m}_{F\left( R\right) }\left( \left( r\otimes h\right) \square
_{H}\left( s\otimes k\right) \right) =\omega (r_{-1}\otimes s_{-1}\otimes
k_{1})r_{0}\varepsilon (h)\cdot _{R}s_{0}\otimes k_{2},
\end{equation*}%
\begin{equation*}
\overline{u}_{F\left( R\right) }\left( h\right) =F\left( u_{R}\right) \psi
_{0}\left( h\right) =1_{R}\otimes h,
\end{equation*}%
\begin{equation*}
\overline{\Delta }_{F\left( R\right) }\left( r\otimes h\right) =\omega
^{-1}(r_{-1}^{1}\otimes r_{-2}^{2}\otimes h_{1})(r_{0}^{1}\otimes
r_{-1}^{2}h_{2})\square _{H}(r_{0}^{2}\otimes h_{3}),
\end{equation*}%
\begin{equation*}
\overline{\varepsilon }_{F\left( R\right) }\left( r\otimes h\right) =\psi
_{0}^{-1}F\left( \varepsilon _{R}\right) \left( r\otimes h\right) =\psi
_{0}^{-1}\left( \varepsilon _{R}\left( r\right) \otimes h\right)
=\varepsilon _{R}\left( r\right) h.
\end{equation*}%
In view of \ref{constr cot}, the forgetful functor $({{_{H}^{H}\mathfrak{M}%
_{H}^{H},\square }}_{H},H){\rightarrow }({{^{H}\mathfrak{M}^{H},\square }}%
_{H},H)$ is a strict monoidal functor. Being $\epsilon _{A}:(F\left(
R\right) ,\overline{\Delta }_{F\left( R\right) },\overline{\varepsilon }%
_{F\left( R\right) })\rightarrow (A,\overline{\Delta }_{A},\overline{%
\varepsilon }_{A}=\pi )$ a coalgebra morphism in $({{_{H}^{H}\mathfrak{M}%
_{H}^{H},\square _{H},H)}}$, we have that $\epsilon _{A}:(F\left( R\right) ,%
\overline{\Delta }_{F\left( R\right) },\overline{\varepsilon }_{F\left(
R\right) })\rightarrow (A,\overline{\Delta }_{A},\overline{\varepsilon }%
_{A}=\pi )$ is a coalgebra morphism in $({{^{H}\mathfrak{M}^{H},\square }}%
_{H},H).$ Apply Lemma \ref{lax mon} to the case $(\mathcal{M},\otimes ,%
\mathbf{1})=({{\mathfrak{M}}},\otimes ,\mathbf{\Bbbk })$ and $C=H$. Let $%
j_{X,Y}:X\square _{H}Y\rightarrow X\otimes Y$ be the canonical map. Then $%
\epsilon _{A}:(F\left( R\right) ,j_{F\left( R\right) ,F\left( R\right)
}\circ \overline{\Delta }_{F\left( R\right) },\varepsilon _{H}\circ 
\overline{\varepsilon }_{F\left( R\right) })\rightarrow (A,j_{A,A}\circ 
\overline{\Delta }_{A},\varepsilon _{H}\circ \overline{\varepsilon }_{A})$
is a coalgebra morphism in $({{\mathfrak{M}}},\otimes ,\mathbf{\Bbbk })$. In
other words it is an ordinary coalgebra morphism. Note that $(A,j_{A,A}\circ 
\overline{\Delta }_{A},\varepsilon _{H}\circ \overline{\varepsilon }%
_{A})=(A,\Delta _{A},\varepsilon _{A}).$ Set $(\Delta _{F\left( R\right)
},\varepsilon _{F\left( R\right) }):=(j_{F\left( R\right) ,F\left( R\right)
}\circ \overline{\Delta }_{F\left( R\right) },\varepsilon _{H}\circ 
\overline{\varepsilon }_{F\left( R\right) }).$ Let us compute explicitly
these maps. We have%
\begin{eqnarray*}
\Delta _{F\left( R\right) }\left( r\otimes h\right) &=&\left( j_{F\left(
R\right) ,F\left( R\right) }\circ \overline{\Delta }_{F\left( R\right)
}\right) \left( r\otimes h\right) =\omega ^{-1}(r_{-1}^{1}\otimes
r_{-2}^{2}\otimes h_{1})(r_{0}^{1}\otimes r_{-1}^{2}h_{2})\otimes
(r_{0}^{2}\otimes h_{3}), \\
\varepsilon _{F\left( R\right) }\left( r\otimes h\right) &=&\left(
\varepsilon _{H}\circ \overline{\varepsilon }_{F\left( R\right) }\right)
\left( r\otimes h\right) =\varepsilon _{R}\left( r\right) \varepsilon
_{H}\left( h\right) .
\end{eqnarray*}%
Thus $\epsilon _{A}:(F\left( R\right) ,\Delta _{F\left( R\right)
},\varepsilon _{F\left( R\right) })\rightarrow (A,\Delta _{A},\varepsilon
_{A})$ is an ordinary coalgebra morphism.

Being $\epsilon _{A}:(F\left( R\right) ,\overline{m}_{F\left( R\right) },%
\overline{u}_{F\left( R\right) })\rightarrow (A,\overline{m}_{A},\overline{u}%
_{A}=\sigma )$ an algebra morphism in $({{_{H}^{H}\mathfrak{M}%
_{H}^{H},\square }}_{H},H{)}$, then, in view of Lemma \ref{equi tensH, cot}, 
\begin{equation*}
\epsilon _{A}:(F\left( R\right) ,\Xi \left( \overline{m}_{F\left( R\right)
}\right) \circ \gamma _{2}(F\left( R\right) ,F\left( R\right) ),\Xi \left( 
\overline{u}_{F\left( R\right) }\right) \circ \gamma _{0})\rightarrow (A,\Xi
\left( \overline{m}_{A}\right) \circ \gamma _{2}(A,A),\Xi \left( \overline{u}%
_{A}\right) \circ \gamma _{0})
\end{equation*}%
is an algebra morphism in $({{_{H}^{H}\mathfrak{M}_{H}^{H},\otimes }}_{H},H)$%
. Note that%
\begin{eqnarray*}
\Xi \left( \overline{m}_{A}\right) \circ \gamma _{2}(A,A) &=&\overline{m}%
_{A}\circ \vartheta _{2}^{-1}(A,A)=m_{A}^{\prime }, \\
\Xi \left( \overline{u}_{A}\right) \circ \gamma _{0} &=&\overline{u}%
_{A}=\sigma
\end{eqnarray*}%
so that 
\begin{equation*}
(A,\Xi \left( \overline{m}_{A}\right) \circ \gamma _{2}(A,A),\Xi \left( 
\overline{u}_{A}\right) \circ \gamma _{0})=\left( A,m_{A}^{\prime },\sigma
\right) .
\end{equation*}%
Set $\left( m_{F\left( R\right) }^{\prime },u_{F\left( R\right) }^{\prime
}\right) :=\left( \Xi \left( \overline{m}_{F\left( R\right) }\right) \circ
\gamma _{2}(F\left( R\right) ,F\left( R\right) ),\Xi \left( \overline{u}%
_{F\left( R\right) }\right) \circ \gamma _{0}\right) .$ We have%
\begin{eqnarray*}
&&m_{F\left( R\right) }^{\prime }\left( \left( r\otimes h\right) \otimes
_{H}\left( s\otimes k\right) \right) \\
&=&\left[ \Xi \left( \overline{m}_{F\left( R\right) }\right) \circ \gamma
_{2}(F\left( R\right) ,F\left( R\right) )\right] \left( \left( r\otimes
h\right) \otimes _{H}\left( s\otimes k\right) \right) \\
&=&\overline{m}_{F\left( R\right) }\left[ \left( r\otimes h\right)
_{0}\left( s\otimes k\right) _{-1}\otimes _{H}\left( r\otimes h\right)
_{1}\left( s\otimes k\right) _{0}\right] \\
&=&\overline{m}_{F\left( R\right) }\left[ \left( r\otimes h_{1}\right)
\left( s_{-1}k_{1}\right) \otimes _{H}h_{2}\left( s_{0}\otimes k_{2}\right) %
\right] \\
&=&\omega ^{-1}\left[ r_{-1}\otimes h_{1}\otimes s_{-2}k_{1}\right] 
\overline{m}_{F\left( R\right) }\left[ r_{0}\otimes \left[ h_{2}\left(
s_{-1}k_{2}\right) \right] \otimes _{H}h_{3}\left( s_{0}\otimes k_{3}\right) %
\right] \\
&=&\left[ 
\begin{array}{c}
\omega ^{-1}\left[ r_{-1}\otimes h_{1}\otimes s_{-3}k_{1}\right] \omega
(h_{3}\otimes s_{-1}\otimes k_{3})\omega ^{-1}((h_{4}\vartriangleright
s_{0})_{-1}\otimes h_{5}\otimes k_{4}) \\ 
\overline{m}_{F\left( R\right) }\left[ r_{0}\otimes \left[ h_{2}\left(
s_{-2}k_{2}\right) \right] \otimes _{H}[(h_{4}\vartriangleright
s_{0})_{0}\otimes h_{6}k_{5}\right]%
\end{array}%
\right] \\
&=&\left[ 
\begin{array}{c}
\omega ^{-1}\left[ r_{-2}\otimes h_{1}\otimes s_{-3}k_{1}\right] \omega
(h_{3}\otimes s_{-1}\otimes k_{3})\omega ^{-1}((h_{4}\vartriangleright
s_{0})_{-2}\otimes h_{5}\otimes k_{4}) \\ 
\omega (r_{-1}\otimes (h_{4}\vartriangleright s_{0})_{-1}\otimes
h_{6}k_{5})r_{0}\varepsilon _{H}\left[ h_{2}\left( s_{-2}k_{2}\right) \right]
\cdot _{R}(h_{4}\vartriangleright s_{0})_{0}\otimes h_{7}k_{6}%
\end{array}%
\right] \\
&=&\left[ 
\begin{array}{c}
\omega ^{-1}\left[ r_{-2}\otimes h_{1}\otimes s_{-2}k_{1}\right] \omega
(h_{2}\otimes s_{-1}\otimes k_{2}) \\ 
\omega ^{-1}((h_{3}\vartriangleright s_{0})_{-2}\otimes h_{4}\otimes
k_{3})\omega (r_{-1}\otimes (h_{3}\vartriangleright s_{0})_{-1}\otimes
h_{5}k_{4}) \\ 
r_{0}\cdot _{R}(h_{3}\vartriangleright s_{0})_{0}\otimes h_{6}k_{5}%
\end{array}%
\right]
\end{eqnarray*}%
so that%
\begin{eqnarray*}
&&m_{F\left( R\right) }^{\prime }\left( \left( r\otimes h\right) \otimes
_{H}\left( s\otimes k\right) \right) \\
&=&\left[ 
\begin{array}{c}
\omega ^{-1}\left[ r_{-2}\otimes h_{1}\otimes s_{-2}k_{1}\right] \omega
(h_{2}\otimes s_{-1}\otimes k_{2}) \\ 
\omega ^{-1}((h_{3}\vartriangleright s_{0})_{-2}\otimes h_{4}\otimes
k_{3})\omega (r_{-1}\otimes (h_{3}\vartriangleright s_{0})_{-1}\otimes
h_{5}k_{4}) \\ 
r_{0}\cdot _{R}(h_{3}\vartriangleright s_{0})_{0}\otimes h_{6}k_{5}%
\end{array}%
\right] .
\end{eqnarray*}%
Moreover%
\begin{equation*}
u_{F\left( R\right) }^{\prime }\left( h\right) =\left[ \Xi \left( \overline{u%
}_{F\left( R\right) }\right) \circ \gamma _{0}\right] \left( h\right) =%
\overline{u}_{F\left( R\right) }\left( h\right) =1_{R}\otimes h.
\end{equation*}%
Apply Lemma \ref{lax mon} to the case $(\mathcal{M},\otimes ,\mathbf{1})=({{%
^{H}\mathfrak{M}^{H}}},\otimes ,\mathbf{\Bbbk })$ and $A=H$. Then 
\begin{equation*}
\epsilon _{A}:(F\left( R\right) ,m_{F\left( R\right) }^{\prime }\circ \chi
_{F\left( R\right) ,F\left( R\right) },u_{F\left( R\right) }^{\prime }\circ
u_{H})\rightarrow \left( A,m_{A}^{\prime }\circ \chi _{A,A},\sigma \circ
u_{H}\right)
\end{equation*}%
is an algebra homomorphism in $({{^{H}\mathfrak{M}^{H}}},\otimes ,\mathbf{%
\Bbbk })$. Note that $\left( A,m_{A}^{\prime }\circ \chi _{A,A},\sigma \circ
u_{H}\right) =\left( A,m_{A},u_{A}\right) $. Moreover, if we set $%
(m_{F\left( R\right) },u_{F\left( R\right) }):=(m_{F\left( R\right)
}^{\prime }\circ \chi _{F\left( R\right) ,F\left( R\right) },u_{F\left(
R\right) }^{\prime }\circ u_{H})$, we get 
\begin{eqnarray*}
&&m_{F\left( R\right) }\left( \left( r\otimes h\right) \otimes \left(
s\otimes k\right) \right) \\
&=&\left[ 
\begin{array}{c}
\omega ^{-1}\left[ r_{-2}\otimes h_{1}\otimes s_{-2}k_{1}\right] \omega
(h_{2}\otimes s_{-1}\otimes k_{2}) \\ 
\omega ^{-1}((h_{3}\vartriangleright s_{0})_{-2}\otimes h_{4}\otimes
k_{3})\omega (r_{-1}\otimes (h_{3}\vartriangleright s_{0})_{-1}\otimes
h_{5}k_{4}) \\ 
r_{0}\cdot _{R}(h_{3}\vartriangleright s_{0})_{0}\otimes h_{6}k_{5}%
\end{array}%
\right] .
\end{eqnarray*}%
Moreover%
\begin{equation*}
u_{F\left( R\right) }\left( k\right) =1_{R}\otimes k.
\end{equation*}%
Thus $\epsilon _{A}:(F\left( R\right) ,m_{F\left( R\right) },u_{F\left(
R\right) })\rightarrow \left( A,m_{A},u_{A}\right) $ is an algebra
isomorphism in $({{^{H}\mathfrak{M}^{H}}},\otimes ,\mathbf{\Bbbk })$ and $%
\epsilon _{A}:(F\left( R\right) ,\Delta _{F\left( R\right) },\varepsilon
_{F\left( R\right) })\rightarrow (A,\Delta _{A},\varepsilon _{A})$ is an
ordinary coalgebra isomorphism. Thus 
\begin{eqnarray*}
m_{A}\circ \left( \epsilon _{A}\otimes \epsilon _{A}\right) &=&\epsilon
_{A}\circ m_{F\left( R\right) },\qquad \epsilon _{A}\circ u_{F\left(
R\right) }=u_{A}, \\
\left( \epsilon _{A}\otimes \epsilon _{A}\right) \circ \Delta _{F\left(
R\right) } &=&\Delta _{A}\circ \epsilon _{A},\qquad \varepsilon _{A}\circ
\epsilon _{A}=\varepsilon _{F\left( R\right) },
\end{eqnarray*}%
so that $m_{F\left( R\right) },u_{F\left( R\right) },\Delta _{F\left(
R\right) },\varepsilon _{F\left( R\right) }$ are exactly the morphisms
induced by $m_{A},u_{A},\Delta _{A},\varepsilon _{A}$ via the vector space
isomorphism $\epsilon _{A}:F\left( R\right) \rightarrow A.$ Let $\omega
_{F\left( R\right) }$ be the map induced by $\omega _{A}$ via the vector
space isomorphism $\epsilon _{A}$ i.e.%
\begin{equation*}
\omega _{F\left( R\right) }:=\omega _{A}\circ \left( \epsilon _{A}\otimes
\epsilon _{A}\otimes \epsilon _{A}\right) :F\left( R\right) \otimes F\left(
R\right) \otimes F\left( R\right) \rightarrow \Bbbk .
\end{equation*}%
Then $\epsilon _{A}:\left( F\left( R\right) ,\Delta _{F\left( R\right)
},\varepsilon _{F\left( R\right) },m_{F\left( R\right) },u_{F\left( R\right)
},\omega _{F\left( R\right) }\right) \rightarrow \left( A,m_{A},u_{A},\Delta
_{A},\varepsilon _{A},\omega _{A}\right) $ is clearly an isomorphism of dual
quasi-bialgebras. Since, for all $r\in R,$ we have $\pi \left( r\right)
=\varepsilon _{A}\left( r_{1}\right) \pi \left( r_{2}\right) =\varepsilon
_{A}\left( r\right) 1_{H}$, then, for $r,s,t\in R,h,k,l\in H$, we get%
\begin{eqnarray*}
&&\omega _{F\left( R\right) }\left[ \left( r\otimes h\right) \otimes \left(
s\otimes k\right) \otimes \left( t\otimes l\right) \right] =\omega
_{A}\left( r\sigma \left( h\right) \otimes s\sigma \left( k\right) \otimes
t\sigma \left( l\right) \right) \\
&=&\omega _{H}\left[ \pi \left( r\sigma \left( h\right) \right) \otimes \pi
\left( s\sigma \left( k\right) \right) \otimes \pi \left( t\sigma \left(
l\right) \right) \right] =\omega _{H}\left[ \pi \left( r\right) h\otimes \pi
\left( s\right) k\otimes \pi \left( t\right) l\right] \\
&=&\omega _{H}\left[ \varepsilon _{A}\left( r\right) h\otimes \varepsilon
_{A}\left( s\right) k\otimes \varepsilon _{A}\left( t\right) l\right]
=\varepsilon _{A}\left( r\right) \varepsilon _{A}\left( t\right) \varepsilon
_{A}\left( s\right) \omega _{H}\left( h\otimes k\otimes l\right)
\end{eqnarray*}%
so that%
\begin{equation*}
\omega _{F\left( R\right) }\left[ \left( r\otimes h\right) \otimes \left(
s\otimes k\right) \otimes \left( t\otimes l\right) \right] =\varepsilon
_{A}\left( r\right) \varepsilon _{A}\left( t\right) \varepsilon _{A}\left(
s\right) \omega _{H}\left( h\otimes k\otimes l\right) .
\end{equation*}%
Note that $\left( F\left( R\right) ,\Delta _{F\left( R\right) },\varepsilon
_{F\left( R\right) },m_{F\left( R\right) },u_{F\left( R\right) },\omega
_{F\left( R\right) }\right) =R\#H$ once proved that $\left(
R,m_{R},u_{R},\Delta _{R},\varepsilon _{R}\right) $ is a bialgebra in the
monoidal category $\left( {_{H}^{H}\mathcal{YD}},\otimes ,\Bbbk \right) .$
It remains to prove that $m_{R}$ and $u_{R}$ are coalgebra maps. Since ${%
_{H}^{H}\mathcal{YD}}$ is a pre-braided monoidal category and $(R,\Delta
_{R},\varepsilon _{R})$ is a coalgebra in this category, then we can define
two morphisms $\Delta _{R\otimes R}$ and $\varepsilon _{R\otimes R}$ in ${%
_{H}^{H}\mathcal{YD}}$ such that $(R\otimes R,\Delta _{R\otimes
R},\varepsilon _{R\otimes R})$ is a coalgebra in ${_{H}^{H}\mathcal{YD}}$
too. We have%
\begin{eqnarray*}
\Delta _{R\otimes R} &:&=a_{R,R,R\otimes R}^{-1}\circ (R\otimes
a_{R,R,R})\circ (R\otimes c_{R,R}\otimes R)\circ (R\otimes
a_{R,R,R}^{-1})\circ a_{R,R,R\otimes R}\circ (\Delta _{R}\otimes \Delta
_{R}), \\
\varepsilon _{R\otimes R} &:&=\varepsilon _{R}\otimes \varepsilon _{R}.
\end{eqnarray*}%
Explicitly $\Delta _{R\otimes R}$ satisfies \eqref{form:DeltaRotR}.
In order to prove that $m_{R}$ is a morphism of coalgebras in ${_{H}^{H}%
\mathcal{YD}}$, we have to check the following equality%
\begin{equation*}
(m_{R}\otimes m_{R})\Delta _{R\otimes R}=\Delta _{R}m_{R}.
\end{equation*}%
Since we already obtained that $B:=F\left( R\right) $ is a dual
quasi-bialgebra, we know that 
\begin{equation*}
\Delta _{B}[(r\otimes 1_{H})\cdot _{B}(s\otimes 1_{H})]=(r\otimes
1_{H})_{1}\cdot _{B}(s\otimes 1_{H})_{1}\otimes (r\otimes 1_{H})_{2}\cdot
_{B}(s\otimes 1_{H})_{2}.
\end{equation*}%
By applying $R\otimes \varepsilon _{H}\otimes R\otimes \varepsilon _{H}$ on
both sides we get:%
\begin{eqnarray*}
&&(r\cdot _{R}s)^{1}\otimes (r\cdot _{R}s)^{2} \\
&=&\left[ 
\begin{array}{c}
\omega ^{-1}((r^{1})_{-2}\otimes (r_{-1}^{2})_{1}\otimes \left( s^{1}\right)
_{-2}\left( s_{-1}^{2}\right) _{1}) \\ 
\omega (\left( r_{-1}^{2}\right) _{2}\otimes \left( s^{1}\right)
_{-1}\otimes \left( s_{-1}^{2}\right) _{2}) \\ 
\omega ^{-1}[(\left( r_{-1}^{2}\right) _{3}\vartriangleright \left(
s^{1}\right) _{0})_{-2}\otimes \left( r_{-1}^{2}\right) _{4}\otimes \left(
s_{-1}^{2}\right) _{3}] \\ 
\omega (\left( r^{1}\right) _{-1}\otimes (\left( r_{-1}^{2}\right)
_{3}\vartriangleright \left( s^{1}\right) _{0})_{-1}\otimes \left(
r_{-1}^{2}\right) _{5}\left( s_{-1}^{2}\right) _{4}) \\ 
\left( r^{1}\right) _{0}\cdot _{R}(\left( r_{-1}^{2}\right)
_{3}\vartriangleright \left( s^{1}\right) _{0})_{0}%
\end{array}%
\right] \otimes (r_{0}^{2}\cdot _{R}s_{0}^{2}) \\
&=&\left[ 
\begin{array}{c}
\omega ^{-1}(r_{-2}^{1}\otimes r_{-5}^{2}\otimes s_{-2}^{1}s_{-4}^{2})\omega
(r_{-4}^{2}\otimes s_{-1}^{1}\otimes s_{-3}^{2}) \\ 
\omega ^{-1}[(r_{-3}^{2}\vartriangleright s_{0}^{1})_{-2}\otimes
r_{-2}^{2}\otimes s_{-2}^{2}] \\ 
\omega (r_{-1}^{1}\otimes (r_{-3}^{2}\vartriangleright
s_{0}^{1})_{-1}\otimes r_{-1}^{2}s_{-1}^{2}) \\ 
r_{0}^{1}\cdot _{R}(r_{-3}^{2}\vartriangleright s_{0}^{1})_{0}\otimes
(r_{0}^{2}\cdot _{R}s_{0}^{2})%
\end{array}%
\right] \\
&=&(m_{R}\otimes m_{R})\Delta _{R\otimes R}(r\otimes s).
\end{eqnarray*}%
The compatibility of $m_{R}$ with $\varepsilon _{R}$ and the fact 
that $u_{R}$ is a coalgebra morphism can be easily proved.%
\end{proof}

\section{Applications\label{C5}}

Here we collect some applications of the results of the previous sections.

\subsection{The associated graded coalgebra}

\begin{example}
Let $(A,m_{A},u_{A},\Delta _{A},\varepsilon _{A},\omega _{A})$ be a dual
quasi-bialgebra with the dual Chevalley property i.e. such that the
coradical $H$ of $A$ is a dual quasi-subbialgebra of $A$. Since $A$ is an
ordinary coalgebra, we can consider the associated graded coalgebra 
\begin{equation*}
\mathrm{gr}A:=\bigoplus\limits_{n\in \mathbb{N}}\mathrm{gr}^{n}A\qquad \text{%
where }\mathrm{gr}^{n}A:=\frac{A_{n}}{A_{n-1}}.
\end{equation*}%
Here $A_{-1}:=\{0\}$ and, for all $n\geq 0,$ $A_{n}$ is the $n$th term of
the coradical filtration of $A$. The coalgebra structure of $\mathrm{gr}A$
is given as follows. The $n$th graded component of the counit is the map $%
\varepsilon _{\mathrm{gr}A}^{n}:A_{n}/A_{n-1}\rightarrow \Bbbk $ defined by
setting%
\begin{equation*}
\varepsilon _{\mathrm{gr}A}^{n}(x+A_{n-1})=\delta _{n,0}\varepsilon _{A}(x).
\end{equation*}%
The $n$th graded component of comultiplication is the map 
\begin{equation*}
\Delta _{\mathrm{gr}A}^{n}:\mathrm{gr}^{a+b}A\rightarrow
\bigoplus\limits_{a+b=n,a,b\geq 0}\mathrm{gr}^{a}A\otimes \mathrm{gr}^{b}A
\end{equation*}%
defined as the diagonal map of the family $(\Delta _{\mathrm{gr}%
A}^{a,b})_{a+b=n,a,b\geq 0}$ where 
\begin{equation*}
\Delta _{\mathrm{gr}A}^{a,b}:\mathrm{gr}^{a+b}A\rightarrow \mathrm{gr}%
^{a}A\otimes \mathrm{gr}^{b}A,\Delta _{\mathrm{gr}%
A}^{a,b}(x+A_{a+b-1})=(x_{1}+A_{a-1})\otimes (x_{2}+A_{b-1}).
\end{equation*}
\end{example}

\begin{proposition}
Let $A$ be a dual quasi-bialgebra with the dual Chevalley property. Then 
\begin{equation*}
\left( \mathrm{gr}A,m_{\mathrm{gr}A},u_{\mathrm{gr}A},\Delta _{\mathrm{gr}%
A},\varepsilon _{\mathrm{gr}A},\omega _{\mathrm{gr}A}\right)
\end{equation*}
is a dual quasi-bialgebra where the graded components of the structure maps
are given by the maps%
\begin{gather*}
m_{\mathrm{gr}A}^{a,b}:\mathrm{gr}^{a}A\otimes \mathrm{gr}^{b}A\rightarrow 
\mathrm{gr}^{a+b}A,\qquad u_{\mathrm{gr}A}^{n}:\Bbbk \rightarrow \mathrm{gr}%
^{n}A, \\
\Delta _{\mathrm{gr}A}^{a,b}:\mathrm{gr}^{a+b}A\rightarrow \mathrm{gr}%
^{a}A\otimes \mathrm{gr}^{b}A,\qquad \varepsilon _{\mathrm{gr}A}^{n}:\mathrm{%
gr}^{n}A\rightarrow \Bbbk , \\
\omega _{\mathrm{gr}A}^{a,b,c}:\mathrm{gr}^{a}A\otimes \mathrm{gr}%
^{b}A\otimes \mathrm{gr}^{c}A\rightarrow \Bbbk,
\end{gather*}%
defined by%
\begin{gather*}
m_{\mathrm{gr}A}^{a,b}\left[ (x+A_{a-1})\otimes (y+A_{b-1})\right]
:=xy+A_{a+b-1},\qquad u_{\mathrm{gr}A}^{n}\left( k\right) :=\delta
_{n,0}1_{A}+A_{-1}=\delta _{n,0}1_{A}, \\
\Delta _{\mathrm{gr}A}^{a,b}(x+A_{a+b-1}):=(x_{1}+A_{a-1})\otimes
(x_{2}+A_{b-1}),\qquad \varepsilon _{\mathrm{gr}A}^{n}(x+A_{n-1}):=\delta
_{n,0}\varepsilon _{A}(x), \\
\omega _{\mathrm{gr}A}^{a,b,c}[(x+A_{a-1})\otimes (y+A_{b-1})\otimes
(z+A_{c-1})]:=\delta _{a,0}\delta _{b,0}\delta _{c,0}\omega _{A}(x\otimes
y\otimes z).
\end{gather*}%
Here $\delta _{i,j}$ denotes the Kronecker delta.
\end{proposition}

\begin{proof}
The proof of the facts that $m_{\mathrm{gr}A}$ and $u_{\mathrm{gr}A}$ are
well-defined, are coalgebra maps and that $m_{\mathrm{gr}A}$ is unitary is
analogous to the classical case, and depend on the fact that the coradical
filtration is an algebra filtration. This can be proved mimicking \cite[Lemma
5.2.8]{Montgomery}. The cocycle condition and the quasi-associativity of $m_{%
\mathrm{gr}A}$ are straightforward.

\end{proof}

\begin{proposition}
\label{pro: grA}Let $A$ be a dual quasi-bialgebra with the dual Chevalley
property and coradical $H$. Then $\left( \mathrm{gr}A,H,\sigma ,\pi \right) $
is a dual quasi-bialgebra with projection onto $H,$ where%
\begin{equation*}
\sigma :H\longrightarrow \mathrm{gr}A:h\longmapsto h+A_{-1},
\end{equation*}%
\begin{equation*}
\pi :\mathrm{gr}A\longrightarrow H:a+A_{n-1}\longmapsto \delta _{n,0}a,\text{
for all }a\in A_{n}.
\end{equation*}
\end{proposition}

\begin{proof}
It is straightforward. %
\end{proof}

\begin{corollary}
\label{coro: grA}Let $A$ be a dual quasi-bialgebra with the dual Chevalley
property and coradical $H$. Assume that $H$ has a preantipode. Then there is
a bialgebra $R$ in ${_{H}^{H}\mathcal{YD}}$ such that $\mathrm{gr}A$ is
isomorphic to $R\#H$ a dual quasi-bialgebra.
\end{corollary}

\begin{proof}
It follows by Proposition \ref{pro: grA} and Theorem \ref{teo: projection}.
\end{proof}

\begin{definition}
Following \cite[Definition, page 659]{AS-Lifting}, the bialgebra $R$ in ${%
_{H}^{H}\mathcal{YD}}$ of Corollary \ref{coro: grA}, is called the \emph{%
diagram} of $A$.
\end{definition}

\subsection{On pointed dual quasi-bialgebras}

We conclude this section considering the pointed case.

\begin{lemma}
\label{lem:gruppo}Let $G$ be a monoid and consider the monoid algebra $%
H:=\Bbbk G.$ Suppose there is a map $\omega \in (H\otimes H\otimes H)^{\ast
} $ such that $(H,\omega )$ is a dual quasi-bialgebra. Then $(H,\omega )$
has a preantipode $S$ if and only if $G$ is a group. In this case%
\begin{equation*}
S(g)=[\omega (g\otimes g^{-1}\otimes g)]^{-1}g^{-1}.
\end{equation*}
\end{lemma}

\begin{proof}
Suppose that $S$ is a preantipode for $(H,\omega ).$ Since $H$ is a
cocommutative ordinary bialgebra, by Theorem \ref{teo:cocom}, we have that $%
\Bbbk G$ is an ordinary Hopf algebra, where the antipode is defined, for all 
$g\in G$, by 
\begin{equation*}
s\left( g\right) :=S\left( g\right) _{1}\omega \left[ g\otimes S\left(
g\right) _{2}\otimes g\right] .
\end{equation*}%
Moreover one has $S=\varepsilon S\ast s$. Now, since $\Bbbk G$ is a Hopf
algebra, one has that the set of grouplike elements in $\Bbbk G,$ namely $G$
itself, form a group, where $g^{-1}:=s(g),$ for all $g\in G.$

Now, since $s$ is an anti-coalgebra map, we have%
\begin{equation*}
S\left( g\right) _{1}\otimes S\left( g\right) _{2}=\varepsilon S(g)s\left(
g\right) _{1}\otimes s\left( g\right) _{2}=\varepsilon S(g)s\left( g\right)
\otimes s\left( g\right) =S(g)\otimes g^{-1}
\end{equation*}%
so that $s\left( g\right) =S\left( g\right) _{1}\omega \left[ g\otimes
S\left( g\right) _{2}\otimes g\right] =S(g)\omega \left( g\otimes
g^{-1}\otimes g\right) .$ Hence $S(g)=[\omega (g\otimes g^{-1}\otimes
g)]^{-1}g^{-1}.$

The other implication is trivial (see \cite[Example 3.14]{Ardi-Pava}).
\end{proof}

The motivation for the previous result is Corollary \ref{Scapp} below.

\begin{proposition}
\label{pro: G(A)}Let $(A,m,u,\Delta ,\varepsilon ,\omega )$ be a dual
quasi-bialgebra. Then the set of grouplike elements $\mathbb{G}\left(
A\right) $ of $A$ is a monoid and the monoid algebra $\Bbbk \mathbb{G}\left(
A\right) $ is a dual quasi-subbialgebra of $A$.
\end{proposition}

\begin{proof}
It is straightforward. 
%
%
%
\end{proof}

\begin{corollary}
\label{Azero}Let $(A,m,u,\Delta ,\varepsilon ,\omega )$ be a pointed dual
quasi-bialgebra. Then $A_{0}=\Bbbk \mathbb{G}\left( A\right) $ is a dual
quasi-subbialgebra of $A$.
\end{corollary}

\begin{proof}
By Remark \ref{rem: pointed}, $A_{0}=\Bbbk \mathbb{G}\left( A\right) .$ In
view of Proposition \ref{pro: G(A)}, we conclude.
\end{proof}

\begin{corollary}
\label{Scapp}Let $(A,m,u,\Delta ,\varepsilon ,\omega ,s,\alpha ,\beta )$ be
a pointed dual quasi-Hopf algebra. Then $\mathbb{G}\left( A\right) $ is a
group and $A_{0}=\Bbbk \mathbb{G}\left( A\right) $ is a dual quasi-Hopf
algebra with respect to the induced structures.
\end{corollary}

\begin{proof}
Set $G:=\mathbb{G}\left( A\right) $. By Corollary \ref{Azero}, $A_{0}=\Bbbk
G $ is a dual quasi-subbialgebra of $A$. It remains to check that the
antipode on $A$ induces an antipode on $A_{0}.$ We have 
\begin{eqnarray*}
\Delta s(g) &=&s(g_{2})\otimes s(g_{1})=s(g)\otimes s(g), \\
\varepsilon s(g) &=&\varepsilon (g)=1,
\end{eqnarray*}

i.e. $s(g)\in G,$ for any $g\in G.$ Let $s_{0},\alpha _{0},\beta _{0},\omega
_{0},m_{0},u_{0},\Delta _{0},\varepsilon _{0}$ be the induced maps from $%
s,\alpha ,\beta ,\omega ,m,u,\Delta ,\varepsilon ,$ respectively. It is then
clear from the definition that $A_{0}$, with respect to these structures, is a dual quasi-Hopf
algebra. Since any dual quasi-Hopf algebra has a preantipode, by Lemma \ref%
{lem:gruppo}, $G$ is a group.
\end{proof}

Pointed dual quasi-Hopf algebras have been investigated also in \cite[page 2]%
{Huang-QuiverAppr} under the name of pointed Majid algebras. In view of
Corollary \ref{Scapp}, which seems to be implicitly assumed in \cite[page 2]%
{Huang-QuiverAppr}, we can apply Corollary \ref{coro: grA} to obtain the
following result.

\begin{theorem}
\label{teo:grpointed}Let $A$ be a pointed dual quasi-Hopf algebra. Then $%
\mathrm{gr}A$ is isomorphic to $R\#\Bbbk \mathbb{G}\left( A\right) $ as dual
quasi-bialgebra where $R$ is the diagram of $A$.
\end{theorem}

\appendix

\section{The weak right center}

\begin{definition}
\cite[Section 1.5]{bulacu-YDquasi} Let $\left( \mathcal{M},\otimes ,\mathbf{1%
},a,l,r\right) $ be a monoidal category. The weak right center $\mathcal{W}%
_{r}\left( \mathcal{M}\right) $ of $\mathcal{M}$ is a category defined as
follows. An object in $\mathcal{W}_{r}\left( \mathcal{M}\right) $ is a pair $%
\left( V,c_{-,V}\right) ,$ where $V$ is an object of $\mathcal{M}$ and $%
c_{-,V}$ is a family of morphisms in $\mathcal{M}$, $c_{X,V}:X\otimes
V\rightarrow V\otimes X$, defined for any object $X$ in $\mathcal{M},$ which
is natural in the first entry, such that, for all $X,Y\in \mathcal{M}$ we
have%
\begin{equation}
a_{V,X,Y}^{-1}\circ c_{X\otimes Y,V}\circ a_{X,Y,V}^{-1}=\left(
c_{X,V}\otimes Y\right) \circ a_{X,V,Y}^{-1}\circ \left( X\otimes
c_{Y,V}\right)  \label{form:cweak1}
\end{equation}%
and such that $r_{V}\circ c_{\mathbf{1},V}=l_{V}.$ A morphism $f:\left(
V,c_{-,V}\right) \rightarrow \left( W,c_{-,W}\right) $ is a morphism $%
f:V\rightarrow W$ in $\mathcal{M}$ such that, for each $X\in \mathcal{M}$ we
have%
\begin{equation*}
\left( f\otimes X\right) \circ c_{X,V}=c_{X,W}\circ \left( X\otimes f\right)
.
\end{equation*}%
$\mathcal{W}_{r}\left( \mathcal{M}\right) $ becomes a monoidal category with
unit $\left( \mathbf{1},l^{-1}\circ r\right) $ and tensor product%
\begin{equation*}
\left( V,c_{-,V}\right) \otimes \left( W,c_{-,W}\right) =\left( V\otimes
W,c_{-,V\otimes W}\right)
\end{equation*}%
where, for all $X\in \mathcal{M},$ the morphism $c_{X,V\otimes W}:X\otimes
\left( V\otimes W\right) \rightarrow \left( V\otimes W\right) \otimes X$ is
defined by%
\begin{equation*}
c_{X,V\otimes W}:=a_{V,W,X}^{-1}\circ \left( V\otimes c_{X,W}\right) \circ
a_{V,X,W}\circ \left( c_{X,V}\otimes W\right) \circ a_{X,V,W}^{-1}.
\end{equation*}%
The constraints are the same of $\mathcal{M}$ viewed as morphisms in $%
\mathcal{W}_{r}\left( \mathcal{M}\right) $. Moreover the monoidal category $%
\mathcal{W}_{r}\left( \mathcal{M}\right) $ is pre-braided, with braiding 
\begin{equation*}
c_{\left( V,c_{-,V}\right) ,\left( W,c_{-,W}\right) }:\left(
V,c_{-,V}\right) \otimes \left( W,c_{-,W}\right) \rightarrow \left(
W,c_{-,W}\right) \otimes \left( V,c_{-,V}\right)
\end{equation*}%
given by $c_{V,W}.$
\end{definition}

\begin{theorem}
\label{teo:WeakCenter}Let $H$ be a dual quasi-bialgebra. The categories $%
\mathcal{W}_{r}\left( {^{H}\mathfrak{M}}\right) $ and ${_{H}^{H}\mathcal{YD}}
$ are isomorphic, where ${^{H}\mathfrak{M}}$ is regarded as a monoidal
category as in Section \ref{C1}.
\end{theorem}

\begin{proof}
The proof is analogue to \cite[Theorem 3.5]{Balan}.

\end{proof}

\subsection{Example: the group algebra}

We now investigate the category of Yetter-Drinfeld modules over a particular
dual quasi-Hopf algebra.

Let $G$ be a group. Let $\theta :G\times G\times G\rightarrow \Bbbk ^{\ast
}:=\Bbbk \backslash \left\{ 0\right\} $ be a normalized $3$-cocycle on the
group $G$ in the sense of \cite[Example 2.3.2, page 54]{Maj2} i.e. a map
such that, for all $g,h,k,l\in H$ 
\begin{eqnarray*}
\theta \left( g,1_{G},h\right) &=&1 \\
\theta \left( h,k,l\right) \theta \left( g,hk,l\right) \theta \left(
g,h,k\right) &=&\theta \left( g,h,kl\right) \theta \left( gh,k,l\right) .
\end{eqnarray*}%
Then $\theta $ can be extended by linearity to a reassociator $\omega :\Bbbk
G\otimes \Bbbk G\otimes \Bbbk G\rightarrow \Bbbk $ making $\Bbbk G$ a dual
quasi-bialgebra with usual underlying algebra and coalgebra structures. This
dual quasi-bialgebra is denoted by $\Bbbk ^{\theta }G.$ Note that in
particular $\Bbbk ^{\theta }G$ is an ordinary bialgebra but with nontrivial
reassociator. In particular it is associative as an algebra. Let us
investigate the category ${_{\Bbbk ^{\theta }G}^{\Bbbk ^{\theta }G}\mathcal{%
YD}}$ of Yetter-Drinfeld module over $\Bbbk ^{\theta }G$.

\begin{definition}
Let $\theta :G\times G\times G\rightarrow \Bbbk ^{\ast }$ be a normalized $3$%
-cocycle on a group $G.$ The category of cocycle crossed left $G$-modules $%
\left( G,\theta \right) $\textrm{-}$\mathrm{Mod}$ is defined as follows. An
object in $\left( G,\theta \right) $\textrm{-}$\mathrm{Mod}$ is a pair $%
\left( V,\blacktriangleright \right) ,$ where $V=\oplus _{g\in G}V_{g}$ is a 
$G$-graded vector space endowed with a map $\blacktriangleright :G\times
V\rightarrow V$ such that, for all $g,h,l\in H$ and $v\in V,$ we have 
\begin{equation}
h\blacktriangleright V_{g}\in V_{hgh^{-1}},  \label{form:CrossedComp}
\end{equation}%
\begin{equation}
h\blacktriangleright \left( l\blacktriangleright v\right) =\frac{\theta
\left( hlgl^{-1}h^{-1},h,l\right) \theta \left( h,l,g\right) }{\theta \left(
h,lgl^{-1},l\right) }\left( hl\right) \blacktriangleright v,
\label{form:CrossedAss}
\end{equation}%
\begin{equation}
1_{H}\blacktriangleright v=v.  \label{form:CrossedUnit}
\end{equation}%
A morphism $f:\left( V,\blacktriangleright \right) \rightarrow \left(
V^{\prime },\blacktriangleright ^{\prime }\right) $ in $\left( G,\theta
\right) $\textrm{-}$\mathrm{Mod}$ is a morphism $f:V\rightarrow V^{\prime }$
of $G$-graded vector spaces such that, for all $h\in H,v\in V,$ we have $%
f(h\blacktriangleright v)=h\blacktriangleright ^{\prime }f(v).$
\end{definition}

The following result is inspired by \cite[Proposition 3.2]{Maj-QuantDouble}.

\begin{proposition}
Let $\theta :G\times G\times G\rightarrow \Bbbk ^{\ast }$ be a normalized $3$%
-cocycle on a group $G.$ Then the category ${_{\Bbbk ^{\theta }G}^{\Bbbk
^{\theta }G}\mathcal{YD}}$ is isomorphic to $\left( G,\theta \right) $%
\textrm{-}$\mathrm{Mod}$.
\end{proposition}

\begin{proof}
Set $H:=\Bbbk ^{\theta }G$ and let $\left( V,\rho _{V},\vartriangleright
\right) \in {_{H}^{H}\mathcal{YD}}$. Then $(V,\rho _{V})$ is an object in ${%
^{\Bbbk G}\mathfrak{M}}$. Hence, see e.g. \cite[Example 1.6.7]{Montgomery},
we have that $V=\oplus _{g\in G}V_{g}$ where $V_{g}=\left\{ v\in V\mid \rho
_{V}\left( v\right) =g\otimes v\right\}$. Define the map $%
\blacktriangleright :G\times V\rightarrow V,$ by setting $%
g\blacktriangleright v:=g\vartriangleright v$. It is easy to prove that the
assignments%
\begin{equation*}
\left( V,\rho _{V},\vartriangleright \right) \mapsto \left( V=\oplus _{g\in
G}V_{g},\blacktriangleright \right) \qquad f\mapsto f
\end{equation*}%
define a functor $L:{{_{H}^{H}\mathcal{YD}}}\rightarrow \left( G,\theta
\right) $\textrm{-}$\mathrm{Mod}$. Conversely, let $\left( V=\oplus _{g\in
G}V_{g},\blacktriangleright \right) $ be an object in $\left( G,\theta
\right) $\textrm{-}$\mathrm{Mod}$. Then $\blacktriangleright $ can be
extended by linearity to a map $\vartriangleright :\Bbbk G\otimes
V\rightarrow V.$ Define $\rho _{V}:V\rightarrow \Bbbk G\otimes V,$ by
setting $\rho _{V}\left( v\right) =g\otimes v$ for all $v\in V_{g}.$
Therefore, the assignments%
\begin{equation*}
\left( V=\oplus _{g\in G}V_{g},\blacktriangleright \right) \mapsto \left(
V,\rho _{V},\vartriangleright \right) \qquad f\mapsto f
\end{equation*}
define a functor $R:\left( G,\theta \right) \text{-}\mathrm{Mod}\rightarrow {%
_{H}^{H}\mathcal{YD}}$. It is clear that $LR=\mathrm{Id}$ and $RL=\mathrm{Id}
$.

\end{proof}

\begin{claim}
As a consequence of the previous result, the pre-braided monoidal structure
on ${_{\Bbbk ^{\theta }G}^{\Bbbk ^{\theta }G}\mathcal{YD}}$ induces a
pre-braided monoidal structure on $\left( G,\theta \right) $\textrm{-}$%
\mathrm{Mod}$ as follows. The unit is $\Bbbk $ regarded as a $G$-graded
vector space whose homogeneous components are all zero excepted the one
corresponding to $1_{G}.$ Moreover $h\blacktriangleright k=\varepsilon
_{H}\left( h\right) k$ for all $h\in H,k\in \Bbbk $. The tensor product is
defined by 
\begin{equation*}
\left( V,\blacktriangleright \right) \otimes \left( W,\blacktriangleright
\right) =\left( V\otimes W,\blacktriangleright \right)
\end{equation*}%
where 
\begin{equation*}
\left( V\otimes W\right) _{g}=\oplus _{h\in H}(V_{h}\otimes W_{h^{-1}g})
\end{equation*}%
and, for all $v\in V_{g},w\in W_{l},$ we have%
\begin{equation*}
h\blacktriangleright \left( v\otimes w\right) =\frac{\theta \left(
hgh^{-1},hlh^{-1},h\right) \theta \left( h,g,l\right) }{\theta \left(
hgh^{-1},h,l\right) }\left( h\blacktriangleright v\right) \otimes \left(
h\blacktriangleright w\right) .
\end{equation*}%
The constraints are the same of ${^{H}\mathfrak{M}}$ viewed as morphisms in $%
{_{H}^{H}\mathcal{YD}}$. 

The braiding $c_{V,W}:V\otimes W\rightarrow
W\otimes V$ is given, for all $v\in V_{g},w\in W_{l},$ by 
\begin{equation*}
c_{V,W}\left( v\otimes w\right) =\left( g\blacktriangleright w\right)
\otimes v.
\end{equation*}
\end{claim}

\end{document}